\documentclass{article}
\usepackage{psfrag}
\usepackage[parfill]{parskip}
\usepackage{float, graphicx}
\usepackage[]{epsfig}
\usepackage{amsmath, amsthm, amssymb,}
\usepackage{epsfig}
\usepackage{verbatim}
\usepackage{multicol}
\usepackage{url}

\usepackage{latexsym}
\usepackage{mathrsfs}
\usepackage[colorlinks, bookmarks=true]{hyperref}
\usepackage{graphicx}
\usepackage{amsmath}
\usepackage{enumerate}
\usepackage[normalem]{ulem}
\usepackage{bm}
\usepackage{dsfont}
\usepackage{stmaryrd}
\usepackage{enumitem}
\usepackage[noadjust]{cite}

\usepackage{color}
\usepackage{xcolor}
\usepackage{hyperref}
\usepackage{mathtools}
\usepackage[margin=1.1in]{geometry}
\usepackage{cleveref}

\makeatletter

\def\@settitle{\begin{center}%
    \bfseries
 \normalfont\LARGE\@title
  \end{center}%
}
\def\@setauthors{\begin{center}%
 \normalsize\@author
  \end{center}%
}

\makeatother

\numberwithin{equation}{section}

\renewcommand{\cal}{\mathcal}

\newcommand\cB{{\mathcal B}}
\newcommand\cJ{{\mathcal J}}
\newcommand{\cC}{{\cal C}}
\newcommand{\cD}{{\cal D}}
\newcommand{\cE}{{\cal E}}
\newcommand{\cF}{{\cal F}}
\newcommand{\cG}{{\cal G}}

\newcommand{\cK}{{\cal K}}

\newcommand{\cM}{{\cal M}}

\newcommand{\cS}{{\mathcal S}}

\newcommand{\cV}{{\mathcal V}}

\newcommand{\cX}{{\mathcal X}}
\newcommand{\cY}{{\mathcal Y}}
\newcommand{\sfL}{{\mathsf L}}
\newcommand{\sfA}{{\mathsf A}}

\newcommand{\sfD}{{\mathsf D}}

\newcommand{\bfS}{{\bf S}}

\newcommand{\fa}{{\mathfrak a}}
\newcommand{\fb}{{\mathfrak b}}
\newcommand{\fc}{{\mathfrak c}}
\newcommand{\fd}{{\mathfrak d}}

\newcommand{\fr}{{\mathfrak r}}

\newcommand{\bmu}{{\bm{u}}}

\newcommand{\bms}{{\bm s}}

\newcommand{\bfi}{{\bf i}}
\newcommand{\fC}{{\mathfrak C}}

\newcommand{\be}{\begin{equation}}
\newcommand{\ee}{\end{equation}}

\newcommand{\rd}{{\rm d}}

\newcommand{\ri}{\mathrm{i}}

\newcommand{\bC}{{\mathbb C}}
\newcommand{\bE}{\mathbb{E}}

\newcommand{\bP}{\mathbb{P}}

\newcommand{\bR}{{\mathbb R}}
\newcommand{\bX}{{\mathbb X}}

\newcommand{\bT}{\mathbb T}
\newcommand{\bV}{\mathbb V}
\newcommand{\bW}{\mathbb W}
\newcommand{\bF}{\mathbb F}
\newcommand{\bZ}{\mathbb{Z}}

\newcommand{\al}{\alpha}

\newcommand{\la}{\lambda}

\DeclareMathOperator{\dist}{dist}

\DeclareMathOperator{\diag}{diag}

\DeclareMathOperator{\OO}{O}
\DeclareMathOperator{\oo}{o}
\DeclareMathOperator{\spec}{spec}

\DeclareMathOperator{\poly}{poly}

\renewcommand{\Im}{\mathop{\mathrm{Im}}}

\newcommand{\ft}{\mathfrak t}

\newcommand{\deq}{\mathrel{\mathop:}=} %define :=
 
\renewcommand{\leq}{\leqslant}
\renewcommand{\geq}{\geqslant}

\newcommand{\del}{\partial}

\newcommand{\wh}{\widehat}
\newcommand{\wt}{\widetilde}

\newcommand{\qq}[1]{[\![{#1}]\!]}

\newcommand{\beq}{\begin{equation}}
\newcommand{\eeq}{\end{equation}}
\theoremstyle{plain} %plain, definition, remark
\newtheorem{theorem}{Theorem}[section]
\newtheorem*{theorem*}{Theorem}
\newtheorem{lemma}[theorem]{Lemma}
\newtheorem*{lemma*}{Lemma}
\newtheorem{corollary}[theorem]{Corollary}
\newtheorem*{corollary*}{Corollary}
\newtheorem{proposition}[theorem]{Proposition}
\newtheorem*{proposition*}{Proposition}

\newtheorem*{assumption*}{Assumption}

\newtheorem{definition}[theorem]{Definition}
\newtheorem*{definition*}{Definition}

\newtheorem*{example*}{Example}

\newtheorem*{remark*}{Remark}
\newtheorem*{remarks*}{Remarks}

\newcommand{\col}{\vcentcolon}
\newcommand{\msc}{m_{\rm sc}}
\newcommand{\md}{m_d}

\newcommand{\sfS}{{\sf S}}

\newcommand{\tP}{{\widetilde P}}
\newcommand{\sfF}{{\sf F}}
\def\author#1{\par
    {\centering{\authorfont#1}\par\vspace*{0.05in}}
}

\def\titlefont{\fontsize{13}{15}\bfseries\boldmath\selectfont\centering{}}
\def\authorfont{\fontsize{13}{15}}

\let\affiliationfont\rhfont

\def\address#1{\par
    {\centering{\affiliationfont#1\par}}\par\vspace*{11pt}
}

\def\body{
\setcounter{footnote}{0}
\def\thefootnote{\alph{footnote}}
\def\@makefnmark{{$^{\rm \@thefnmark}$}}
}

\def\title#1{
    \thispagestyle{plain}
    \vspace*{-14pt}
    \vskip 79pt
    {\centering{\titlefont #1\par}}%
    \vskip 1em
}

\usepackage{scalerel}
\newcommand{\cT}{{\mathcal T}}

\definecolor{forestgreen}{RGB}{34, 139, 34}

\newcommand{\GG}{{\mathcal G}}
\newcommand{\fR}{{\mathfrak R}}
\newcommand{\tG}{{\widetilde G}}
\newcommand{\tcG}{{\widetilde \cG}}
\newcommand{\oOmega}{\overline{\Omega}}

\begin{document}

\title{Optimal Eigenvalue Rigidity of Random Regular Graphs}

\vspace{1.2cm}

\noindent \begin{minipage}[c]{0.32\textwidth}
 \author{Jiaoyang Huang}
\address{University of Pennsylvania\\
   huangjy@wharton.upenn.edu}
 \end{minipage}
  \begin{minipage}[c]{0.32\textwidth}
 \author{Theo McKenzie}
\address{Stanford University\\
   theom@stanford.edu }
 \end{minipage}
\begin{minipage}[c]{0.32\textwidth}
 \author{Horng-Tzer Yau}
\address{Harvard University \\
   htyau@math.harvard.edu}

 \end{minipage}

\begin{abstract}
 Consider the normalized adjacency matrices of random $d$-regular graphs on $N$ vertices with fixed degree $d\geq 3$, and denote the eigenvalues as $\lambda_1=d/\sqrt{d-1}\geq \la_2\geq\la_3\cdots\geq \la_N$. We prove that the optimal  (up to an extra $N^{\oo_N(1)}$ factor, where $\oo_N(1)$ can be arbitrarily small) eigenvalue rigidity holds. More precisely, denote $\gamma_i$ as the classical location of the $i$-th  eigenvalue under the Kesten-Mckay law in decreasing order. Then with probability $1-N^{-1+\oo_N(1)}$,
    \begin{align*}
        |\lambda_i-\gamma_i|\leq \frac{N^{\oo_N(1)}}{N^{2/3} (\min\{i,N-i+1\})^{1/3}},\quad \text{ for all } i\in \{2,3,\cdots,N\}.
    \end{align*}
   In particular, the fluctuations of extreme eigenvalues are bounded by  $N^{-2/3+\oo_N(1)}$. This gives the same order of fluctuation as for the eigenvalues of matrices from the Gaussian Orthogonal Ensemble. 
\end{abstract}

\bigskip
{
\hypersetup{linkcolor=black}
\setcounter{tocdepth}{1}
\tableofcontents
}

\newpage
\section{Introduction}\label{s:intro}

A random $d$-regular graph on $N$ vertices is chosen uniformly from the set of all $N$-vertex $d$-regular graphs. These graphs serve as ubiquitous models in various fields such as computer science, number theory, and statistical physics. Particularly, they find applications in constructing optimally expanding networks \cite{hoory2006expander}, analyzing graph $\zeta$-functions \cite{terras2010zeta}, and studying quantum chaos \cite{smilansky2013discrete}.

A fundamental problem in the study of random regular graphs is the analysis of the spectrum of the adjacency matrix, or equivalently, the normalized adjacency matrix denoted as $H$
\begin{equation} \label{def_H}
H\deq \frac{A}{\sqrt{d-1}}.
\end{equation}
Normalizing in this manner facilitates comparisons across different values of $d$ and with other matrix ensembles. The eigenvalues of $H$ are denoted as $\lambda_1 \geq \lambda_2 \geq \cdots \geq \lambda_N$. According to the Perron-Frobenius theorem, since all vertices have degree $d$, $\lambda_1=d/\sqrt{d-1}$ is always the largest eigenvalue in modulus. The second largest eigenvalue is particularly significant as it determines the spectral gap. 

The Alon-Boppana bound \cite{nilli1991second} asserts that the second largest eigenvalue $\lambda_2$ of \emph{any} family of $d$-regular graphs with growing diameter satisfies $\lambda_2\geq 2-\oo_N(1)$, where $2$ is the spectral radius of the adjacency operator on the infinite tree; see \cite{kesten1959symmetric}. Alon in \cite{alon1986eigenvalues} conjectured that this bound is tight for random regular graphs, proposing that $\lambda_2\leq 2+o_N(1)$ with high probability. This conjecture was proven in  \cite{friedman2008proof} by Friedman, and later an alternative proof was given by Bordenave \cite{Bordenave2015ANP}. Their proofs are based on the moment methods,  involving classifying walks of different lengths and cycle types, and the error term
in eigenvalue  is of order $\OO((\log \log N/ \log N)^2)$.
In \cite{huang2024spectrum}, the first and third authors of the present paper, building upon joint work \cite{bauerschmidt2019local} with Bauerschmidt, presented a new proof of Alon’s conjecture with a polynomially small error of $\OO(N^{-c})$ (for some small $c>0$). This proof is based on a careful analysis of the Green's function. Moreover, this approach also establishes the concentration of all eigenvalues, not just $\lambda_2$.

The precise location of the second eigenvalue is a crucial problem with applications in computer science and number theory \cite{hoory2006expander}. Graphs for which $\max\{\lambda_2,|\lambda_N|\}\leq 2$ are known as \emph{Ramanujan graphs}, introduced by Lubotzky, Phillips, and Sarnak in \cite{sarnak2004expander}. They are the best possible spectral expander graphs. A major open question in combinatorics and spectral graph theory is whether there exist infinite families of Ramanujan graphs for every $d\geq 3$.  Algebraic constructions exist when $d-1$ is a prime power \cite{lubotzky1988ramanujan,margulis1988explicit,morgenstern1994existence}. Moreover, there are infinite families of \emph{bipartite} Ramanujan graphs, for which $\lambda_1=-\lambda_N=d/\sqrt{d-1}$, and $\lambda_2,|\lambda_{N-1}|\leq 2$ for any $n$ and $d$, constructed through interlacing families of the expected characteristic polynomial \cite{marcus2013interlacing,marcus2018interlacing}.
     It was conjectured based on numerical simulations that the distribution of the second largest eigenvalue of random $d$-regular graphs after normalizing by $N^{2/3}$ is the
    same as that of the largest eigenvalue of the Gaussian Orthogonal Ensemble \cite{miller2008distribution,sarnak2004expander}. Namely,  there is a constant $c_{N, d}$ so that  $ N^{2/3} (\lambda_2- 2)- c_{N, d} $ has the Tracy-Widom distribution \cite{tracy1996orthogonal}, and similar statement holds for the smallest eigenvalue. 
     This would imply that the fluctuations of extreme eigenvalues are of order $\OO(N^{-2/3})$ if $c_{N, d} $ is of order one. 
 If $c_{N, d}=0 $ (which seems to be the most probable scenario), then it would imply that  slightly more than half of  all $d$-regular graphs are Ramanujan graphs.  

In this work, we take a major step towards this universality conjecture. We derive an optimal bound (up to an extra $N^{\oo(1)}$ factor)  on the fluctuations of extreme eigenvalues, i.e. the fluctuations are bounded by  $N^{-2/3+\oo(1)}$.

The empirical eigenvalue density of random $d$-regular graphs converges to that of the infinite $d$-regular tree, which is known as the Kesten-McKay distribution; see \cite{kesten1959symmetric,mckay1981expected}. This density is given by 
\[
\varrho_d(x):=\mathbf1_{x\in [-2,2]} \left(1+\frac1{d-1}-\frac {x^2}d\right)^{-1}\frac{\sqrt{4-x^2}}{2\pi}.
\]
For $2\leq i \leq N$, we expect the $i$-th largest eigenvalue of the normalized adjacency matrix $H$ to be closed to the classical eigenvalue locations $\gamma_i$, where $\gamma_i$ satisfies
\be\label{eq:gammadef}
\int_{\gamma_i}^2 \varrho_d(x)\rd x=\frac{i-1/2}{N-1},\quad 2\leq i\leq N.
\ee
Our main results given optimal concentration for each eigenvalues of the random $d$-regular graphs.

\begin{theorem}\label{thm:eigrigidity}
Fix $d\geq 3$. There is a positive integer $\omega_d\geq 1$ depending only on $d$, such that with probability $1-N^{(1-\oo(1))\omega_d}$, the eigenvalues $\{\lambda_i\}_{i\in \qq{N}}$ of the normalized adjacency matrix of a $d$-regular random regular graph satisfy
\begin{align}\label{e:optimal_rigidity}
|\lambda_i-\gamma_i|\leq N^{-2/3+\oo(1)}(\min\{i,N-i+1\})^{-1/3}
\end{align}
for every $2\leq i\leq N$ and $\gamma_i$ are the classical eigenvalue locations, as defined in \eqref{eq:gammadef}.
\end{theorem}

The aforementioned theorem implies that $\lambda_2,|\lambda_N|=2+N^{-2/3+\oo(1)}$. In the bulk, for fixed $\delta>0$ and $\delta N\leq i\leq (1-\delta)N$, the fluctuation of $\lambda_i$ is of order $N^{-1+\oo(1)}$. This level of fluctuation aligns with that observed for the eigenvalues of matrices from the Gaussian Orthogonal Ensemble, and more broadly, Wigner matrices; see \cite{erdHos2012spectral}. While our discussion thus far has centered on $\lambda_2$, it's worth noting that the spectral statistics of every eigenvalue are of interest from a statistical physics perspective, and have been subjects of study since the early works of Wigner and Dyson \cite{wigner1955characteristic,dyson1962brownian}.

We prove \Cref{thm:eigrigidity}  through a careful analysis of the Green's function and their variations, which has proven a highly successful way to analyze spectral information of random matrices (see \cite{erdHos2017dynamical} for an overview of this approach). In the rest of this section, we introduce the notations and state our main result on the optimal estimates of the Green's function. 

We define the \emph{Green's function} of the normalized adjacency matrix $H$ by
\begin{equation*}
  G(z) \deq  (H-z)^{-1}=\sum_{i=1}^N \frac{\bmu_i \bmu_i^\top}{\la_i-z},\quad z\in \{z\in \bC \col \Im[z]>0\}.
\end{equation*}
We denote the normalized trace of $G$, which is also the Stieltjes transform of the empirical eigenvalue density of $H$, by
\begin{equation} \label{e:m}
  m(z) \deq \frac{1}{N} \sum_i G_{ii}(z)=\frac1N\sum_{i=1}^N\frac{1}{\la_i-z}.
\end{equation}
 We refer to $m(z)$ simply as the \emph{Stieltjes transform}. 
The \emph{Ward identity} states that the Green's function $G$ satisfies
\begin{equation} \label{e:WdI}
  \frac{1}{N} \sum_{j} |G_{ij}(z)|^2 = \frac{\Im[G_{ii}(z)]}{N\Im[z]},
  \qquad
  \frac{1}{N} \sum_{ij} |G_{ij}(z)|^2 = \frac{\Im[m(z)]}{\Im[z]}.
\end{equation}
Our goal is to approximate $m(z)$ by $\md(z)$,
the Stieltjes transform of the Kesten--McKay law
\begin{align*}
    \md(z)=\int_\bR \frac{\varrho_d(x)\rd x}{x-z},\quad z\in \{z\in \bC \col \Im[z]>0\}.
\end{align*}
We recall the semi-circle distribution $\varrho_{\rm sc}(x)$, its Stieltjes transform $\msc(z)$, and the quadratic equation satisfied by $\msc(z)$,
\begin{align*}
 \varrho_{\rm sc}(x)=\bm1_{x\in[-2,2]}\frac{\sqrt{4-x^2}}{2\pi},
 \quad 
  \msc(z)=\int_\bR \frac{\varrho_{\rm sc}(x)\rd x}{x-z}=\frac{-z+\sqrt{z^2-4}}{2},\quad 
  \msc(z)^2+z\msc(z)+1=0.
\end{align*}
Explicitly the Stieltjes transform of the Kesten--McKay law $\md(z)$ can be expressed in terms of the Stieltjes transform $\msc(z)$,
\begin{align}
    \md(z)=\frac{1}{-z-\frac{d}{d-1}\msc(z)}.
\end{align}

The following Theorem states that $m(z)$ can be approximated by $\md(z)$
for $z$ belonging to the following spectral domain
\begin{equation} \label{e:D}
  \mathbf D \deq \{z=E+\ri \eta\col |E| \leq 2+\fa, 0< \eta \leq 1/\fa, N\eta\sqrt{\min\{|E-2|, |E+2|\}+\eta}\geq N^{\fa}\},
\end{equation}
where $0<\fa<1$ can be arbitrarily small. Theorem \ref{thm:eigrigidity} is a consequence of the following theorem.

\begin{theorem}\label{t:rigidity}

For every $d\geq 3$, there is a positive integer $\omega_d\geq 1$ such that for sufficiently small $0<\fa<1$, and $N$ large enough, with probability $1-\OO(N^{-(1-\oo(1))\omega_d})$, the following is true for every $z=E+\ri\eta\in \mathbf D$ (recall from \eqref{e:D}),
\begin{align}\label{e:mbond}
|m(z)-m_d(z)|\leq N^{\oo(1)} 
\left\{\begin{array}{cc}
 \frac{1}{N\eta}, & -2\leq E\leq 2,\\
 \frac{1}{\sqrt{\kappa+\eta}}\left(\frac{1}{N\eta^{1/2}}+\frac{1}{(N\eta)^2}\right), & |E|\geq 2,
\end{array}
\right.
\end{align}
where $\kappa=\min\{|E-2|, |E+2|\}$.
\end{theorem}
Theorem \ref{t:rigidity} provides optimal concentration of the Stieltjes transform of the emprical eigenvalue distribution of $H$. The imaginary part of the Stieltjes transform has long been utilized as a means to access information about the empirical eigenvalue density, and precisely following such reasoning leads to Theorem \ref{thm:eigrigidity}. We present the proof of \Cref{thm:eigrigidity} using \Cref{t:rigidity} as an input to \Cref{s:proofrigidity}.

\subsection{Proof ideas}

To establish Theorem \ref{t:rigidity}, we start with the self-consistent equation of a modified Green's function quantity, as introduced in \cite{bauerschmidt2019local, huang2024spectrum}. The quantity is the average of $G_{jj}^{(i)}(z)$ over all pairs of adjacent vertices $i\sim j$:
\begin{align}\label{e:Qsum}
Q(z;\cG):=\frac{1}{Nd}\sum_{i\sim j}G_{jj}^{(i)}(z).
\end{align}
For $d$-regular graphs, although intricate, $Q(z;\cG)$ emerges as a more useful entity than the Stieltjes transform of the empirical eigenvalue distribution of the normalized adjacent matrix $H$. 
To compute $G_{jj}^{(i)}(z)$, we approximate it by the Green's function of a neighborhood of radius $\ell$ around vertex $j$, with vertex $i$ removed, incorporating suitable weights at each boundary vertex. Given that most vertices in random $d$-regular graphs possess large tree neighborhoods, the majority of vertices $j$ in the summation of \eqref{e:Qsum} have large tree neighborhoods. For such vertices $j$, we can subsequently replace the boundary weights in the approximation with the weight $Q(z;\cG)$. Furthermore, the neighborhoods of these vertices $j$, with vertex $i$ removed, are truncated $(d-1)$-ary trees of depth $\ell$.

Let $Y_\ell(Q(z;\cG))$ be the Green's function at the root vertex of a truncated $(d-1)$-ary tree of depth $\ell$, with boundary weights $Q(z;\cG)$. This leads us to the following self-consistent equation for $Q(z;\cG)$:
\begin{align}\label{e:selfeq}
Q(z;\cG)=\frac{1}{Nd}\sum_{i\sim j}G_{jj}^{(i)}(z)\approx Y_\ell(Q(z;\cG)).
\end{align}
for arbitrary $\ell$. The fixed point of the above self-consistent equation \eqref{e:selfeq} is given by the Stieltjes transform of the semi-circle distribution $\msc(z)$. Indeed, $Y_\ell(\msc(z))$ represents the Green's function at the root vertex of an infinite $(d-1)$-ary tree, whose spectral density is governed by the semi-circle distribution. Consequently, the self-consistent equation \eqref{e:selfeq} was utilized in \cite{bauerschmidt2019local, huang2024spectrum} to demonstrate that $Q(z;\cG)$ is closely approximated by $\msc(z)$ with high probability:
\begin{align*}
  \msc(z)=Y_\ell(\msc(z)),\quad  Q(z;\cG)\approx \msc(z).
\end{align*}
The Stieljes transform $m(z)$ of the empirical eigenvalue distribution of $H$ can also be recovered from the quantity $Q$, through the following approximation
\begin{align}\label{e:mz_approx}
    m(z)=\frac{1}{N}\sum_{i=1}^N G_{ii}(z)\approx X_\ell(Q(z;\cG)),
\end{align}
where $X_\ell(Q(z;\cG))$ denotes the Green's function at the root vertex of a truncated  $d$-regular tree of depth $\ell$, with boundary weights $Q(z;\cG)$.

Utilizing the self-consistent equations \eqref{e:selfeq} and \eqref{e:mz_approx}, \cite{bauerschmidt2019local, huang2024spectrum} established that with high probability, uniformly for any $z$ in the upper half-plane with $\Im[z]\geq N^{-1+\oo(1)}$, the Stieltjes transform of the empirical eigenvalue distribution closely approximates $\md(z)$:
\begin{align}\label{e:weak}
|m(z)-\md(z)|\leq N^{-\delta},
\end{align}
for some small $\delta>0$. Furthermore, it was shown that the eigenvectors are completely delocalized. Achieving optimal rigidity of eigenvalue locations, as stated in Theorem \ref{thm:eigrigidity}, necessitates an optimal error bound much stronger than $N^{-\delta}$ in \eqref{e:weak}.  This constitutes a standard problem in random matrix theory, which requires an optimal error bound for the high moments of the self-consistent equation:
\begin{align}\label{e:hmm}
\bE[|Q(z;\cG)-Y_\ell(Q(z;\cG))|^{2p}],
\end{align}
 for large integers $p>0$. Analogous estimates have been established for Wigner matrices \cite{erdHos2012rigidity,erdHos2013local, he2018isotropic}, Erd\H{o}s-R\'enyi graphs \cite{lee2018local,erdHos2013spectral,erdHos2012spectral, he2021fluctuations,huang2020transition,huang2022edge,lee2021higher}, $d$-regular graphs with growing degrees \cite{bauerschmidt2020edge, huang2023edge, he2022spectral}, and $\beta$-ensembles \cite{bourgade2022optimal}.

Distinguished from Wigner matrices, the primary challenge in studying the adjacency matrices of random $d$-regular graphs lies in the correlations between matrix entries, where row and column sums are fixed at $d$.  To address this constraint, the concept of local resampling has been pivotal, initially explored to unravel randomness under such correlations. This technique was first employed to derive spectral statistics for random $d$-regular graphs with $d=N^{\oo(1)}$ in \cite{bauerschmidt2017local}, and subsequently extended in \cite{bauerschmidt2019local, huang2024spectrum} to establish the self-consistent equation \eqref{e:selfeq} and bound its high moments \eqref{e:hmm}. In this method, local resampling randomizes the boundary of a neighborhood $\cT$ (as opposed to randomizing edges near a vertex as in \cite{bauerschmidt2017local}) by exchanging the edge boundary of $\cT$ with randomly chosen edges elsewhere in the graph. Notably, this local resampling is reversible, meaning the law governing the graphs and their switched counterparts is exchangeable.

However, the error bound for the high moments of the self-consistent equation \eqref{e:hmm} in \cite{huang2024spectrum} falls short of optimality, as do the results regarding eigenvalue rigidity. 
Our main contributions extend these findings to the optimal scale. Specifically, we demonstrate that fluctuations in extreme eigenvalues are bounded by $N^{-2/3+\oo(1)}$. This improves the weak $N^{-\oo(1)}$ bound in \cite{huang2024spectrum}, and gives the same order of fluctuation as for the eigenvalues of matrices from the Gaussian orthogonal ensemble. 
Below, we outline the challenges encountered, and the new strategy and key ideas developed to overcome them.

\subsubsection{New Strategy}\label{s:newstrategy}

To illustrate the ideas that will lead to an optimal error bound for \eqref{e:hmm},  let's begin with the first moment of the self-consistent equation:
\begin{align}\label{e:sample}
    \bE[Q(z;\cG)-Y_\ell(Q(z;\cG))]
    =\bE[G_{oo}^{(i)}-Y_\ell(Q(z;\cG))]
    =\bE[\widetilde G_{oo}^{(i)}-Y_\ell(Q(z;\cG))],
\end{align}
where in the first statement we exploit the permutation invariance of vertices, so the expectation of $Q(z;\cG)$ is the same as the expectation of $G_{oo}^{(i)}$ for any pair of adjacent vertices $i\sim o$. Given that most vertices have large tree neighborhoods, we can focus on cases where $o$ has a large tree neighborhood.  The second statement arises from the local resampling process. Instead of computing the expectation of $G_{oo}^{(i)}$, we perform a local resampling around vertex $o$, by switching the boundary edges of the radius $\ell$ neighborhood $\cT=\cB_\ell(o, \cG)$ of vertex $o$, with randomly selected edges $\{(b_\al, c_\al)\}_{\al\in \qq{\mu}}$ from $\cG$.
 We denote the resulting graph as $\tcG$, its Green's function as $\widetilde G(z)$, and the new boundary vertices of $\cT$ after local resampling as $\{c_\al\}_{\al\in\qq{\mu}}$, which are typically distanced apart in terms of graph distance (see Section \ref{s:local_resampling} for further details). As local resampling is reversible, $G_{oo}^{(i)}$ and $\widetilde G_{oo}^{(i)}$ share the same law and expectation, which gives the last expression in \eqref{e:sample}.

To demonstrate the smallness of the final expression in \eqref{e:sample}, we expand  $\widetilde G_{oo}^{(i)}$ using the Schur complement formula.
The radius $\ell$ neighborhood of $o$ in $\tcG^{(i)}$ (where vertex $i$ is removed) is  $\cT^{(i)}$, a truncated $(d-1)$-ary tree at level $\ell$. 
The Schur complement formula states that $\widetilde G_{oo}^{(i)}$ is the same as the Green's function of $\cT^{(i)}$ with boundary weights given by $\widetilde G_{c_\al c_\beta}^{(\bT)}$, which are the Green's functions of $\tcG^{(\bT)}$ (with the vertex set $\bT$ of $\cT$ removed). With high probability, the new boundary vertices $\{c_\al\}_{\al\in\qq{\mu}}$ are far from each other, and exhibiting large tree neighborhoods. Consequently, the neighborhoods of $c_\al$ in $\cG^{(\bT)}$ are given by the truncated $(d-1)$-ary trees, and the boundary weights can be approximated by the Green's function of $\cG$ (the graph before switching) as
\begin{align}\label{e:GtoQ}
   \widetilde G_{c_\al c_\beta}^{(\bT)}
   \approx  G_{c_\al c_\beta}^{(b_\al b_\beta)}\approx \bm1({\al=\beta}) Q(z;\cG),\quad 1\leq \al, \beta\leq \mu.
\end{align}
Consequently, the leading-order term of 
 $\widetilde G_{oo}^{(i)}$ is given by $Y_\ell(Q(z;\cG))$, and \eqref{e:sample} is small. This strategy, utilized in  \cite{huang2024spectrum}, provided a weak bound for \eqref{e:hmm}, by bounding all errors from the approximations such as \eqref{e:GtoQ} by $N^{-\delta}$ for some small $\delta >0$. 

 To achieve optimal estimates for Green's functions, we need to analyze the approximation errors from the Schur complement formula \eqref{e:GtoQ} more carefully. These errors comprise weighted sums of terms involving factors such as:
 \begin{align}\label{e:error_factor}
     (G^{(b_\al)}_{c_\al c_\al}-Q(z;\cG)),  \quad G_{c_\al c_\beta}^{(b_\al b_\beta)},\quad Q(z;\cG)-\msc(z).
 \end{align}
For a more precise description of these error terms, we refer to \Cref{lem:diaglem}. Instead of simply bounding each term in \eqref{e:error_factor} by $N^{-\delta}$, we carefully examine all possibilities of error terms. 
The crucial observation is that the expectation of the first factor in \eqref{e:error_factor} with respect to the randomness of the simple switching is very small:
\begin{align*}
\mathbb{E}_{\mathbf{S}}[G^{(b_{\alpha})}_{c_{\alpha} c_{\alpha}}-Q(z;\mathcal{G})]=\OO\left(\frac{1}{N^{1-\oo(1)}}\right);
\end{align*}
and, up to negligible error, the expectation of the second factor in \eqref{e:error_factor} can be expressed to include the first factor,
\begin{align*}
   \bE_\bfS[G_{c_\al c_\beta}^{(b_\al b_\beta)}]=\frac{d}{\sqrt{d-1}}\bE_\bfS\left[ (G_{c_\al c_\al}^{(b_\al)}-Q(z;\cG))G_{b_\al c_\beta}^{(b_\beta)}\right]+\text{``negligible error"}.
\end{align*}
In this way, we demonstrate that either the error term is negligible or contains one ``diagonal factor" akin to the first factor in \eqref{e:error_factor}. Such a factor, $(G^{(b_\al)}_{c_\al c_\al}-Q(z;\cG))$, is in the same form as the expression \eqref{e:sample} we initiated with. To evaluate the expectation, we will perform another local resampling around the vertex $c_\alpha$, chosen randomly from the last local switching step. 

Our new strategy to derive the optimal error bound for \eqref{e:hmm} is an iteration scheme. At each step, we perform local resampling and rewrite the Green's function of the switched graph in terms of the original graph. Next, we show that each term contains at least one ``diagonal factor" $G_{c_\alpha c_\alpha}^{(b_\alpha)}$ where $(b_\alpha, c_\alpha)$ is an edge selected during local resampling, or it is negligible. Then, we can perform a local resampling around $c_\alpha$, and repeat this procedure. We formalize this iterative scheme using a sequence of forests, as introduced in \Cref{s:forest}. Crucially, we demonstrate that each iteration of local resampling yields an additional factor of at least $N^{-\delta}$ for some small $\delta>0$. After a finite number of steps, all error terms become negligible, leading to an optimal bound for the high moments of the self-consistent equation \eqref{e:hmm}. This iteration is detailed in \Cref{p:add_indicator_function} and \Cref{p:iteration}.
 
\subsubsection{New Technical Ideas}

For each local resampling, we need to rewrite the Green's function of the switched graph in terms of the original graph. While the Schur complement formula suffices for expressions like \eqref{e:sample}, for higher moments, we must also track changes such as:
\begin{align*}
    (\wt Q-Y_\ell (\wt Q)-(Q-Y_\ell(Q))=(1-Y_\ell'(Q))(\wt Q-Q)+\OO(|\wt Q-Q|^2),
\end{align*}
where $\widetilde Q=Q(z;\tcG)$ and $Q=Q(z;\cG)$. Since $\widetilde H-H$ (the difference of the normalized adjacency matrices of $\cG$ and $\tcG$) is low rank,  one can use the Ward identity \eqref{eq:wardex} to show that with high probability
\begin{align*}
    |\widetilde Q-Q|\leq \frac{N^{\oo(1)}\Im[m(z)]}{N\eta}, \quad \eta=\Im[z],
\end{align*}
as done  in prior work \cite{huang2024spectrum}. However it is insufficient for the optimal rigidity.

To achieve improved accuracy for $\wt Q-Q$, it is essential  to compute the difference of the Green's functions $\wt G(z)$ and $G(z)$ for the graphs $\tcG$ and $\cG$. The local resampling around vertex $o$ can be represented by a matrix $\xi:=\wt H-H$. Thus using the resolvent identity, we can write
\begin{align}\label{e:resolvent}
    \wt G(z)-G(z)=-G(z)\xi G(z) -\sum_{k\geq 1}G(z)\xi (-G(z)\xi)^k G(z).
\end{align}
The aforementioned expansion was utilized in \cite{huang2023edge} to prove the edge universality of random $d$-regular graphs, when the degree $d=N^{\oo(1)}$ grows with the size the graph. In this case, owing to our normalization, each entry of $\xi$ scales as $\OO(1/\sqrt{d})$. Thus, the terms in \eqref{e:resolvent} exhibit exponential decay in $1/\sqrt d$. However, in our scenario where $d$ is fixed, this decay is too slow. 

For $d\geq 3$ with $d$ fixed, instead of relying on the  resolvent expansion \eqref{e:resolvent}, we introduce a novel new expansion based on the Woodbury formula, taking advantage of the local tree structure as detailed in \Cref{lem:woodbury}. Specifically, since the rank of the matrix $\xi=\wt H-H$ is at most $\OO((d-1)^\ell)$ (the number of edges involved in local resampling), let's denote the eigenvalue decomposition of $\xi$ as $\xi=UCU^\top$. The Woodbury formula yields the the difference of the Green's functions $\wt G(z)-G(z)$ as 
\begin{align}\label{e:WB}
\wt G(z)-G(z)&=-G(z)U(C^{-1}+U^\top G(z)U)^{-1}U^\top G(z).
\end{align}
Here, the nonzero rows of $U$ correspond to the vertices involved in local resampling. Therefore, the  term $U^\top G(z)U$ in \eqref{e:WB} depends solely on the Green's function entries restricted to the subgraph $\cF:=\cB_{\ell+1}(o,\cG)\cup \{(b_\al, c_\al)\}_{\al\in\qq{\mu}}$. With high probability, these randomly chosen edges $(b_\al, c_\al)$ have tree neighborhoods, and are far apart from each other and the vertex $o$. Let $P(z)$ denote the Green's function of the tree extension of $\cF$ (extending each connected component to an infinite $d$-regular tree). Then $U^\top G(z)U-U^\top P(z)U$ is small, leading to the following expansion
\begin{align}\label{e:tG-G}
\wt G(z)-G(z) = G(z)F(z)G(z)+\sum_{k\geq 1}G(z)F(z)((G(z)-P(z))F(z))^kG(z).
\end{align}
Here, $F(z)=-U(C^{-1}+U^\top P(z)U)^{-1}U^\top$ is an explicit matrix. When restricted to the subgraph, each entry of $G(z)-P(z)$ is smaller than $N^{-\delta}$ for some small $\delta>0$. Hence, the terms in \eqref{e:tG-G} exhibit exponential decay in $N^{-\delta}$ which is much faster than \eqref{e:resolvent}. And we can truncate the expansion \eqref{e:tG-G} at some finite $k$, and the remainder is negligible.  

Another technical idea involves a Ward identity-type bound for the entries of the Green's function, which serves to constrain various error terms. The Ward identity plays a crucial role in mean-field random matrix theory. It states that the average over the  Green's function entries can be controlled by the Stieltjes transform of the empirical eigenvalue distribution, thus ensuring smallness:
\begin{align}\label{eq:wardex}
\frac{1}{N^2}\sum_{ij}|G_{ij}(z)|^2= \frac{\Im[m(z)]}{N\eta},\quad \eta=\Im[z].
\end{align}
Consequently, when selecting two vertices randomly from our graph, the Green's function entries are expected to have small modulus. While some error terms adhere to this form, we also encounter Green's function entries taking such as 
$G_{ij}^{(o)}$,
where $i,j$ are two adjacent vertices of a vertex $o$, i.e. $o\sim i,  o \sim j$. Namely, we take two vertices of distance two, delete their common neighbor, then take the Green's function. In the initial graph $\cG$, $i$ and $j$ have distance two, so $G_{ij}$ is not small.
In \Cref{lem:deletedalmostrandom}, we establish a Ward identity type result for the expectation of $|G_{ij}^{(o)}|^2$, similarly to \eqref{eq:wardex}. The proof again leverages the idea of local resampling. By local resampling around vertex $o$, we reduce the computation to 
\begin{align}\label{e:Gijo}
   \bE[| G_{ij}^{(o)}(z)|^2]=\bE[|\widetilde G_{ij}^{(o)}(z)|^2],
\end{align}
for the switched graph. We then expand it using the Schur complement formula, similarly to \eqref{e:sample}. Crucially, we can bound \eqref{e:Gijo}, by itself times a small factor, and errors as in \eqref{eq:wardex}, leading to the desired bound given by the right-hand side of \eqref{eq:wardex}.

In summary, the iteration scheme presented in this paper  for the computation of the high moments of self-consistent equation  \eqref{e:hmm}  offers a potent method 
to analyze the spectral properties of random $d$-regular graphs. This method yields optimal rigidity for the eigenvalues of random $d$-regular graphs (up to an additional $N^{\oo(1)}$ factor).
Random $d$-regular graphs can also be constructed from $d$ copies of random perfect matchings, or random lifts of a base graph containing two vertices and $d$ edges between them. This class of random graphs obtained from random lifts and in  particular  their extremal  eigenvalues have been extensively studied 
\cite{amit2002random, amit2006random,friedman2003relative, bilu2006lifts, friedman2014relativized, puder2015expansion,lubetzky2011spectra, bordenave2019eigenvalues}. It would be interesting to explore if the approach in this paper can be applied to analyze  extremal eigenvalues in this setting.
 Moreover, our results establish a  $N^{-2/3+\oo(1)}$ fluctuation bound for the second-largest and the smallest eigenvalue, matching that of the Gaussian orthogonal ensemble.  We hope this can be utilized in the future to prove the edge universality of the extremal eigenvalues for random $d$-regular graphs.

\subsection{Related Work}

The eigenvalue statistics of  random graphs have been intensively studied in the past decade. Thanks to a series of papers \cite{bauerschmidt2017bulk, huang2015bulk, lee2018local,erdHos2013spectral,erdHos2012spectral, he2021fluctuations,huang2020transition,huang2022edge,lee2021higher, bauerschmidt2020edge, huang2023edge, he2022spectral}, the bulk and edge statistics of Erd{\H o}s--R{\'e}nyi graphs $G(N,p)$ with $Np\geq N^{\oo(1)}$ and random $d$-regular graphs with $d\geq N^{\oo(1)}$ are now well understood. Universality holds; namely, after proper normalization and shifts, they agree with those from Gaussian orthogonal ensemble. For random $d$-regular graphs, we anticipate such a universality phenomenon holds even for a fixed degree $d\geq 3$. However the situation is dramatically different for very sparse Erd{\H o}s--R{\'e}nyi graphs.

In the sparser regime $Np=\OO(\ln N)$, for Erd{\H o}s--R{\'e}nyi graphs, there exists a critical value $b_*=1/(\ln 4-1)$ such that  if $Np\geq b_*\ln N/N$, the extreme eigenvalues of the normalized adjacency matrix converge to $\pm 2$ \cite{benaych2019largest, alt2021extremal, tikhomirov2021outliers, benaych2020spectral}, and all the eigenvectors are delocalized \cite{alt2022completely, erdHos2013spectral}. For $(\ln  \ln  N)^4\ll Np<b_*\ln N/N$, there exist outlier eigenvalues \cite{tikhomirov2021outliers, alt2021extremal}.
The spectrum splits into three phases: a delocalized phase
in the bulk, a fully localized phase near the spectral edge, and a semilocalized phase in between \cite{alt2023poisson,alt2021delocalization}. Moreover, the joint fluctuations of the eigenvalues near the spectral edges form a Poisson point process. For constant degree Erd\"os-R\'enyi graphs, it was proven in \cite{hiesmayr2023spectral} that the largest eigenvalues are determined by small neighborhoods around vertices of close to maximal degree and the corresponding eigenvectors are localized.

This paper focuses on the eigenvalue statistics. The eigenvectors of random $d$-regular graphs are also important. The complete delocalization of eigenvectors in $\ell_\infty$ norm was proven in previous works \cite{bauerschmidt2019local,huang2024spectrum}. For sparse random regular graphs, several results provide information on the structure of eigenvectors without relying on a local law.  For example, random regular graphs are quantum ergodic
\cite{anantharaman2015quantum}, their local eigenvector statistics converge to those of multivariate Gaussians \cite{backhausz2019almost}, and their high-energy eigenvalues have many nodal domains \cite{ganguly2023many}. Gaussian statistics have been conjectured in broad generality for chaotic systems in both the manifold and graph setting \cite{berry1977regular}. A rich line of research exists toward this idea. For an overview in the manifold setting, see the book \cite{zelditch2017eigenfunctions}, and for the graph setting, refer to \cite{smilansky2013discrete}.

As mentioned above, the second eigenvalue governs the spectral gap of the matrix. Showing the expansion properties of random regular graphs and attempting to find deterministic families of graphs that exhibit these same properties has been a major area of research. For example, the simple random walk on random regular graphs is known to have cutoff \cite{lubetzky2010cutoff}, and these graphs are optimal vertex expanders (see \cite{hoory2006expander}).

\subsection{Outline of the Paper}
In \Cref{s:preliminary}, we recall the concept of local resampling and present results on the estimation of the Green's function of random $d$-regular graphs from \cite{huang2024spectrum}. In \Cref{s:outline}, we state our main results, \Cref{t:recursion}, regarding the high moments estimate of the self-consistent equation, and prove both \Cref{thm:eigrigidity} and \Cref{t:rigidity}, utilizing \Cref{t:recursion} as input. We also outline an iteration scheme to prove \Cref{t:recursion}. For this iteration scheme, we will repeatedly express the Green's functions after local resampling in terms of the original Green's functions before local resampling. In \Cref{sec:expansions}, we gather estimates on the difference in Green's functions before and after local resampling. Additionally, in \Cref{e:error_term}, we collect bounds, which will be used to constrain various error terms encountered in the iteration scheme. Finally, the proof of \Cref{t:recursion} is detailed in \Cref{s:proof_main}.

\subsection{Notation}
We reserve letters in mathfrak mode, e.g. $\fa, \fb, \fc$, to represent small universal constants, and $\fC$ for large universal constants, which
may be different from line by line. We use letters in mathcal mode, e.g. $\cB, \cG, \cT, \cF$, to represent graphs, or subgraphs, and letters in mathbb mode, e.g. $ \bT, \mathbb X$, to represent set of vertices. 
For two quantities $X$ and $Y$ depending on $N$, 
we write that $X = \OO(Y )$ or $X\lesssim Y$ if there exists some universal constant such
that $|X| \leq \fC Y$ . We write $X = \oo(Y )$, or $X \ll Y$ if the ratio $|X|/Y\rightarrow \infty$ as $N$ goes to infinity. We write
$X\asymp Y$ if there exists a universal constant $\fC>0$ such that $ Y/\fC \leq |X| \leq  \fC Y$. We remark that the implicit constants may depend on $d$.
We denote $\qq{a,b} = [a,b]\cap\bZ$ and $\qq{N} = \qq{1,N}$. We say an event $\Omega$ holds with high probability if $\bP(\Omega)\geq 1-\oo(1)$; we say an event $\Omega$ holds with overwhelming high probability, if for any $\fC>0$, 
$\bP(\Omega)\geq 1-N^{-\fC}$ holds provided $N$ is large enough. For two random variables $X_N, Y_N\geq 0$, we write $X_N\prec Y_N$ to mean that for any small $\delta>0$ and $N$ large enough, with overwhelmingly high probability it holds $X_N\leq N^\delta Y_N$,

 \subsection*{Acknowledgements.}
 The research of J.H. is supported by  NSF grant DMS-2331096 and the Sloan research award. 
The research of T.M. is supported by NSF Grant DMS-2212881.
The research of H-T.Y. is supported by NSF grants DMS-1855509 and DMS-2153335. 

\section{Preliminaries}\label{s:preliminary}

In this section we recall some results and concepts from \cite{huang2024spectrum}. 
In \Cref{s:pre} we recall some useful estimates on the Green's functions of random $d$-regular graphs. In \Cref{s:local_resampling}, we recall local resampling and its properties. 

\subsection{Green's Function Estimations}\label{s:pre}
Here, and throughout the following, we use the notation
$
z = E + \ri \eta
$
for the decomposition of $z$ in the upper half complex plane, into its real and imaginary parts, and $\kappa=\kappa(z)=\min\{|E-2|, |E+2|\}$.

In the rest of this paper, we fix the degree $d\geq 3$, and arbitrarily small constants $0<\fc<\fa<1$.  We also take a large radius $\fR:=(\fc/4)\log_{d-1}N$,  a length parameter $(\fc/128)\log_{d-1} N\leq \ell \leq (\fc/64)\log_{d-1} N$ so that $N^{\fc/128}\leq (d-1)^\ell \leq N^{\fc/64}$. We  recall the spectral domain from \eqref{e:D}
\begin{equation} \label{e:D0}
  \mathbf D \deq \{z=E+\ri \eta\col |E| \leq 2+\fa, 0< \eta \leq 1/\fa, N\eta\sqrt{\min\{|E-2|, |E+2|\}+\eta}\geq N^{\fa}\}.
\end{equation}

We recall the integer $\omega_d\geq 1$ from \cite[Definition 2.6]{huang2024spectrum}, which is the number of cycles a graph of degree at most $d$ can have while its Green's function has exponential decay.  Rather than go through the technical definition, the important properties are that $\omega_d\geq 1$ and that $\omega_d$ is nondecreasing in $d$.

\begin{definition}\label{def:omegabar}
Fix $d\geq 3$ and a sufficiently small $0<\fc<1$, $\fR=(\fc /4)\log_{d-1}N$. We define the event $\oOmega$,  where the following occur: 
    \begin{enumerate}
        \item 
The number of vertices that do not have a radius $\fR$ tree neighborhood is at most $N^{\fc}$.
        \item 
        The radius $\fR$ neighborhoods of all vertices have excess at most $\omega_d$. 
    \end{enumerate}
\end{definition}

The event $\oOmega$ is a typical event. The following proposition from \cite[Proposition 2.1]{huang2024spectrum} states that $\oOmega$ holds with high probability. 
\begin{proposition}[{\cite[Proposition 2.1]{huang2024spectrum}}]\label{lem:omega}
$\oOmega$ occurs with probability $1-\OO(N^{-(1-\fc)\omega_d})$.
\end{proposition}
As we will see, once we are in $\oOmega$,  Green's functions can be approximated by tree extensions with overwhelmingly high probability. Our starting point is the bound from {\cite[Theorem 4.2]{huang2024spectrum}}, which gives a polynomially small bound on the approximation of the Green's functions.  For this, we recall $Q(z;\cG)$ defined in \eqref{e:Qsum}. We recall the idea of a Green's function extension with general weight $\Delta(z)$ from {\cite[Section 2.3]{huang2024spectrum}}.

\begin{definition}\label{def:pdef}
    Fix degree $d\geq 3$, and a graph $\cT$ with degrees bounded by $d$. We define the function $P(\cT,z,\Delta(z))$ as follows. Define $H(\cT)$ to be the normalized adjacency matrix of $\cT$, and $D(\cT)$ to be the diagonal matrix of degrees of $\cT$. Then 
    \begin{align}\label{e:defP}
    P(\cT,z,\Delta(z)):=\frac{1}{-z+H(\cT)-\frac{\Delta(z)}{d-1}(d\mathbb I-D(\cT))}.
    \end{align}
\end{definition}
When $\Delta(z)=\msc(z)$, \eqref{e:defP} is the Green's function of the tree extension of $\cT$, i.e. extending $\cT$ by attaching copies of infinite $(d-1)$-ary trees to $\cT$ to make each vertex degree $d$.
We recall from \cite[Proposition 2.12]{huang2024spectrum}, if the graph $\cT$ has excess at most $\omega_d$, and ${\rm diam}(\cT)|\Delta(z)-\msc(z)|\ll 1$, then for any $i,j\in \cT$
\begin{align}\label{e:Pijbound}
    |P_{ii}(\cT, z, \Delta(z))|\asymp 1,
    \quad |P_{ij}(\cT, z, \Delta(z))|\lesssim \left(\frac{|\msc(z)|}{\sqrt{d-1}}\right)^{\dist_\cT(i,j)}.
\end{align}

Typically, we will take our graph $\cT$ to be the ball of $\cG$ with some fixed radius, thus for any $r\geq 1$, we define 
$\cB_r(i,\GG)$ to be the ball of radius $r$ around vertex $i$ in $\GG$. We similarly define $\cB_r(i,j,\GG)=\cB_r(i,\GG)\cup \cB_r(j,\GG)$.  The following theorem gives polynomial bounds for the Green's functions and the Stieltjes transform of the empirical eigenvalue distribution. 
\begin{theorem}[{\cite[Theorem 4.2]{huang2024spectrum}}] \label{thm:prevthm}
Fix any sufficiently small $0<\fr<\fa<1$, and $r=\fr \log_{d-1}N$.  For any $z:=E+\ri \eta$, we define $\kappa=\min\{|E-2|, |E+2|\}$, and the error parameters
\begin{align}\label{e:defepsilon}
\varepsilon'(z):=(\log N)^{100}\left(N^{-\fr}+\sqrt{\frac{\Im[m_d(z)]}{N\eta}}+\frac{1}{(N\eta)^{2/3}}\right),\quad \varepsilon(z):=\frac{\varepsilon'(z)}{\sqrt{\kappa(z)+\eta(z)+\varepsilon'(z)}}.
\end{align}
For any $\fC\geq 1$ and $N$ large enough, with probability at least $1-\OO(N^{-\fC})$ with respect to the uniform measure on $\oOmega$, 
\be\label{eq:infbound}
|G_{ij}(z)-P_{ij}(\cB_{r}(i,j,\GG),z,m_{sc}(z))|,\quad |Q(z;\cG)-m_{sc}(z)|,\quad |m(z)-m_d(z)|\lesssim \varepsilon(z).
\ee
uniformly for every $i,j\in \qq{N}$, and any $z=E+\ri \eta$ with $\eta\geq N^{-1+\fa}$. We denote the event that \eqref{eq:infbound} holds as $\Omega(z)$. 
\end{theorem}

For any integer $\ell\geq 1$, we define the functions $X_\ell(\Delta(z),z), Y_\ell(\Delta(z),z)$ as
\begin{align}\label{def:Y}
X_\ell(\Delta(z),z)=P_{oo}(\cB_\ell(o,\cX),z,\Delta(z)),\quad Y_\ell(\Delta(z),z)=P_{oo}(\cB_\ell(o,\cY),z,\Delta(z)),
\end{align}
where $\cX$ is the infinite $d$-regular tree with root vertex $o$, and $\cY$ is the infinite $(d-1)$-ary tree with root vertex $o$. 
If the context is clear, we will simply write $X_\ell(\Delta(z),z)=X_\ell(\Delta(z))$ and $Y_\ell(\Delta(z),z)=Y_\ell(\Delta(z))$.

Then $m_{sc}(z)$ is a fixed point of the function $Y_\ell$, i.e. $Y_\ell(m_{sc}(z))=\msc(z)$. And $X_\ell(\msc(z))=\md(z)$. If we take $\Delta(z)=Q(z;\cG)$ as in \eqref{e:Qsum}, then $Y_\ell(Q(z;\cG))$ recovers $Y_\ell(Q)$ as in \eqref{e:selfeq}. The following proposition states that if $\Delta(z)$ is sufficiently close to $m_{sc}(z)$, then $Y_\ell(\Delta(z))$ is close to $m_{sc}(z)$, $X_\ell(\Delta(z))$ is close to $\md(z)$.

\begin{proposition}[{\cite[Proposition 2.10]{huang2024spectrum}}]\label{p:recurbound}
Given a function $\Delta(z)$, if $\ell|\Delta(z)-m_{sc}(z)|\ll 1$, then the functions $X_\ell(\Delta(z)), Y_\ell(\Delta(z))$ 
satisfy
\begin{align}\begin{split}\label{e:Xrecurbound}
&\phantom{{}={}}X_\ell(\Delta(z))-\md(z)
=\frac{d}{d-1}(\md(z))^2(m_{sc}(z))^{2\ell}(\Delta(z)-m_{sc}(z))
+\OO\left(\ell|\Delta(z)-m_{sc}(z)|^2\right),
\end{split}\end{align}
and
\begin{align}\begin{split}\label{e:recurbound}
&\phantom{{}={}}Y_\ell(\Delta(z))-m_{sc}(z)
=(m_{sc}(z))^{2\ell+2}(\Delta(z)-m_{sc}(z))\\
&+(m_{sc}(z))^{2\ell+3}\left(\frac{1-(m_{sc}(z))^{2\ell+2}}{1-(m_{sc}(z))^{2}}\right)(\Delta(z)-m_{sc}(z))^2
+\OO\left(\ell^2|\Delta(z)-m_{sc}(z)|^3\right).
\end{split}\end{align}
\end{proposition}

We will later use the Taylor expansion of $Y_\ell(\Delta(z))$ around $\msc(z)$ and utilize the fact that according to \eqref{e:recurbound} the derivatives of $Y_\ell(\Delta(z))$ satisfies,
\begin{align}\begin{split}\label{eq:Yprime}
&Y_\ell'(\Delta(z))=m_{sc}^{2\ell+2}+2m_{sc}^{2\ell+3}\left(\frac{1-m_{sc}^{2\ell+2}}{1-m_{sc}^2}\right)(\Delta-m_{sc})+O(\ell^2|\Delta-m_{sc}|^2),\\
&Y_\ell''(\Delta(z))=2m_{sc}^{2\ell+3}\left(\frac{1-m_{sc}^{2\ell+2}}{1-m_{sc}^2}\right)+O(\ell^2|\Delta-m_{sc}|).
\end{split}\end{align}

As a consequence of \Cref{thm:prevthm}, for $\cG\in \Omega(z)$, there is some sufficiently large integer $\fb\geq 1$ such that for any $\eta\geq N^{-1+\fa}$,
\begin{equation}
|Q(z;\cG)-Y_\ell(Q(z;\cG))|^\fb, \quad |Q(z;\cG)-\msc(z)|^\fb\leq \frac{1}{N^2},
\end{equation}
where $Y_\ell$ is as defined in \eqref{def:Y}.
Moreover, for any $i,j\in \qq{N}$, if $\cB_{\fR/2}(i,j,\cG)$ is a tree (does not contain any cycle), then
\begin{align*}
 |G_{ij}(z)-P_{ij}(z)|\lesssim \varepsilon(z),\quad  |G_{ij}(z)-P_{ij}(z)|^{\fb}\leq \frac{1}{N^2},
\end{align*}
where $P_{ij}$ is the Green's function of infinite $d$-regular tree (the tree extension of $\cB_{\fR/2}(i,j,\cG)$).

\subsection{Local Resampling}  
\label{s:local_resampling}

In this section, we recall the local resampling and its properties. This gives us the framework to talk about resampling from the random regular graph distribution as a way to get an improvement in our estimates of the Green's function.

For any graph $\cG$, we denote the set of unoriented edges by $E(\cG)$,
and the set of oriented edges by $\vec{E}(\cG):=\{(u,v),(v,u):\{u,v\}\in E(\cG)\}$.
For a subset $\vec{S}\subset \vec{E}(\cG)$, we denote by $S$ the set of corresponding non-oriented edges.
For a subset $S\subset E(\cG)$ of edges we denote by $[S] \subset \qq{N}$ the set of vertices incident to any edge in $S$.
Moreover, for a subset $\bV\subset\qq{N}$ of vertices, we define $E(\cG)|_{\bV}$ to be the subgraph of $\cG$ induced by $\bV$.

\begin{definition}
A (simple) switching is encoded by two oriented edges $\vec S=\{(v_1, v_2), (v_3, v_4)\} \subset \vec{E}$.
We assume that the two edges are disjoint, i.e.\ that $|\{v_1,v_2,v_3,v_4\}|=4$.
Then the switching consists of
replacing the edges $\{v_1,v_2\}, \{v_3,v_4\}$ by the edges $\{v_1,v_4\},\{v_2,v_3\}$.
We denote the graph after the switching $\vec S$ by $T_{\vec S}(\cG)$,
and the new edges $\vec S' = \{(v_1,v_4), (v_2,v_3)\}$ by
$
  T(\vec S) = \vec S'
$.
\end{definition}

The local resampling involves a fixed center vertex, we now assume to be vertex $o$,
and a radius $\ell$.
Given a $d$-regular graph $\cG$, we abbreviate $\cT=\cB_{\ell}(o,\cG)$ (which may not be a tree) and its vertex set $\bT$.\index{$\cT, \bT$}
The edge boundary $\del_E \cT$ of $\cT$ consists of the edges in $\cG$ with one vertex in $\bT$ and the other vertex in $\qq{N}\setminus\bT$.
We enumerate the edges of $\del_E \cT$ as $ \del_E \cT = \{e_1,e_2,\dots, e_\mu\}$, where $e_\al=\{l_\al, a_\al\}$ with $l_\al \in \bT$ and $a_\al \in \qq{N} \setminus \bT$. We orient the edges $e_i$ by defining $\vec{e}_i=(l_\al, a_\al)$.
We notice that $\mu$ and the edges $e_1,e_2, \dots, e_\mu$ depend on $\cG$. The edges $e_\al$ are distinct, but
the vertices $a_\al$ are not necessarily distinct and neither are the vertices $l_\al$. Our local resampling switches the edge boundary of $\cT$ with randomly chosen edges in $\cG^{(\bT)}$
if the switching is admissible (see below), and leaves them in place otherwise.
To perform our local resampling, we choose $(b_1,c_1), \dots, (b_\mu,c_\mu)$ to be independent, uniformly chosen oriented edges from the graph $\cG^{(\bT)}$, i.e.,
the oriented edges of $\cG$ that are not incident to $\bT$,
and define 
\begin{equation}\label{e:defSa}
  \vec{S}_\al= \{\vec{e}_\al, (b_\al,c_\al)\},
  \qquad
  {\bf S}=(\vec S_1, \vec S_2,\dots, \vec S_\mu).
\end{equation}
The sets $\bf S$ will be called the \emph{resampling data} for $\cG$. We remark that repetitions are allowed in the resampling data $(b_1, c_1), (b_2, c_2),\cdots, (b_\mu, c_\mu)$.
We define an indicator that will be crucial to the definition of switch.

\begin{definition}
For $\al\in\qq{\mu}$,
we define the indicator functions
$I_\al \equiv I_\al(\cG,{\bf S})=1$\index{$I_\alpha$} 
\begin{enumerate}
\item
 the subgraph $\cB_{\fR/4}(\{a_\al, b_\al, c_\al\}, \cG^{(\bT)})$ after adding the edge $\{a_\al, b_\al\}$ is a tree;
\item 
and $\dist_{\cG^{(\bT)}}(\{a_\al,b_\al,c_\al\}, \{a_\beta,b_\beta,c_\beta\})> {\fR/4}$ for all $\beta\in \qq{\mu}\setminus \{\al\}$.
\end{enumerate}
\end{definition}
 The indicator function $I_\alpha$ imposes two conditions. The first one is a ``tree" condition, which ensures that 
 $a_\al$ and $\{b_\al, c_\al\}$ are far away from each other, and their neighborhoods are trees. 
The second one imposes an ``isolation" condition, which ensures that we only perform simple switching when the switching pair is far away from other switching pairs. In this way, we do not need to keep track of the interaction between different simple switchings. 

We define the \emph{admissible set}
\begin{align}\label{Wdef}
{\mathsf W}_{\bf S}:=\{\al\in \qq{\mu}: I_\al(\cG,{\bf S}) \}.
\end{align}
We say that the index $\al \in \qq{\mu}$ is \emph{switchable} if $\al\in {\mathsf W}_{\bf S}$. We denote the set $\bW_{\bf S}=\{b_\al:\al\in {\mathsf W}_{\bf S}\}$\index{$\bW_{\bf S}$}. Let $\nu:=|{\mathsf W}_{\bf S}|$ be the number of admissible switchings and $\al_1,\al_2,\dots, \al_{\nu}$
be an arbitrary enumeration of ${\mathsf W}_{\bf S}$.
Then we define the switched graph by
\begin{equation} \label{e:Tdef1}
T_{\bf S}(\cG) := \left(T_{\vec S_{\al_1}}\circ \cdots \circ T_{\vec S_{\al_\nu}}\right)(\cG),
\end{equation}
and the resampling data by
\begin{equation} \label{e:Tdef2}
  T({\bf S}) := (T_1(\vec{S_1}), \dots, T_\mu(\vec{S_\mu})),
  \quad
  T_\al(\vec{S}_\al) \deq
  \begin{cases}
    T(\vec{S}_\al) & (\al \in {\mathsf W}_{\bf S}),\\
    \vec{S}_\al & (\al \not\in {\mathsf W}_{\bf S}).
  \end{cases}
\end{equation}

To make the structure more clear, we introduce an enlarged probability space.
Equivalent to the definition above, the sets $\vec{S}_\al$ as defined in \eqref{e:defSa} are uniformly distributed over 
\begin{align*}
{\sf S}_{\al}(\cG)=\{\vec S\subset \vec{E}: \vec S=\{\vec e_\al, \vec e\}, \text{$\vec{e}$ is not incident to $\cT$}\},
\end{align*}
i.e., the set of pairs of oriented edges in $\vec{E}$ containing $\vec{e}_\al$ and another oriented edge in $\cG^{(\bT)}$.
Therefore ${\bf S}=(\vec S_1,\vec S_2,\dots, \vec S_\mu)$ is uniformly distributed over the set
${\sf S}(\cG)=\sf S_1(\cG)\times \cdots \times \sf S_\mu(\cG)$.

We introduce the following notations on the probability and expectation with respect to the randomness of the $\bfS\in \sf S(\cG)$.
\begin{definition}\label{def:PS}
    Given any $d$-regular graph $\cG$, we 
    denote $\bP_\bfS(\cdot)$ the uniform probability measure on ${\sf S}(\cG)$;
 and $\bE_\bfS[\cdot]$ the expectation  over the choice of $\bfS$ according to $\bP_\bfS$. 
\end{definition}

 The following claim from \cite[Lemma 7.3]{huang2024spectrum} states that this switch is invariant under the random regular graph distribution.

\begin{lemma}[{\cite[Lemma 7.3]{huang2024spectrum}}] \label{lem:exchangeablepair}
Fix $d\geq 3$. We recall the operator $T_\bfS$ from \eqref{e:Tdef1}. Let $\cG$ be a random $d$-regular graph  and $\bfS$ uniformly distributed over $\sfS(\cG)$, then the graph pair $(\cG, T_{\bf S}(\cG))$ forms an exchangeable pair:
\begin{align*}
(\cG, T_{\bf S}(\cG))\stackrel{law}{=}(T_{\bf S}(\cG), \cG).
\end{align*}
\end{lemma}

In the following two lemmas we show that with high probability with respect to the randomness of $\bfS$, the randomly selected edges $(b_\al, c_\al)$ are far away from each other, and have large tree neighborhood. In particular $T_\bfS \cG\in \oOmega$. 
\begin{lemma}\label{lem:configuration}
Fix $d\geq 3$, a $d$-regular graph $\cG\in \oOmega$ and a vertex set $\bF\subset \qq{N}$ containing vertex $o$, with $|\bF|\leq N^{\fc/2}$. We consider the local resampling around vertex $o$, with resampling data $\{(l_\al, a_\al), (b_\al, c_\al)\}_{\al\in \qq{\mu}}$. Then with probability at least $1-N^{-1+2\fc}$ over the randomness of the resampling data ${\bf S}$, we have for any $\al\neq \beta\in \qq{\mu}$, $\dist_\cG(\{b_\al, c_\al\}\cup \bF,\{b_\beta, c_\beta\} )\geq3\fR$; and 
for any  $v\in \cB_{\ell}(\{b_\al, c_\al\}_{\al\in \qq{\mu}}, \cG)$, the radius $\fR$ neighborhood of $v$ is a tree. 
\end{lemma}
\begin{lemma} \label{lem:goodresamplingdata}
Fix $d\geq 3$, a $d$-regular graph $\cG\in \oOmega$. We denote the set of resampling data $\sfF(\cG)\subset \sfS(\cG)$ such that for any $\al\neq \beta\in \qq{\mu}$, $\dist_\cG(\{b_\al, c_\al, o\},\{b_\beta, c_\beta\} )\geq3\fR$, and
for any $v\in \{l_\al, b_\al, c_\al\}_{\al\in \qq{\mu}}$, $\cB_{\fR}(v, \cG)$ is a tree. The for any $\bfS\in \sfF(\cG)$, it holds that 
then $\mu=d(d-1)^\ell$, ${\mathsf W}_{\bf S}=\qq{\mu}$, and $\tcG=T_\bfS(\cG)\in \oOmega$. Moreover, $\cB_\fR(v,\tcG)$ is a tree for any $v\in \cB_\ell(o, \cG)$.
\end{lemma}
\begin{proof}[Proof of \Cref{lem:configuration}]
    We sequentially select $(b_\al, c_\al)$ uniformly random from $\cG^{(\bT)}$. For any fixed $\alpha$, we consider all edges that would break the requirements of the lemma. For the first requirement, we have
    \begin{align}\label{e:small1}
        \bP_{\bfS}(\dist_\cG(\bF\cup_{1\leq \beta\leq \al-1}\{b_\beta, c_\beta\},\{b_\al, c_\al\} )\leq3\fR)\lesssim N^{-1}(|\bF|+2\mu)d(d-1)^{3\fR}\leq N^{-1+3\fc/2},
    \end{align}
    where $\bP_{\bfS}(\cdot)$ is the probability with respect to the randomness of $\bfS$ as in \Cref{def:PS}.
    For the second requirement, we recall that  $\cG\in \oOmega$, in which all vertices except for $N^\fc$ many have radius $\fR$ tree neighborhood. Thus  
    \begin{align}\label{e:small2}
        \bP_{\bfS}(\cB_\fR(v,\cG) \text{ is not a tree for some } v\in \cB_\ell(\{b_\al, c_\al\}, \cG))\leq N^{-1} N^\fc d(d-1)^\ell\leq N^{-1+3\fc/2}.
    \end{align}
   The claim of \Cref{lem:configuration} follows from union bounding over all $\alpha$ using \eqref{e:small1} and \eqref{e:small2}. 
\end{proof}

\begin{proof}[Proof of \Cref{lem:goodresamplingdata}]
    Under our assumption, the radius $\fR/2\geq \ell$ neighborhood of $o$, $\cB_{\fR/2}(o, \cG)$, is a truncated $d$-regular tree at level $\fR/2$. Thus $\mu=d(d-1)^\ell$. Moreover, the neighborhoods $\cB_{3\fR/2}(o, \cG)$, and $\cB_{\fR}(\{b_\al, c_\al\}, \cG)$ for $\al\in \qq{\mu}$ are disjoint. It follows that $\dist_{\cG^{(\bT)}}(\{a_\al,b_\al,c_\al\}, \{a_\beta,b_\beta,c_\beta\})> {\fR/4}$ for all $\al\neq \beta\in \qq{\mu}$, and the subgraph $\cB_{\fR/4}(\{a_\al, b_\al, c_\al\}, \cG^{(\bT)})$ after adding the edge $\{a_\al, b_\al\}$ is a tree for all $\al \in \qq{\mu}$. We conclude that ${\mathsf W}_{\bf S}=\qq{\mu}$.

    Next we show that for any vertex $v\in \qq{N}$, the excess of $\cB_{\fR}(v, \tcG)$ is no bigger than that of $\cB_{\fR}(v, \cG)$. Then it follows that $\tcG\in \oOmega$. 
    If $\dist(v, \{a_\al, b_\al, c_\al\}_{\al\in \qq{\mu}})\geq \fR$, then $\cB_{\fR}(v, \tcG)=\cB_{\fR}(v, \cG)$, and the statement follows.
    Otherwise either $v\in \cB_{\fR}(\{a_\al\}_{\al\in \qq{\mu}}, \cG)\subset\cB_{3\fR/2}(o, \cG)$ or $v\in \cB_{\fR}(\{b_\al, c_\al\}, \cG)$ for some $\al\in \qq{\mu}$. We will discuss the first case. The second case can be proven in the same way, so we omit. 
    If $v\in \cB_\ell(o, \cG)$, we denote $\min_{\al \in \qq{\mu}}\dist_\cG(v, \{l_\al\})=r\leq \fR$. Then $\cB_{\fR}(v, \tcG)$ is a subgraph of 
    $\cB_{\fR}(v, \cG)\cup_{\al\in \qq{\mu}} \cB_{\fR-r-1}(\{c_\al,\cG)$ after removing $\{(b_\al,c_\al)\}_{\al\in \qq{\mu}}$ and adding $\{(l_\al, c_\al)\}_{\al\in \qq{\mu}}$. By our assumption, $\cB_{\fR-r-1}(c_\al,\cG)$ are disjoint trees. We conclude that $\cB_{\fR}(v, \tcG)$ is a tree.
    If $v\not\in \cB_\ell(o, \cG)$,
    we denote $\min_{\al \in \qq{\mu}}\dist_\cG(v, a_\al)=r\leq \fR$, then $\cB_{\fR}(v, \cG)$ is a subgraph of 
    $\cB_{\fR}(v, \tcG)\cup_{\al\in \qq{\mu}} 
    \cB_{\fR-r-1}(b_\al,\cG)$ 
    after removing $\{(b_\al,c_\al)\}_{\al\in \qq{\mu}}$ and adding $\{(a_\al, b_\al)\}_{\al\in \qq{\mu}}$. 
   Again by our assumption, $\cB_{\fR-r-1}(c_\al,\cG)$ are disjoint trees, we conclude the excess of $\cB_{\fR}(v, \tcG)$ is at most that of $\cB_{\fR}(v, \cG)$.

\end{proof}

\section{Main Results and Proof Outlines}\label{s:outline}
In \Cref{s:mainresult}, we state our main results \Cref{t:recursion} on the high moments estimate of the self-consistent equation.  
The proof of \Cref{t:recursion} follows an iteration scheme, with foundational concepts introduced in \Cref{s:forest} and the proof outlined in \Cref{t:proofoutline}. Furthermore, the proofs of both \Cref{thm:eigrigidity} and \Cref{t:rigidity}, utilizing \Cref{t:recursion} as an input, are detailed in \Cref{s:proofrigidity}.
 
\subsection{Main Theorems}\label{s:mainresult}
We will use some more flexible error terms, which we define in the following. We notice that they depend on the graph $\cG$ and thus are random quantities. For simplicity of notation, we simply write $Q(z)=Q(z;\cG)$.

\begin{definition}\label{def:phidef}
For any $z=E+\ri\eta$ in the upper half plane, we introduce the control parameter 
\begin{align*}
 \Phi=\Phi(z):=N^\fc\cdot \frac{\Im[m(z)]}{N\eta}. 
\end{align*}
For any $p\geq 1$, we define the error parameter
\begin{align}\begin{split}\label{eq:phidef}
\Psi_p=\Psi_p(z)
&:= \bm1(\cG\in \Omega(z))\bigg[(d-1)^{-\ell/2}(|Q(z)-Y_\ell(Q(z))|+\Phi(z))^{2p}\\
&+\sum_{k=0}^{2p-1}(|Q(z)-Y_\ell(Q(z))|+\Phi(z))^{k}(\Phi(z)+N^{-1+2\fc})^{2p-k}\\
&+(|Q(z)-Y_\ell(Q(z))|+\Phi(z))^{2p-2}\Phi(z)\frac{|1-Y'_\ell(Q(z))|}{N\eta}\\
&+(d-1)^{2\ell}|Q(z)-\msc(z)|^2(|Q(z)-Y_\ell(Q(z))|+\Phi(z))^{2p-1}\bigg].
\end{split} \end{align}  
 \end{definition}

The following is our main result on the high moments estimate of the self-consistent equation.
\begin{theorem}
\label{t:recursion}
For any integer $d\geq 3$ and sufficiently small $0<\fc<\fa<1$,  the Green's function of  a random $d$-regular graph satisfies that,  for any $z=E+\ri\eta\in  \mathbf D$ (recall from \eqref{e:D})  and integer $p\geq 1$,
\begin{align}\begin{split}
\label{e:QY}&\phantom{{}={}}\bE[\bm1(\cG\in \Omega(z))|Q(z)-Y_\ell(Q(z),z)|^{2p}]\lesssim \bE[\Psi_p(z)].
\end{split}\end{align}
Moreover,  \eqref{e:QY} holds with the left side replaced by 
 $\bE[\bm1(\cG\in \Omega(z))|m(z)-X_\ell(Q(z),z)|^{2p}]$.
\end{theorem}

An easy reformulation of \Cref{t:recursion} gives a useful bound, as stated in the following corollary. 
\begin{corollary}\label{cor:QYQbound}
Adopt the notations and assumptions of Theorem \ref{t:recursion}, the following holds
\begin{align}\begin{split}\label{e:QYQbound}
&\phantom{{}={}}\bE[\bm1(\cG\in \Omega(z))|Q(z)-Y_\ell(Q(z),z)|^{2p}]\\
&\lesssim N^{4p\fc}\bE\left[ \bm1(\GG\in \Omega(z))\left(\frac{\Im[m(z)]+\sqrt{(1-Y'_\ell(Q(z)))\Im[m(z)]}}{N\eta} +|Q(z)-\msc(z)|^2\right)^{2p}\right].
\end{split}\end{align}
Moreover,  \eqref{e:QY} holds with the left side replaced by 
 $\bE[{\bm1(\cG\in \Omega(z))}|m(z)-X_\ell(Q(z),z)|^{2p}]$.
\end{corollary}
\begin{proof}[Proof of \Cref{cor:QYQbound}]
We will only prove the statement for $|Q-Y_\ell(Q)|$. The statement for $|m-X_\ell(Q)|$ can be proven by the same way. First we notice that $(|x|+|y|)^q\lesssim |x|^q +|y|^q$ for fixed $q\geq 0$, thus we can rewrite \eqref{e:QY} as
\begin{align}\begin{split}\label{e:Q-YQ_bound}
   &\phantom{{}={}}\bE[\bm1(\cG\in \Omega )|Q-Y_\ell(Q)|^{2p}]\lesssim (d-1)^{-\ell/2}\bE[\bm1(\cG\in \Omega)|Q-Y_\ell(Q)|^{2p}]\\
   &+ \bE\left[\bm1(\cG\in \Omega)|Q-Y_\ell(Q)|^{2p-2}\Phi\frac{|1-Y'_\ell(Q)|}{N\eta}\right]+
   \bE\left[\bm1(\cG\in \Omega) \Phi^{2p-1}\frac{|1-Y'_\ell(Q)|}{N\eta}\right]
  \\
   &+\bE\left[\bm1(\cG\in \Omega)\sum_{k=0}^{2p-1}|Q-Y_\ell(Q)|^{k}(\Phi+N^{-1+2\fc})^{2p-k}\right]+\bE[\bm1(\cG\in \Omega)(\Phi+N^{-1+2\fc})^{2p}]
   \\
& +(d-1)^{2\ell}\left(\bE[\bm1(\cG\in \Omega)|Q-Y_\ell(Q)|^{2p-1}|Q-\msc|^2]+\bE\left[\bm1(\cG\in \Omega)\Phi^{2p-1}|Q-\msc|^2\right]\right).
\end{split}\end{align}
We recall Young's inequality, for $1/q+1/r=1$, $a/q+b/r\geq a^{1/q}b^{1/r}$. Take a small $\lambda>0$, we can bound
\begin{align}\begin{split}\label{e:Young}
    &|Q-Y_\ell(Q)|^{2p-2}\Phi\frac{|1-Y'_\ell(Q)|}{N\eta}
    \leq \lambda |Q-Y_\ell(Q)|^{2p} + \lambda^{1-p}\left( \Phi\frac{|1-Y'_\ell(Q)|}{N\eta}\right)^{p},\\
    &|Q-Y_\ell(Q)|^{k}(\Phi+N^{-1+2\fc})^{2p-k}
    \leq \lambda |Q-Y_\ell(Q)|^{2p}  + \lambda^{1-2p/k}(\Phi+N^{-1+2\fc})^{2p}\\
    &|Q-Y_\ell(Q)|^{2p-1}(d-1)^{2\ell}|Q-\msc|^2
    \leq \lambda |Q-Y_\ell(Q)|^{2p}  + \lambda^{-1/(2p-1)}N^{2p\fc}|Q-\msc|^{4p}.
\end{split}\end{align}
We can then plug \eqref{e:Young} back into \eqref{e:Q-YQ_bound}, we get
    \begin{align*}\begin{split}
&\phantom{{}={}}\bE[\bm1(\cG\in \Omega)|Q-Y_\ell(Q)|^{2p}]\lesssim \lambda\bE[\bm1(\cG\in \Omega)|Q-Y_\ell(Q)|^{2p}]\\
&+N^{4p\fc}\bE\left[ \bm1(\GG\in \Omega)\left(\frac{\Im[m ]+\sqrt{(1-Y'_\ell(Q))\Im[m]}}{N\eta}+|Q-\msc|^2\right)^{2p}\right].
\end{split}\end{align*}
For sufficiently small $\lambda>0$, the claim \eqref{e:QYQbound} follows from rearranging the above expression.

\end{proof}

 \Cref{thm:eigrigidity} and \Cref{t:rigidity} are consequences of  \Cref{cor:QYQbound}, we postpone the proofs of \Cref{thm:eigrigidity} and \Cref{t:rigidity} to \Cref{s:proofrigidity}.

\subsection{Switching Edges and the Forest}\label{s:forest}
As explained in the introduction \Cref{s:newstrategy}, our new strategy to prove \Cref{t:recursion} is an iteration scheme. At each iteration, we conduct local resampling and express the Green's function of the switched graph in terms of the original graph. Subsequently, we demonstrate that each non-negligible term incorporates at least one ``diagonal factor",  denoted as $G_{c_\alpha c_\alpha}^{(b_\alpha)}$, where $(b_\alpha, c_\alpha)$ represents an edge selected during local resampling. With this observation, we execute a local resampling around $c_\alpha$ and repeat the process.

Because almost all neighborhoods in the graph have no cycles, with high probability, the edges involved in local resamplings have large tree neighborhoods, and are typically far from each other. Therefore, we think of the edges involved as a forest. In this section we formalize this iterative scheme using a sequence of forests, which encode all the edges involved in local resamplings.

Fix $\mu:=d(d-1)^\ell$, which is the number of boundary edges for a $d$-regular tree truncated at level $\ell$. We start from a graph $\cF_0:=(\bfi_0=\{o_0,i_0\}, E_0:=\{e=\{o_0,i_0\}\})$, which consists of a single edge. Then we construct $\cF_1$ from $\cF_0$, by 
\begin{itemize}
    \item extending $e$ to a $d$-regular tree truncated at level $\ell$, $T_\ell(o_0)$ (with $o_0$ as the root vertex);
    \item adding $\mu$ boundary edges $\{ e'_\al\}_{\al\in\qq{\mu}}$ to $\cT_\ell(o_0)$;
    \item adding $\mu$ new (directed) switching edges $\{ e_\al\}_{\al\in\qq{\mu}}$;
    \item and adding $\mu_0\leq 2p-1$ special edges $\{ \widehat e_\al\}_{\al\in\qq{\mu_0}}$.
\end{itemize}

\begin{figure}
    \centering
    \includegraphics[scale=0.7]{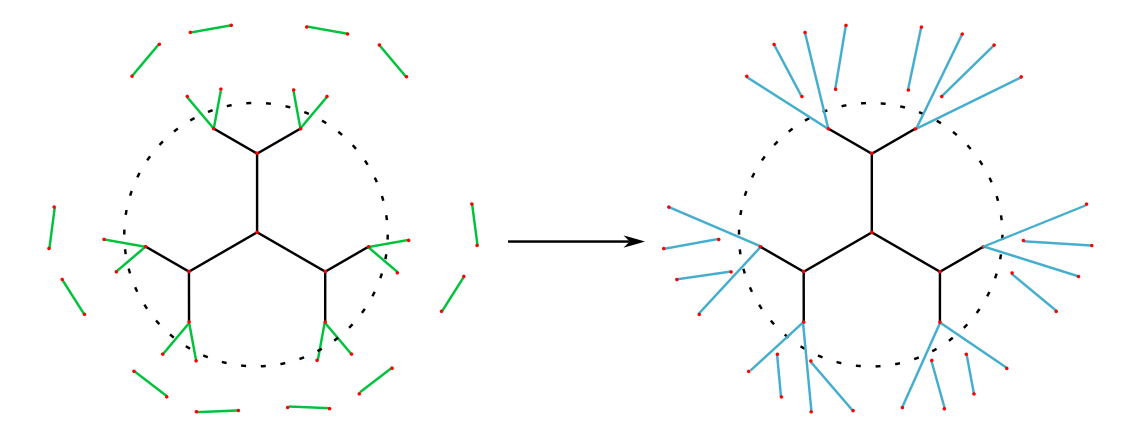}
    \caption{An example of the local resampling performed on the graph. We switch the green edges, on the boundary of the radius $\ell$ neighborhood of a vertex $o$, with randomly chosen edges in the graph. This creates new, blue edges, and a new boundary.}
    \label{fig:switchingproc}
\end{figure}

Therefore
\begin{align}\label{eq:forestdef1}
    \cF_1:=\cT_\ell(o_0)\bigcup\wt\cM_1\cup \cM_1\cup \widehat\cM_1.
\end{align}
where $\cM'_1=\{ e'_\al\}_{\al\in\qq{\mu}}$,
$\cM_1=\{e_\al\}_{\al\in\qq{\mu}}$ and 
$\widehat\cM_1=\{ \widehat e_\al\}_{\al\in\qq{\mu_0}}$.

In general, given the forest $\cF_t=(\bfi_t,E_t)$, with edge sets $\cM'_t, \cM_t, \widehat\cM_t$ constructed in last step, 
we construct $\cF_{t+1}=(\bfi_{t+1},E_{t+1})$ by 
\begin{itemize}
    \item picking one of the switching edges (from last step) $e=(i_t, o_t)\in \cM_t$  and extending $e$ to a $d$-regular tree  truncated at level $\ell$, $\cT_\ell(o_t)$(with $o_t$ as the root vertex);
    \item adding $\mu$ boundary edges $\cM_{t+1}'=\{ e_\al'\}_{\al\in\qq{\mu}}$ to $\cT_\ell(o_t)$;
    \item adding another $\mu$ new (directed) switching edges $\cM_{t+1}=\{e_\al\}_{\al\in\qq{\mu}}$;
    \item and adding $\mu_t$ special edges $\widehat \cM_{t+1}=\{\widehat e_\al\}_{\al \in\qq{\mu_t}}$ such that the total number of special edges is bounded by $2p-1$.
\end{itemize} 
Explicitly, the new forest $\cF_{t+1}=(\bfi_{t+1}, E_{t+1})$ is 
\begin{align}\begin{split}\label{eq:forestdef2}
    &\cF_{t+1}=\cF_t\cup \cT_\ell(o_t)\bigcup\wt \cM_{t+1}\cup \cM_{t+1}\cup \widehat\cM_{t+1},
\end{split}\end{align}
where $\bfi_{t+1}, E_{t+1}$ are its vertex set and edge set, respectively.

The forest $\cF_t$ encodes the local resamplings up to the $t$-th step. Along with this, we denote
\begin{enumerate}
\item 
the set of switching edges $\cK_t=\{(o_0,i_0)\}\cup \cM_1\cup  \cM'_1\cup \cdots \cup \cM_t\cup  \cM'_t$; 
\item 
the set of special edges $\cV_t=\widehat \cM_1\cup \widehat \cM_2\cup\cdots \cup\widehat\cM_t$, which is of size at most $2p-1$;
\item the
core of the forest $\cF_t$ as $\cC_t=(i_0,o_0)\cup \cM_1\cup \cM_2\cdots \cup\cM_t$. Each connected component of $\cF_t$ contains either one core edge, or a special edge;
\item the set of unused core edges
$\cC_t^\circ=\cC_t\setminus \{(i_0,o_0), \cdots, (i_{t-1},o_{t-1})\}$.
\end{enumerate}
For the most of the later analysis, we will focus on one step, and simplify the notations as 
\begin{align}\begin{split}\label{e:cFtocF+}
    &\cF=\cF_t, \quad \cK=\cK_t, \quad \cV=\cV_t, \quad \cC=\cC_t,\quad \cC^\circ =\cC^\circ_t,\\
    &\cF^+=\cF_{t+1}, \quad \cK^+=\cK_{t+1},\quad \cV^+=\cV_{t+1}, \quad \cC^+=\cC_{t+1},\quad (\cC^\circ)^+ =\cC^\circ_{t+1}.
\end{split}\end{align}

We now describe the event that local resamplings select edges that are far away from each other. Our analysis will consist of giving loose bounds when this is not the case, and then a far tighter bound when this does occur. The following indicator functions can be used to make sure that our new resampling data does not fall near the switching edges of any previous switch and are far away from each other. 
\begin{definition}
    \label{def:indicator}
We consider a forest $\cF=(\bfi, E)$ (as in \eqref{e:cFtocF+}) with core edges $\cC$ and special edges $\cV$, viewed as a subgraph of a $d$-regular graph $\cG\in \Omega$. We denote the indicator function $I(\cF, \cG)=1$ if vertices close to core edges have radius $\fR$ tree neighborhoods, and core edges are distance $3\fR$ away from each other. Explicitly,  $I(\cF, \cG)$ is given by
\begin{align*}
   I(\cF, \cG):= \prod_{(u,v)\in \cF}A_{uv}\prod_{(b,c)\in \cC}\left(\prod_{v\in \cB_\ell(c,\cG)}\bm1(\cB_{\fR}(v, \cG) \text{ is a tree})\right) \prod_{(b,c)\neq (b',c')\in \cC}\bm1(\dist_\cG(c,c')\geq 3\fR).
\end{align*}
\end{definition}
    
The following lemma states that this occurs with high probability.  
\begin{lemma}\label{lem:wellbehavedswitch}
Fix a $d$-regular graph $\GG\in \oOmega$, and a forest $\cF=(\bfi, E)$ (as in \eqref{e:cFtocF+}), viewed as a subgraph of a $d$-regular graph $\cG\in \Omega$. If $I(\cF,\cG)=1$, then with probability at least $1-N^{-1+2\fc}$ over the randomness of the resampling data ${\bf S}$, we have $\tcG\in \oOmega$, $I(\cF^+, \cG)=1$ and $I(\cF, \wt\cG)=1$, where $\wt \cG=T_\bfS \cG$.
\end{lemma}

\begin{proof}
Thanks to \Cref{lem:configuration} with $\bF$ the vertex set of $\cF$, with probability $1-N^{-1+2\fc}$ over the randomness of the resampling data ${\bf S}$, we have that for any $\al\neq \beta\in \qq{\mu}$, $\dist_\cG(\{b_\al, c_\al\}\cup \bF,\{b_\beta, c_\beta\} )\geq3\fR$; and 
for any  $v\in \cB_{\ell}(\{b_\al, c_\al\}_{\al\in \qq{\mu}}, \cG)$, the radius $\fR$ neighborhood of $v$ is a tree. It follows that $I(\cF^+, \cG)=1$. It also follows from  \Cref{lem:goodresamplingdata} that $\tcG\in \oOmega$ and $\cB_\fR(v,\tcG)$ is a tree for any $v\in \cB_\ell(o, \cG)$. One can then check that $I(\cF,\tcG)=1$. This finishes the proof of \Cref{lem:wellbehavedswitch}.

\end{proof}

From the construction above, each connected component of $\cF$ is either a single edge, or a radius $\ell+1$ ball, and each connected component contains a core edge or a special edge. The following proposition states that the total number of embeddings where $I(\cF,\cG)=1$ is approximately the same as that of choosing each connected component independently.

\begin{proposition}\label{p:sumA}
Given a forest $\cF=(\bfi, E)$ as in \eqref{e:cFtocF+}, and a $d$-regular graph $\cG\in \oOmega$, we have 
\[
\sum_{\bfi} I(\cF,\cG)=Z_{\cF}\left(1+\OO\left(\frac{1}{N^{1-2\fc}}\right)\right),
\]
where
\begin{align}\label{e:sumA2}\begin{split}
Z_{\cF}=(Nd)^{\theta_1(\cF)}\left(Nd[(d-1)!]^{1+d+d(d-1)+\cdots+d(d-1)^{\ell-1}}\right)^{\theta_2(\cF)}.
\end{split}\end{align}
Here $\theta_1(\cF)$ is the number of connected components in $\cF$ which is a single edge; and $\theta_2(\cF)$ is the number of connected components in $\cF$ which is a ball of radius $\ell+1$. We remark that $Z_\cF$ depends only on the forest $\cF$ but not $\cG$.
\end{proposition}

\begin{proof}[Proof of Proposition \ref{p:sumA}]
\begin{comment}
In the special case that $ \cT$ is a tree, for the sum $\sum_\bfi A_{ \cT}$, we can sequentially sum over leaf vertices. Each gives a factor $d$, and the sum over the root vertex gives a factor $N$. Thus
$\sum_{\bfi} A_{ \cT}= N d^{|\bfi|-1}$. In the general case, when $ \cT$ is a forest. We can sum over each of its connected components. There are $\theta(\cT)$ connected components, and $|E|=|\bfi|-\theta(\cT)$, so \eqref{e:sumA} follows.

We assume that $\cT$ contains exactly one cycle $i_1 i_2 \cdots i_\ell$ of length $\ell$, we can sequentially sum over other vertices as in \eqref{e:sumA}
\begin{align}\label{e:sumAtriangle2}
\sum^\top _{\bfi} A_{\cT}\leq  N^{\theta(\cT)-1} d^{|E|-\ell}\sum^\top _{i_1,i_2, \cdots, i_{\ell}}A_{i_1i_2}A_{i_2i_3}\cdots A_{i_\ell i_1}=N^{\theta(\cT)-1} d^{|E|-\ell}Z,
\end{align}
where $Z$ counts the number of cycles of length $\ell$. To upper bound the righthand side of \eqref{e:sumAtriangle2}, we recall from \cite[Theorem 3]{mckay2004short}: for any finite subgraph $\cJ$ with $|\cJ|$ edges, 
\begin{align}\label{e:ppbb}
\bP(\cJ\subset \cG)=(1+\oo(1))\left(\frac{d}{n}\right)^{|\cJ|}.
\end{align}
Using \eqref{e:ppbb}, for any $k\geq 1$, $\bE[Z^k]\leq C_k d^{k\ell}$. It follows by Markov's inequality
\begin{align*}
\bP(Z\geq d^\ell N^\varepsilon)\leq \frac{C_k}{N^{\varepsilon k}}.
\end{align*}
We conclude that $Z\prec d^\ell$, and the claim \eqref{e:sumAtriangle} follows from combining with \eqref{e:sumAtriangle2}. 
\end{comment}
 We can prove \eqref{e:sumA2} by 
induction on the number of connected components. If $\cF$ consists of a single edge $\cF=e=(b,c)$ which is a core edge, then 
\begin{align}\label{e:single_edge}
    \sum_{\bfi}I(\cF,\cG)=\sum_{b,c} A_{bc}\prod_{v\in \cB_\ell(c,\cG)}\bm1(\cB_{\fR}(v, \cG) \text{ is a tree}) )=Nd\left(1+\OO\left(\frac{1}{N^{1-\fc/2}}\right)\right),
\end{align}
where we used the definition of $\overline\Omega$ from \Cref{def:omegabar}. If $\cF=e$ is a speical edge,
then 
\begin{align}\label{e:single_edge2}
 \sum_{\bfi}I(\cF,\cG)=\sum_{e=(u,v)}A_{uv}=Nd.
\end{align}
If $\cF$ consists of a radius $(\ell+1)$-ball, then we can also first sum over its core edge. The number of choices of this is the same as \eqref{e:single_edge}.  Then we sum over the remaining vertices. Each interior vertex of the radius $(\ell+1)$-ball contributes a factor $(d-1)!$, since there are $(d-1)!$ ways to embed its children vertices. We get
\begin{align}\label{e:ball}
    \sum_{\bfi}I(\cF,\cG)=Nd[(d-1)!]^{1+d+d(d-1)+\cdots+d(d-1)^{\ell-1}}\left(1+\OO\left(\frac{1}{N^{1-\fc/2}}\right)\right).
\end{align}

If the statement holds for $\cF$ with $\theta$ connected components, next we show it for $\cF$ with $\theta+1$ connected components. We can first sum over the indices corresponding to a connected component, fixing the other indices. For the summation, we can again first sum over the core edge $e=(b,c)$ or the special edge. The indicator function $I(\cF,\cG)$ requires that any vertex $v\in \cB_\ell(c,\cG)$ has a radius $\fR$-tree neighborhood, and is distance at least $3\fR$ away from other core edges. By the same argument as in \Cref{lem:wellbehavedswitch}, the number of choices is 
\begin{align*}
     Nd\left(1+\OO\left(\frac{1}{N^{1-3\fc/2}}\right)\right).
\end{align*}
Then we sum over the remaining vertices in this connected component. If it is a single edge, we get a factor similarly to \eqref{e:single_edge} and \eqref{e:single_edge2}; if it is a radius $(\ell+1)$-ball, we get a factor similarly to \eqref{e:ball}. Next we can sum over the remaining $\theta$ connected components of $\cF$, which gives \eqref{e:sumA2}.

\end{proof}

\subsection{Proof outline for \Cref{t:recursion}} \label{t:proofoutline}

In the process of proving Theorem \ref{t:recursion}, we will have to prove bounds on a number of functions that depend on the forests $\cF$ which encodes local resamplings. To classify these functions, we give the following definition.

\begin{definition}
    \label{def:pgen}
Fix a large integer $p\geq 1$. Consider forest $\cF=(\bfi, E)$ as defined in \eqref{e:cFtocF+}, with switching edges $\cK$, unused core edges $\cC^\circ$, and special edges $\cV$. Denote $P=P(\cF, z, m_{sc}(z))$, and we call a function $R_\bfi$ \emph{generic with respect to} $\cF$ if $R_\bfi=(G_{cc}^{(b)}-Y_\ell(Q))R_\bfi'$ for some $(b,c)\in \cC^\circ$, 
where $R_\bfi'$ satisfies the following conditions,
\begin{enumerate}
\item 
$R_\bfi'$ contains an arbitrary number of factors of the form $(d-1)^{2\ell}(G_{ss'}-P_{ss'})$ for $s,s'\in \cK$; $(d-1)^{2\ell}(G_{cc}^{(b)}-Q)$ for $(b,c)\in \cC^\circ$; $(d-1)^{2\ell} G_{cc'}^{(bb')}$ for $(b,c)\neq (b',c')\in \cC^\circ$;  $(d-1)^{2\ell} G_{ss'}$ for $s\in \cK, s'\in \cV$; $G_{uv}, G_{vv}, 1/G_{uu}$ for $(u,v)\in \cV$; $1-Y'_\ell(Q), Q-Y_\ell(Q), (d-1)^{2\ell}(Q-\msc)$
and their complex conjugates.

\item For each special edge $(u,v)\in \cV$, there exist at least two terms of the form $(d-1)^{2\ell}G_{ss'}$ for $s\in\cK, s'\in \{u,v\}$, or their complex conjugates.  The number of $1-Y'_\ell(Q)$ factors and its complex conjugate equals the number of special edges.
\item
The total number of special edges and, $Q-Y_\ell(Q)$ factors and its complex conjugate is $2p-1$.

 \end{enumerate}
\end{definition}  

We can regroup the generic terms, depending how many extra factors they have, as given in the following definition.  
\begin{definition}\label{e:defchi}
Fix a large integer $p\geq 1$ and a forest $\cF=(\bfi, E)$ as defined in \eqref{e:cFtocF+}, and we denote $P=P(\cF, z, \msc)$. For any $r\geq 0$, we define $\cE_{r}^{\cF}$ to be the collection of generic terms $R_\bfi$ which contains exact $r$ factors of the form
\begin{align}\begin{split}\label{e:defcE1}
    &\{(d-1)^{2\ell}(G_{c c}^{(b)}-Q)\}_{(b,c)\in \cC^\circ}, \quad \{(d-1)^{2\ell} G_{c c'}^{(bb')}\}_{(b,c)\neq (b',c')\in \cC^\circ},\\
&\{(d-1)^{2\ell} (G_{ss'}-P_{ss'})\}_{s,s'\in \cK},\quad  (d-1)^{2\ell}(Q-\msc).
\end{split}\end{align}
\end{definition}

Thanks to \Cref{thm:prevthm}, these factors in \eqref{e:defcE1} are quite small, bounded by $(d-1)^{2\ell}\varepsilon(z)$. In particular, for $z\in \bf D$ (as defined in \eqref{e:D}), $(d-1)^{2\ell}\varepsilon(z)\leq N^{-\fc}$. Thus, for a generic term in $\cE_r^\cF$, these additional factors contribute to a factor $(d-1)^{2r\ell}\varepsilon^{r}\leq N^{-r\fc}$. Therefore, if we can demonstrate that after each local resampling, $r$ increases and a term remains generic in the sense of Definition \ref{def:pgen}, then this is a favorable outcome, as after a finite number of steps, this will imply that the term is negligible.

\begin{figure}[t]
\begin{center}
\includegraphics[scale=0.7]{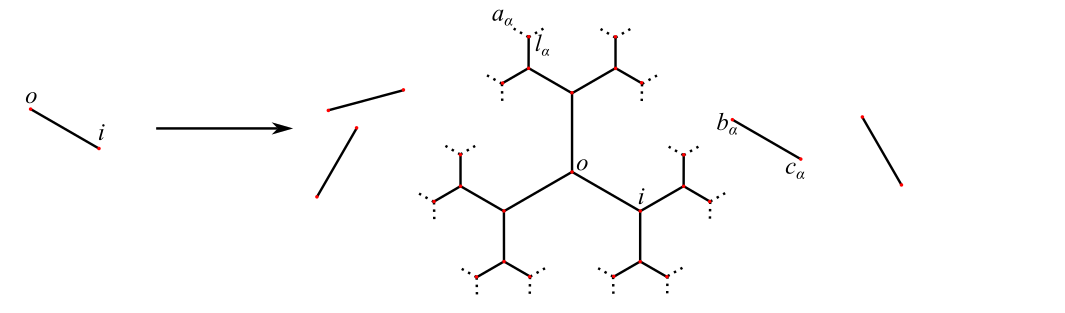}
\end{center}
\caption{Here we give an example of the growth of the tree $\cF$ to $\cF^+$. We take the edge $(i,o)$, then extend it into a depth $\ell$ tree $\cT_\ell(o)$. We then add $\mu$ boundary edges of the form $(l_\alpha,a_\alpha)$ at the boundary of the tree, and $\mu$ new directed switching edges, denoted $e_\alpha:=(b_\alpha, c_\alpha)$ for $\al\in\qq{\mu}$.
\label{fig:construct}}
\end{figure}

The following two propositions state that to compute the expectation of a generic term $R_\bfi\in \cE_r^{\cF}$, we can perform a local resampling around some unused switching edge $(i,o)\in \cC^\circ$. Then we arrive at a weighted sum of the expectation of generic terms $R_{\bfi^+}\in \cE_{r'}^{\cF}$ over some forests $\cF^+=(\bfi^+, E^+)$ (as described in \eqref{e:cFtocF+}). These new terms either have an extra factor $(d-1)^{-\ell/2}$, or $r'\geq r+1$. Then after a finite number of steps this will mean that all the terms are negligible. 

\begin{proposition}\label{p:add_indicator_function}
    Given a forest $\cF=(\bfi, E)$ as in \eqref{e:cFtocF+} and a generic term $R_\bfi=(G_{oo}^{(i)}-Y_\ell(Q))R'_\bfi$. We perform an local resampling around  $(i,o)\in \cF$ with resampling data $\bf S$, and denote $\wt \cG=T_\bfS(\cG)$. Then 
    \begin{align}\begin{split}\label{e:maint}
    &\phantom{{}={}}\frac{1}{Z_\cF}\sum_{\bfi}\bE\left[I(\cF,\cG)\bm1(\cG\in \Omega) R_\bfi(\cG)\right]\\
    &=\frac{1}{Z_\cF}\sum_{\bfi}\bE\left[I(\cF,\cG)I(\cF,\wt\cG)\bm1(\cG,\tcG\in \Omega)(\widetilde G_{oo}^{(i)}-Y_\ell(Q)) R'_\bfi(\wt\cG)\right]+\OO(\bE[\Psi_p]).
\end{split}\end{align}
\end{proposition}

\begin{proposition}\label{p:iteration}
    Given a forest $\cF=(\bfi, E)$ and a generic term $R_\bfi=(G_{oo}^{(i)}-Y_\ell(Q))R'_\bfi\in \cE_r^\cF$. We construct $\cF^+=(\bfi^+, E^+)$ (as in \Cref{s:forest} given by \eqref{e:cFtocF+}) by performing an local resampling around  $(i,o)\in \cF$ with resampling data $\bf S$, and denote $\wt \cG=T_\bfS(\cG)$. Then up to an error of size $\OO(\bE[\Psi_p])$,  
    \begin{align}\label{e:IFIF}
        \frac{1}{Z_\cF}\sum_{\bfi}\bE\left[I(\cF,\cG)I(\cF,\wt\cG)\bm1(\cG,\tcG\in \Omega)(\widetilde G_{oo}^{(i)}-Y_\ell(Q)) R'_\bfi(\wt\cG)\right]
    \end{align}
    can be rewritten as a weighted sum of  terms in the following form
     \begin{align}\label{e:case1}
       \frac{1}{(d-1)^{\ell/2}}\frac{1}{Z_{\cF^+}}\sum_{\bfi^+}  \bE[\bm1(\cG\in \Omega)I(\cF^+, \cG)R_{\bfi^+}(\cG)],\quad R_{\bfi^+}\in \cE_{r}^{\cF^+},
    \end{align}
    or 
    \begin{align}\label{e:case2}
        \frac{1}{Z_{\cF^+}}\sum_{\bfi^+}  \bE[\bm1(\cG\in \Omega)I(\cF^+, \cG)R_{\bfi^+}(\cG)],\quad R_{\bfi^+}\in \cE_{r'}^{\cF^+},\quad r'\geq r+1
    \end{align}
    where the sum of total weights is $\OO(\poly(\ell))$.
   In this way, we either gain an extra factor $(d-1)^{-\ell/2}$, or $r'\geq r+1$ and we gain some extra factor in the form of  \eqref{e:defcE1}. 
\end{proposition}

\begin{proof}[Proof of \Cref{t:recursion}]
We will only prove the bound for $\bE[{\bm1(\cG\in \Omega)}|Q-Y_\ell(Q)|^{2p}]^{1/2p}$, the bound for $\bE[{\bm1(\cG\in \Omega)}|m(z)-X_\ell (Q)|^{2p}]^{1/2p}$ can be proven in the same way. 
Denote $\cF_0=(\bfi_0=(i,o), E=\{(i,o)\})$, we introduce an indicator function $I(\cF_0, \cG)=1$ if any vertex $v\in \cB_\ell(o,\cG)$ has a radius $\fR$ tree neighborhood in $\cG$.  Then we split
\begin{align}\begin{split} \label{e:tt1}
&\phantom{{}={}}\bE[(Q-Y_\ell (Q))^{p}\overline{(Q-Y_\ell (Q))^p}{\bm1(\cG\in \Omega)}]\\
&=\frac{1}{Nd}\bE\left[\sum_{o,i}A_{oi}(G_{oo}^{(i)}-Y_\ell (Q))(Q-Y_\ell (Q))^{p-1}\overline{(Q-Y_\ell (Q))^p}{\bm1(\cG\in \Omega)}\right]\\
&=\frac{1}{Nd}\sum_{\bfi_0}\bE\left[I(\cF_0,\cG)(G_{oo}^{(i)}-Y_\ell (Q))(Q-Y_\ell (Q))^{p-1}\overline{(Q-Y_\ell (Q))^p}{\bm1(\cG\in \Omega)}\right]\\
&+\OO\left(\frac{1}{N^{1-2\fc}}\right)\bE\left[|Q-Y_\ell (Q)|^{2p-1}{\bm1(\cG\in \Omega)}\right],
\end{split}\end{align}
where we used that for $\cG\in \Omega$, $A_{oi}|G_{oo}^{(i)}|,|Y_\ell (Q)|\lesssim 1$ from \Cref{thm:prevthm}, and $I(\cF_0,\cG)=1$ except for $\OO(d(d-1)^\ell N^{\fc})$ vertices from \Cref{def:omegabar}.

We denote 
\begin{align*}
    R_{\bfi_0}=(G_{oo}^{(i)}-Y_\ell (Q))(Q-Y_\ell (Q))^{p-1}\overline{(Q-Y_\ell (Q))^p},
\end{align*}
which satisfies the conditions in \Cref{def:pgen}, thus $R_{\bfi_0}$ is generic with respect to $\cF_0$. With this notation, we can rewrite \eqref{e:tt1} as 
\begin{align*}
   \bE[(Q-Y_\ell (Q))^{p}\overline{(Q-Y_\ell (Q))^p}{\bm1(\cG\in \Omega)}]
   =\frac{1}{Z_{\cF_0}}\sum_{\bfi_0}\bE\left[I(\cF_0,\cG){\bm1(\cG\in \Omega)}R_{\bfi_0}\right]+\OO(\bE[\Psi_p]),
\end{align*}
where $Z_{\cF_0}=Nd$.
The above expression is in the form of \Cref{p:add_indicator_function}, and we can start to use \Cref{p:add_indicator_function} and \Cref{p:iteration} to iterate. Each step of the iteration, we will either get an extra factor $(d-1)^{-\ell/2}\leq N^{-\fc/256}$ as in \eqref{e:case1}; or as in \eqref{e:case2}, we get an extra factor in the form \eqref{e:defcE1}. For $z\in \mathbf D$, these factors are much smaller than $N^{-\fc/256}$, thanks to \Cref{thm:prevthm}. Therefore after a finite number of iterations (at most $256/\fc$), the terms are all bounded by $\OO(\bE[\Psi_p])$. This finishes the proof of \eqref{e:QY}. 
\end{proof}

\subsection{Proof of \Cref{thm:eigrigidity} and \Cref{t:rigidity} }
\label{s:proofrigidity}
In this section, we prove \Cref{thm:eigrigidity} and \Cref{t:rigidity} using  \Cref{cor:QYQbound} as input.
We recall the following estimates from \cite[(6.12), (6.13)]{erdHos2017dynamical} and \eqref{eq:Yprime}, which give that for $z=E+\ri \eta$, $\kappa:=\min\{|E-2|, |E+2|\}$,
\begin{align}\begin{split}\label{eq:mscapprox}
&|1-Y'_\ell(\msc(z))|=|1-\msc(z)^{2\ell+2}|,\quad |1-m_{sc}^2(z)| \asymp \sqrt{\kappa+\eta},\\
&\Im[m_{d}(z)]\asymp\left\{\begin{array}{cc}
\sqrt{\kappa+\eta}&|E|\leq 2,\\
\eta/\sqrt{\kappa+\eta}&|E|>2.
\end{array}
\right. 
\end{split}
\end{align}
Since $|\msc(z)|\leq 1$, we can take $(\fc/128)\log_{d-1} N\leq \ell \leq (\fc/64)\log_{d-1} N$ such that $|1+m_{sc}^2(z)+m_{sc}^4(z)+\cdots+m_{sc}^{2\ell}(z)|\gtrsim 1$.
 The following stability Proposition states that if $|m(z)-X_\ell (Q(z))|$ and $|Q(z)-Y_\ell (Q(z))|$ are small, so are $|Q(z)-\msc(z)|$ and $|m(z)-\md(z)|$.

\begin{proposition}\label{p:stable}
Fix a $z=E+\ri\eta\in \mathbf D$, with $\kappa=\min\{|E-2|, |E+2|\}$, and a $d$-regular graph $\cG\in \Omega(z)$. Assume that 
\begin{align*}
|Q(z)-m(z)|, |m(z)-m_d(z)|\leq \min\left\{\frac{1}{N^\fc}, \frac{\sqrt{\kappa+\eta}}{\log N}\right\},
\end{align*}
and 
there exists a small $0<\delta<1$ so that,
\begin{align*}
|m(z)-X_\ell(Q(z))|, |Q(z)-Y_\ell(Q(z))|=: \delta\leq \frac{\kappa+\eta}{\log N} .
\end{align*}
  Then for $N$ large enough,
\begin{align}\label{e:Q-msc}
\bm1(\cG\in \Omega(z))|Q(z)-\msc(z)|, \bm1(\cG\in \Omega(z))|m(z)-\md(z)|=\OO\left(\frac{\delta}{\sqrt{\kappa+\eta+\delta}}\right),
\end{align}
where the implicit constant is independent of $N$.
\end{proposition}
\begin{proof}[Proof of \Cref{p:stable}]
In the proof, for the simplicity of the notations, we omit the dependence on $z$.   We have from \eqref{e:recurbound} that
\begin{align}\begin{split}\label{e:Q-YQ4}
Q-Y_\ell(Q)&=(Q-m_{sc})-(Y_\ell(Q)-m_{sc})\\
&=(1-m_{sc}^{2\ell+2})(Q-m_{sc})-m_{sc}^{2\ell+3}\left(\frac{1-m_{sc}^{2\ell+2}}{1-m_{sc}^{2}}\right)(Q-m_{sc})^2+O(\ell^2|Q-m_{sc}|^3).
\end{split}
\end{align}
We consider the quadratic equation $ax^2+bx+c=0$, with
\begin{align}\begin{split}\label{e:coeff}
&a=-m_{sc}^{2\ell+3}\left(\frac{1-m_{sc}^{2\ell+2}}{1-m_{sc}^{2}}\right)+\OO\left(\ell^2|Q -m_{sc}|\right),\\
&b=1-m_{sc}^{2\ell+2},
\quad 
c=Y_\ell(Q)-Q, \quad |c|=\delta.
\end{split}\end{align}
We recall that by our choice of $\ell$, it holds $|1+(\msc )^2+(\msc )^4+\cdots+(\msc )^{2\ell}|\gtrsim 1$.
It follows that $|a|\gtrsim |\msc |^{2\ell+3}$. $\sqrt{\kappa+\eta}\lesssim |1-m_{sc} ^2|\lesssim |b|\lesssim \ell|1-m_{sc} ^2|\lesssim \ell \sqrt{\kappa+\eta}$ and $|b/a|=|1-m_{sc}^2 |/|m_{sc} |^{2\ell+3}\gtrsim \sqrt{\kappa+\eta}$.

 By our assumption $\delta\ll \kappa+\eta$, then  $|ac/b^2|\lesssim |c|/(\kappa+\eta)\ll1$.
 The two roots of $ax^2+bx+c=0$ satisfies:$|x_1|=|b/a|+\OO(|c/b|)\asymp |b/a|\gtrsim \sqrt{\kappa+\eta}$ and 
\begin{align}\label{e:cc1}
|x_2|=\OO(|c/b|)\lesssim \frac{|c|}{\sqrt{\kappa+\eta}}\lesssim \frac{\delta}{\sqrt{\kappa+\eta+\delta}}.
\end{align}
Since $|Q-\msc|\ll \sqrt{\kappa+\eta}$, it is necessary that 
\begin{align*}
    |Q-\msc|=|x_2|\lesssim \frac{\delta}{\sqrt{\kappa+\eta+\delta}}.
\end{align*}
The first statement in \eqref{e:Q-msc} follows.
To prove the second statement in \eqref{e:Q-msc}, thanks to Proposition \ref{p:recurbound}, we rewrite $m-m_d$ as
\begin{align}
\begin{split}\label{eq:mdbreakdown}
m-m_d&=m-X_\ell (Q)+X_\ell (Q)-m_d\\
&=m-X_\ell (Q)+\frac{d}{d-1}(\md )^2(m_{sc} )^{2\ell}|Q -m_{sc} |
+\OO\left(\ell|Q -m_{sc} |^2\right).
\end{split}
\end{align}
It follows from plugging \eqref{e:Q-msc} into \eqref{eq:mdbreakdown}, we get
\begin{align*}
   \bm1(\cG\in \Omega)|m-m_d|\lesssim \frac{\delta}{\sqrt{\kappa+\eta+\delta}}.
\end{align*}
\end{proof}

\begin{lemma}\label{p:iterate}
Fix a $z=E+\ri\eta\in \mathbf D$, with $\kappa=\min\{|E-2|, |E+2|\}$. If there exists some deterministic control parameter $\Lambda(z)$ such that the prior estimate holds
\begin{align}\label{e:defLa}
\bm1(\cG\in \Omega(z))|m(z)-m_d(z)|,\quad \bm1(\cG\in \Omega(z))|Q(z)-   \msc(z)|\prec \Lambda(z)\leq \frac{\sqrt{\kappa+\eta}}{N^{2\fc}}.
\end{align}
Then \begin{enumerate}
\item If $|E|\leq 2$, the following holds
\begin{align}\label{e:bb1}
\bm1(\cG\in \Omega(z))|Q(z)-\msc(z)|, \bm1(\cG\in \Omega(z))|m(z)-m_d(z)|\prec  \frac{N^{2\fc}}{N \eta}.
\end{align}
\item If $|E|\geq 2$, the following holds
\begin{align}\label{e:bb2}
\bm1(\cG\in \Omega(z))|Q(z)-\msc(z)|, \bm1(\cG\in \Omega(z))|m(z)-m_d(z)|\prec N^{2\fc}\left(\frac{1}{N\sqrt{\eta(\kappa+\eta)}}+\frac{1}{(N\eta)^2\sqrt{\kappa+\eta}}\right).
\end{align}
\end{enumerate}
\end{lemma}

\begin{proof}

We recall the moment bounds of $Q-Y_\ell(Q)$ amd $m-X_\ell(Q)$ from \Cref{cor:QYQbound},
\begin{align}\begin{split}\label{e:Q-YQ}
&\bE[\bm1(\GG\in \Omega)|Q-Y_\ell(Q)|^{2p}]
\lesssim N^{4p\fc}\bE\left[\bm1(\GG\in \Omega)\left(\frac{\sqrt{\Im[m]|1-Y'_\ell(Q)|}}{N\eta}+\frac{\Im[m]}{N\eta}+|Q-\msc|^2\right)^{2p}\right],\\
&\bE[\bm1(\GG\in \Omega)|m-X_\ell (Q)|^{2p}]
\lesssim N^{4p\fc}\bE\left[\bm1(\GG\in \Omega)\left(\frac{\sqrt{\Im[m]|1-Y'_\ell(Q)|}}{N\eta}+\frac{\Im[m]}{N\eta}+|Q-\msc|^2\right)^{2p}\right].
\end{split}\end{align}

By our assumption \eqref{e:defLa}, we recall that condition on $\cG\in \Omega$
\begin{align}\begin{split}\label{e:Immz}
&\Im[m(  z)]\lesssim \Im[\md]+|m- m_d|\prec\Im[\md] +\Lambda ,\\
&|1-Y'_\ell(Q)|
=|1-Y'_\ell(\msc)|+\OO(|Q-\msc|)
\prec\sqrt{\kappa+\eta}+\Lambda.
\end{split}\end{align}
Plugging \eqref{e:defLa} into the first statement in \eqref{e:Q-YQ}, we get
\begin{align}\begin{split}\label{e:Q-YQ2}
&\phantom{{}={}}\bE[\bm1(\GG\in \Omega)|Q-Y_\ell(Q)|^{2p}]\\
&\prec N^{4p\fc}\bE\left[\bm1(\GG\in \Omega)\left(\frac{\sqrt{(\Im[m_d]+\Lambda)(|1-Y'_\ell(\msc)|+\Lambda)}}{N\eta}+\frac{\Im[m_d]+\Lambda}{N\eta}+\Lambda^2\right)^{2p}\right],
\end{split}\end{align}
Thus the Markov's inequality leads to the following
\begin{align}\begin{split}\label{e:Q-YQ3}
    \bm1(\GG\in \Omega)|Q-Y_\ell(Q)|
    &\prec N^{2\fc}\left(\frac{\sqrt{(\Im[m_d]+\Lambda)(|1-Y'_\ell(\msc)|+\Lambda)}}{N\eta}+\frac{\Im[m_d]+\Lambda}{N\eta}+\Lambda^2\right)\\
    &\prec N^{2\fc}\left(\frac{\sqrt{\Im[m_d]\sqrt{\kappa+\eta}}+\sqrt{\Lambda\sqrt{\kappa+\eta}}}{N\eta}+\Lambda^2\right),
\end{split}\end{align}
where in the last line we used \eqref{eq:mscapprox} that $\Im[m_d]\lesssim \sqrt{\kappa+\eta}$, $|1-Y'(\msc)|\prec \ell \sqrt{\kappa+\eta}$, and $\Lambda\prec \sqrt{\kappa+\eta}$.
Similarly it follows from the second statement in \eqref{e:Q-YQ} that
\begin{align}\label{e:m-XQ2}
     \bm1(\GG\in \Omega)|m-X_\ell (Q)|
    \prec N^{2\fc}\left(\frac{\sqrt{\Im[m_d]\sqrt{\kappa+\eta}}+\sqrt{\Lambda\sqrt{\kappa+\eta}}}{N\eta}+\Lambda^2\right).
\end{align}

Using \eqref{e:Q-YQ3} and \eqref{e:m-XQ2} as input, \Cref{p:stable} implies that
\begin{align}\label{e:newbound}
     \bm1(\GG\in \Omega)|Q-\msc|,  \bm1(\GG\in \Omega)|m-m_d|\prec N^{2\fc}\left(\frac{\sqrt{\Im[m_d]\sqrt{\kappa+\eta}}+\sqrt{\Lambda\sqrt{\kappa+\eta}}}{N\eta\sqrt{\kappa+\eta}}+\frac{\Lambda^2}{\sqrt{\kappa+\eta}}\right).
\end{align}

If $|E|\geq 2$, then $\Im[m_d]\asymp\eta/ \sqrt{|\kappa|+\eta}$, and \eqref{e:newbound} simplifies to
\begin{align}\label{e:outS0}
  \bm1(\GG\in \Omega)|Q-\msc|,  \bm1(\GG\in \Omega)|m-m_d|\prec N^{2\fc}\left(\frac{1}{N\sqrt{\eta (\kappa+\eta)}}+\frac{\sqrt{\Lambda}}{N\eta(\kappa+\eta)^{1/4}}+\frac{\Lambda^2}{\sqrt{\kappa+\eta}}\right).
\end{align}
We can take the new $\Lambda$ as the righthand side of \eqref{e:outS0}, by iterating \eqref{e:outS0}, we get
\begin{align}\label{e:outS}
 \bm1(\GG\in \Omega)|Q-\msc|,  \bm1(\GG\in \Omega)|m-m_d|\prec  \frac{N^{2\fc}}{\sqrt{\kappa+\eta}}\left(\frac{1}{N\sqrt{\eta}}+\frac{1}{(N\eta)^2}\right).
\end{align}
This finishes the proof of \eqref{e:bb1}.

If $|E|\leq 2$, then $\Im[m_d]\asymp\sqrt{\kappa+\eta}$ and $\Lambda(w)\prec \sqrt{|\kappa|+\eta}$, \eqref{e:newbound} simplifies to
\begin{align}\label{e:boundinside}
  \bm1(\GG\in \Omega)|Q-\msc|,  \bm1(\GG\in \Omega)|m-m_d|\prec N^{2\fc}\left(\frac{1}{N\eta}+\frac{\Lambda^2}{\sqrt{\kappa+\eta}}\right).
\end{align}
Again, we can take the new $\Lambda$ as the righthand side of \eqref{e:boundinside}, by iterating \eqref{e:boundinside}, we get
\begin{align}\label{e:boundinside2}
 \bm1(\GG\in \Omega)|Q-\msc|,  \bm1(\GG\in \Omega)|m-m_d|\prec  \frac{N^{2\fc}}{N \eta}.
\end{align}
This finishes the proof of \eqref{e:bb2}.

\end{proof}

\begin{proof}[Proof of  \Cref{t:rigidity}]
We take a lattice grid
\begin{align*}
    \mathbf L=\{E+\ri\eta\in \mathbf D: E\in N^{-2}\bZ,\eta\in N^{-2}\bZ \}.
\end{align*}
For $\cG\in \oOmega$ and $z\in \mathbf D$, $m(z), Q(z), \msc(z), \md(z)$ are all Lipschitz with Lipschitz constant at most $\OO(N)$.
Thus if we can show that for any $z\in \mathbf L$ 
\begin{align}\label{e:mbondcopy}
\bm1(\cG\in \oOmega)|Q(z)-\msc(z)|, \bm1(\cG\in \oOmega)|m(z)-m_d(z)|\prec 
\left\{\begin{array}{cc}
 \frac{N^{2\fc}}{N\eta}, & -2\leq E\leq 2,\\
 \frac{N^{2\fc}}{\sqrt{\kappa+\eta}}\left(\frac{1}{N\eta^{1/2}}+\frac{1}{(N\eta)^2}\right), & |E|\geq 2.
\end{array}
\right.
\end{align}
 then \Cref{t:rigidity} follows.

First for $z\in \mathbf L$ with $\Im[z]\geq 1$, \Cref{thm:prevthm} implies that
\begin{align*}
   \bm1(\cG\in \oOmega)|m(z)-m_d(z)|,\quad \bm1(\cG\in \oOmega)|Q(z)-   \msc(z)|\prec \frac{1}{N^{2\fc}}\leq \frac{\sqrt{\kappa+\eta}}{N^{2\fc}}.
\end{align*}
This verifies the assumption \eqref{e:defLa} in \Cref{p:iterate}. We conclude that \eqref{e:mbondcopy} holds for $z\in \mathbf L$ with $\Im[z]\geq 1$. Then inductively we can show that if \eqref{e:mbondcopy} holds for $z\in \mathbf L$ with $\Im[z]\geq (k+1)/N^2$, then it also holds for $z\in \mathbf L$ with $\Im[z]\geq k/N^2$. More precisely, for $z=E+\ri \eta\in \mathbf L$ with $\Im[z]\geq k/N^2$, we denote $z'=z+(\ri/N^2)\in \mathbf L$. By the inductive hypothesis, and that the Lipschitz constants of $m(z), \md(z)$ are at most $\OO(N)$ 
\begin{align*}
    \bm1(\cG\in \oOmega)|m(z)-m_d(z)|
    =\bm1(\cG\in \oOmega)|m(z')-m_d(z')| +\OO(1/N)
    \prec \frac{N^{2\fc}}{N\eta}\leq \frac{\sqrt{\kappa+\eta}}{N^{2\fc}}.
\end{align*}
\Cref{p:iterate} implies that \eqref{e:mbondcopy} holds for $z$. This finishes the induction and \Cref{t:rigidity}.

\end{proof}

\begin{proof}[Proof of  \Cref{thm:eigrigidity}]
The optimal rigidity \eqref{e:optimal_rigidity} follows from a standard argument using the Stieltjes transform \eqref{e:mbond} estimate as an input, see \cite[Section 11]{erdHos2017dynamical}.  For didactic purposes, we show here how to show that $\lambda_2\leq 2+N^{-2/3+2\fa}$ holds with overwhelmingly high probability on $\overline{\Omega}$. It has been proven in {\cite[Theorem 1.3]{huang2024spectrum}} that with high probability on $\overline{\Omega}$, it holds \begin{align}\label{e:la2bound}
    \lambda_2\leq 2+N^{-\oo(1)}.
\end{align}
In the following we show that with high probability, there is no eigenvalue on the interval $[2+N^{-2/3+2\fa}, 2+\fa]$. It, together with \eqref{e:la2bound}, implies that $\lambda_2\leq 2+N^{-2/3+2\fa}$.

We take $z=2+\kappa+\ri\eta$, with $ N^{-2/3+2\fa}\leq\kappa\leq \fa$ and $\eta=N^{\fa}/(N\sqrt\kappa)\ll \kappa$. With this choice, one can check that $N\eta\sqrt{\kappa}=N^\fa$ and $z\in {\mathbf  D}$, and $\Im[\md(z)]\asymp \eta/\sqrt{\kappa+\eta}\ll 1/{N\eta}$.
In  this regime,   \Cref{t:rigidity} implies that
\begin{align*} \begin{split}
\left|m(z) - \md(z)\right|
&\leq N^{3\fc}\left(\frac{1}{N\sqrt{\eta(\kappa+\eta)}}+\frac{1}{(N\eta)^2\sqrt{\kappa+\eta}}\right)\lesssim \frac{N^{3\fc}}{N^{2\fa} (N\eta)}\ll \frac{1}{N\eta}.
\end{split}\end{align*}
with probability $1-\OO(N^{-\fC})$ with respect to the uniform measure on $\oOmega$, provided $\fa\geq 4\fc$. Thus it follows that 
\begin{align*}
    \Im[m(z)]\leq |m(z)+\md(z)|+\Im[\md(z)]\ll \frac{1}{N\eta}.
\end{align*}
If there is an eigenvalue on the interval $\lambda_i\in [2+\kappa-\eta, 2+\kappa+\eta]$, then 
\begin{align*}
    \Im[m(z)]
    =\frac1N\sum_{i=1}^N \Im\left[\frac{1}{\lambda_i-z}\right]\geq\frac{1}{N}\frac{\eta}{|\lambda_i-(2+\kappa)|^2+\eta^2}\geq \frac{1}{2N\eta},
\end{align*}
which is impossible. It follows that with probability $1-\OO(N^{-\fC})$ with respect to the uniform measure on $\oOmega$, $\lambda_2\leq 2+N^{-2/3+2\fa}$.
\end{proof}

\section{Expansions of Green's Function Differences}\label{sec:expansions}
In this section, we gather estimates on the difference in Green's functions before and after local resampling. For Green's functions related to the center of the local resampling, we employ Schur complement formulas, with results detailed in \Cref{sec:schurlemmas}. This methodology has been previously utilized in similar contexts \cite{huang2024spectrum,bauerschmidt2019local} to establish the local law of random $d$-regular graphs.

For Green's function terms away from the center of the local resampling, we develop a novel expansion using the Woodbury identity, as stated in \Cref{lem:woodbury}. This expansion represents a reorganization of the resolvent identity. In prior research \cite{huang2023edge}, resolvent identities played a pivotal role in analyzing the changes induced by simple switching in Green's functions, yielding an expansion where the terms exhibit exponential decay in $1/\sqrt d$. This decay rate proves adequate when $d$ scales with the size of the graph; however, in our scenario, where $d$ remains fixed, the decay is too slow. Notably, in the new expansion introduced in \Cref{lem:woodbury}, the terms decay exponentially at a rate of $1/N^{\oo(1)}$. 

\subsection{Switching using 
the Schur complement}\label{sec:schurlemmas}
In this section we will use the Schur complement formula to study the Green's function after local resampling. 
We recall the local resampling and related notations from \Cref{s:local_resampling}, and the spectral domain ${\mathbf D}$ from \eqref{e:D0}.

\begin{lemma}\label{lem:diaglem}
Fix $z\in {\mathbf D}$, a $d$-regular graph $\cG\in \Omega(z)$, and an edge $(i,o)\in \cG$. 
We denote $\cT=\cB_\ell(o, \cG)$, $T_{\bf S}$ the local resampling  with resampling data ${\bf S}=\{(l_\al, a_\al), (b_\al, c_\al)\}_{\al\in\qq{\mu}}$ around  $o$, and $\widetilde G=T_{\bf S}G$. Condition on that $\tcG\in \Omega(z)$ and $I(\{(i,o)\}\cup \{(b_\al, c_\al)\}_{\al\in\qq{\mu}},\cG)=1$ (recall from \Cref{def:indicator}), the following holds 
\begin{enumerate}
    \item 
$\widetilde G_{oo}^{(i)}-Y_\ell(Q)$ can be rewritten as a weighted sum  
\begin{align}\label{e:G-Y}
   \widetilde G_{oo}^{(i)}-Y_\ell(Q)=\frac{\fc_1} {(d-1)^\ell}\sum_{\al\in\sfA_i}(G^{(b_\alpha)}_{c_\alpha c_\alpha}-Q)
   +\frac{\fc_2} {(d-1)^\ell}\sum_{\al\neq \beta\in\sfA_i} G^{(b_\alpha b_\beta)}_{c_\alpha c_\beta}+\sum_{\bm\al}\frac{\fc_{\bm\al} U_{\bm\al}}{(d-1)^{2\ell}} +\cE.
\end{align}

\item For $s,s'\in\{o,i\}$, $\widetilde G_{ss'}-P_{ss'}$ can be rewritten as a weighted sum, 
\begin{align*}\begin{split}
    \widetilde G_{ss'}-P_{ss'}
    &=\frac{\fc_1} {(d-1)^\ell}\sum_{\al\in\sfA_i}(G^{(b_\alpha)}_{c_\alpha c_\alpha}-Q)
    +\frac{\fc_2} {(d-1)^\ell}\sum_{\al\notin\sfA_i}(G^{(b_\alpha)}_{c_\alpha c_\alpha}-Q)
    \\
   &+\frac{\fc_3} {(d-1)^\ell}\sum_{\al\neq \beta\in \sfA_i} G^{(b_\alpha b_\beta)}_{c_\alpha c_\beta}
   +\frac{\fc_4} {(d-1)^\ell}\sum_{\al\neq \beta\in\sfA^\complement_i} G^{(b_\alpha b_\beta)}_{c_\alpha c_\beta}
   \\
   &+\frac{\fc_5} {(d-1)^\ell}\sum_{\al\in\sfA_i\atop \beta\in \sfA^\complement_i} G^{(b_\alpha b_\beta)}_{c_\alpha c_\beta}+\fc_6(Q-\msc)+\sum_{\bm\al}\frac{\fc_{\bm\al} U_{\bm\al}}{(d-1)^{2\ell}} +\cE.
\end{split}\end{align*}
\end{enumerate}
where $\sfA_i:=\{\al\in \qq{1, \mu}: \dist_{\cT}(i, l_\al)=\ell+1\}$ and $\sfA^\complement_i=\qq{\mu}\setminus \sfA_i$; the summation is over indices $\bm\al=(\al_1, \al_2,\cdots, \al_{2k})\in \qq{\mu}^{2k}$ for some $2\leq k\leq \fb$;
\begin{align*}
    U_{\bm\al}=\widehat U_{{\al_1}{\al_2}} \widehat U_{{\al_3}{\al_4}}
   \cdots \widehat U_{{\al_{2k-1}}{\al_{2k}}},
\end{align*}
where $\widehat U_{\al \al}=(d-1)^{2\ell}(G_{c_\al c_\al}^{(b_\al)}-Q)$ or $(d-1)^{2\ell}(Q-\msc)$; and $\widehat U_{\al \beta}=(d-1)^{2\ell} G_{c_\al c_\beta}^{(b_\al b_\beta)}$ if $\al\neq \beta$; and the total weights are bounded by $\OO(\poly(\ell))$: 
\begin{align*}
    \sum_j |\fc_j|+\sum_{\bm\al}|\fc_{\bm\al}|=\OO( \poly(\ell)).
\end{align*}
The error $\cE$ satisfies
\begin{align}\label{e:defCE}
 |\cE|\lesssim \sum_{\al} |\wt G^{(\bT)}_{c_\al c_\al}-G^{(b_\al)}_{c_\al c_\al}|+\sum_{\al\neq \beta} |\wt G^{(\bT)}_{c_\al c_\beta}-G^{(b_\al b_\beta)}_{c_\al c_\beta}| +\frac{1}{N^2}.
\end{align}
\end{lemma}

\begin{proof}[Proof of Lemma \ref{lem:diaglem}]
We denote $\cT=\cB_\ell(o,\cG)$ the radius $\ell$ neighborhood of vertex $o$, and its vertex set $\bT$. 
We denote the normalized adjacency matrix of $\cG^{(i)}$ as $H^{(i)}$, and its restriction to $\bT$ as $H^{(i)}_\bT$. Let $\wt B$ be the normalized adjacency matrix of the directed edges $\{(c_\al, l_\al)\}_{\al \in \qq{\mu}}$. We also denote the Green's function of $\cG^{(\bT)}$ and $\wt\cG^{(\bT)}$ as $ G^{(\bT)}$ and $\wt G^{(\bT)}$ respectively.  

Thanks to the Schur complement formula, we have 
\begin{align*}
\widetilde G_{oo}^{(i)}=\left(\frac{1}{H_{\bT}^{(i)}-z-\widetilde B^\top\widetilde G^{(\bT)}\widetilde B}\right)_{oo}.
\end{align*}
Since $o$ has a radius $\fR/2$ tree neighborhood, $H^{(i)}_\bT$ is the normalized adjacency matrix of a truncated $(d-1)$-ary tree, 
\begin{align*}
    \msc=\left(\frac{1}{H_{\bT}^{(i)}-z-\widetilde B^\top \msc{\mathbb I}\widetilde B}\right)_{oo}, \quad P^{(i)}:=P(\cT^{(i)}, z, \msc)=\frac{1}{H_{\bT}^{(i)}-z-\widetilde B^\top \msc{\mathbb I}\widetilde B}.
\end{align*}
By taking the difference of the two above expressions, we have
\begin{equation*}
\widetilde G_{oo}^{(i)}-\msc=\left(\frac{1}{H_{\bT}^{(i)}-z-\widetilde B^\top\widetilde G^{(\bT)}\widetilde B}-\frac{1}{H_{\bT}^{(i)}-z-\widetilde B^\top \msc{\mathbb I}\widetilde B}\right)_{oo}.
\end{equation*}
 Then we expand using the resolvent identity. By our assumption $\tcG\in \Omega$, and thus $|\tG^{(\bT)}_{c_\al c_\al}- \msc|\leq N^{-\fc}$. Thus for some sufficiently large constant $\fb$, the following holds
\begin{align}\begin{split}\label{eq:resolventexp}
    &\phantom{{}={}}\left(\frac{1}{H_{\bT}^{(i)}-z-\widetilde B^\top\widetilde G^{(\bT)}B}-\frac{1}{H_{\bT}^{(i)}-z- \widetilde B^\top \msc{\mathbb I}B}\right)_{oo}\\
    &=\left(P^{(i)}\sum_{k=1}^\fb \left(\widetilde B^\top (\widetilde G^{(\bT)}-\msc\mathbb I)\widetilde B P^{(i)}\right)^k\right)_{oo}+\OO(N^{-2})\\
    &=\left(P^{(i)}\sum_{k=1}^\fb \left(\widetilde B^\top (\widetilde G^{(\bT)}-Q\mathbb I+(Q-\msc)\mathbb I)\widetilde B P^{(i)}\right)^k\right)_{oo}+\OO(N^{-2}).
\end{split}
\end{align}

Recall $Y_\ell(Q)$ from \eqref{def:Y}. By the same argument we also have that
\begin{align}\label{e:YQ-m}
    Y_\ell(Q)-\msc=\left(P^{(i)}\sum_{k=1}^\fb \left(\widetilde B^\top (Q-\msc)\mathbb I\widetilde B P^{(i)}\right)^k\right)_{oo}+\OO(N^{-2}).
\end{align}

By taking the difference of \eqref{eq:resolventexp} and \eqref{e:YQ-m}, up to an error $\OO(N^{-2})$, we get that the difference $\widetilde G_{oo}^{(i)}-Y_\ell(Q)$ is given as
\begin{align}\label{e:GooY}
    \left(P^{(i)}\sum_{k=1}^\fb \left(\widetilde B^\top (\widetilde G^{(\bT)}-Q\mathbb I+(Q-\msc)\mathbb I)\widetilde B P^{(i)}\right)^k\right)_{oo}-\left(P^{(i)}\sum_{k=1}^\fb \left(\widetilde B^\top (Q-\msc)\mathbb I\widetilde B P^{(i)}\right)^k\right)_{oo}.
\end{align}

By \eqref{e:Pijbound}, we see the entries of $P^{(i)}$ decay exponentially, and in particular we have that if $\al \in \sfA_i$, then $\dist_\cT(i, l_\al)=\ell+1$, and
\begin{align*}
    |P^{(i)}_{ol_\al}|\lesssim (d-1)^{-\dist_{\cT}(o,\ell)/2}=(d-1)^{-\ell/2};
\end{align*}
otherwise if $\al \in\sfA_i^\complement$, then $o,l_\al$ are in different connected components of $\cT^{(i)}$, and $|P^{(i)}_{ol_\al}|=0$.
Thus the $k=1$ term in \eqref{e:GooY}, is given by  
\begin{align}\label{e:G-Y2}
    \frac{\fc_1} {(d-1)^\ell}\sum_{\al\in\sfA_i}(\wt G^{(\bT)}_{c_\alpha c_\alpha}-Q)
   +\frac{\fc_2} {(d-1)^\ell}\sum_{\al\neq \beta\in\sfA_i} \wt G^{(\bT)}_{c_\alpha c_\beta},
\end{align}
where the two coefficients $|\fc_1|,|\fc_2|=\OO(1)$.

We obtain the first two terms in \eqref{e:G-Y}, after replacing $\widetilde G_{c_\al c_\al}^{(\bT)},\widetilde G_{c_\al c_\beta}^{(\bT)}$ in \eqref{e:G-Y2} by $G_{c_\al c_\al}^{(b_\al)}, G_{c_\al c_\beta}^{(b_\al b_\beta)}$. We collect the difference in the error term $\cE$ (as in \eqref{e:defCE}).

For the terms $k\geq 2$, in general 
$P^{(i)}(\widetilde B^\top V \widetilde B P^{(i)})^k$ is given as a sum of terms in the following form
\begin{align}\label{e:PUP}
   \sum_{\al_1, \al_2,\cdots,\al_{2k}\in \qq{\mu}}P^{(i)}_{o l_{\al_1}} V_{c_{\al_1} c_{\al_2}} P^{(i)}_{l_{\al_2} l_{\al_3}} V_{c_{\al_3} c_{\al_4}} 
   P^{(i)}_{l_{\al_4} l_{\al_5}}\cdots V_{c_{\al_{2k-1} }c_{\al_{2k}}}P^{(i)}_{l_{\al_{2k}} o}.
\end{align}

Recall from \eqref{e:Pijbound}, the contribution from $P^{(i)}$ terms in \eqref{e:PUP} depends only on the distances between $l_{\al_{2i}}$ and $l_{\al_{2i+1}}$
\begin{align*}
    f(\bm\al):=P^{(i)}_{o l_{\al_1}} P^{(i)}_{l_{\al_2} l_{\al_3}} 
   \cdots P^{(i)}_{l_{\al_{2k}} o},\quad |f(\bm\al)|\lesssim (d-1)^{-\ell}\prod_{j=1}^{k-1} (d-1)^{-\dist_{\cT}(l_{\alpha_{2j}},l_{\alpha_{2j+1}})/2}.
\end{align*}
We can reorganize \eqref{e:PUP} in the following way according to the equivalent class defined above
\begin{align}\begin{split}\label{e:totalsum0}
    &\phantom{{}={}}\sum_{\bm\al}f(\bm\al)V_{c_{\al_1}c_{\al_2}} V_{c_{\al_3}c_{\al_4}}
   \cdots V_{c_{\al_{2k-1}}c_{\al_{2k}}}\\
   &=\sum_{\bm\al}f(\bm\al) (d-1)^{-2k\ell}((d-1)^{2\ell}V_{c_{\al_1}c_{\al_2}}) ((d-1)^{2\ell}V_{c_{\al_3}c_{\al_4}})
   \cdots ((d-1)^{2\ell} V_{c_{\al_{2k-1}}c_{\al_{2k}}})\\
   &=:(d-1)^{-2\ell}\sum_{\bm\al}\fc_{\bm\al}\widehat U_{{\al_1}{\al_2}}\widehat U_{{\al_3}{\al_4}}
   \cdots  \widehat U_{{\al_{2k-1}}{\al_{2k}}},
\end{split}\end{align}
where $\fc_{\bm\al}=f(\bm\al) (d-1)^{-2(k-1)\ell}$ and $\widehat U_{\al_{2j-1}\al_{2j}}:=(d-1)^{2\ell} V_{c_{\al_{2j-1}}c_{\al_{2j}}}$ is one of the three terms $(d-1)^{2\ell}(G_{c_\al c_\al}^{(b_\al)}-Q)$, $(d-1)^{2\ell}(Q-\msc)$, or $(d-1)^{2\ell} G_{c_\al c_\beta}^{(b_\al b_\beta)}$.
Finally for the summation over the weights:
\begin{align*}
   \sum_{\bm\al}\fc_{\bm\al}
   &= \sum_{\bm\al}|f(\bm\al)| (d-1)^{-2(k-1)\ell}\lesssim \sum_{\bm\al}(d-1)^{-(k+1)\ell}\prod_{j=1}^{k-1} (d-1)^{-\dist_{\cT}(l_{\alpha_{2j}},l_{\alpha_{2j+1}})/2}\\
   &=\sum_{\al_1, \al_{2k}}(d-1)^{2\ell}
   \prod_{j=1}^{k-1}\sum_{\al_{2j}, \al_{2j+1}}(d-1)^{-\ell}(d-1)^{-\dist_{\cT}(l_{\alpha_{2j}},l_{\alpha_{2j+1}})/2}= \OO(\poly(\ell)). 
\end{align*}
where in the last inequality, we used that the total number of choices for $\al_1, \al_{2k}$ are $\OO((d-1)^{2\ell})$; and the following estimates
\begin{align}\label{e:sumal_beta}
    \sum_{\al} (d-1)^{-\dist_{\cT}(l_\al, l_\beta)/2}=\OO(\poly(\ell)),\quad \sum_{\al, \beta}(d-1)^{-\ell} (d-1)^{-\dist_{\cT}(l_\al, l_\beta)/2}=\OO(\poly(\ell)),
\end{align}
which is from the simple observation that given the index $\al$, and the distance $r=\dist_\cT(l_\al, l_\beta)$, the number of choices for $\beta$ is of order $\OO((d-1)^{r/2})$. After replacing $\widetilde G_{c_\al c_\al}^{(\bT)},\widetilde G_{c_\al c_\beta}^{(\bT)}$ in \eqref{e:G-Y2} by $G_{c_\al c_\al}^{(b_\al)}, G_{c_\al c_\beta}^{(b_\al b_\beta)}$, this finishes the first statement in \Cref{lem:diaglem}.

Now, we generalize this argument to a term of the form $\wt G_{ss'}-P_{ss'}$ with $s,s'\in\{o,i\}$. We do the exact same expansion as before using the Schur complement formula, 
\begin{align}
\begin{split}\label{eq:XQdecomp}
&\phantom{{}={}}\wt G_{s_1s_2}-P_{s_1s_2}=\left(P\sum_{k=1}^\fb \left(\widetilde B^\top (\widetilde G^{(\bT)}-\msc\mathbb I)\widetilde B P\right)^k\right)_{ss'}+\OO(\Phi).
\end{split}\end{align}
Using the above expansion, the statement follows from the same argument as the first statement.

\end{proof}

We derive a similar expansion for factors that are Green's function entries with at most one indices in $\{o,i\}$.
\begin{lemma}\label{lem:offdiagswitch}
Fix $z\in {\mathbf D}$, a $d$-regular graph $\cG\in \Omega(z)$, and edges $(i,o), (b,c), (b',c')\in \cG$.
We denote $\cT=\cB_\ell(o, \cG)$, $T_{\bf S}$ the local resampling  with resampling data ${\bf S}=\{(l_\al, a_\al), (b_\al, c_\al)\}_{\al\in\qq{\mu}}$ around  $o$, and $\widetilde G=T_{\bf S}G$. Condition on that $\tcG\in \Omega(z)$ and $I(\{(i,o), (b,c), (b',c')\}\cup \{(b_\al, c_\al)\}_{\al\in\qq{\mu}},\cG)=1$ (recall from \Cref{def:indicator}), the following holds  
\begin{enumerate}
    \item  $\widetilde G^{(ib)}_{oc}$ can be rewritten as  a weighted sum 
\begin{align*}
   \widetilde G^{(ib)}_{oc}=\frac{\fc_1} {(d-1)^{\ell/2}}\sum_{\al\in\sfA_i}G_{c_\al c}^{(b_\al b)}
   +\sum_{\bm\al}\fc_{\bm\al}U_{\bm\al} +\cE,
\end{align*}
where $\sfA_i:=\{\al\in \qq{1, \mu}: \dist_{\cT}(i, l_\al)=\ell+1\}$; the summation is over indices $\bm\al=(\al_1, \al_2,\cdots, \al_{2k})\in \qq{\mu}^{2k+1}$ for some $1\leq k\leq \fb$;
\begin{align*}
    U_{\bm\al}=\widehat U_{{\al_1}{\al_2}} \widehat U_{{\al_3}{\al_4}}
   \cdots \widehat U_{{\al_{2k-1}}{\al_{2k}}}G_{c_{\al_{2k+1}}c}^{(b_{\al_{2k+1}}b)},
\end{align*}
where $\widehat U_{\al \al}=(d-1)^{2\ell}(G_{c_\al c_\al}^{(b_\al)}-Y_\ell(Q))$ or $(d-1)^{2\ell}(Q-\msc)$; and $\widehat U_{\al \beta}=(d-1)^{2\ell} G_{c_\al c_\beta}^{(b_\al b_\beta)}$ if $\al\neq \beta$; and the total weights are bounded by $\poly(\ell)$: 
\begin{align*}
    |\fc_1|+\sum_{\bm\al}|\fc_{\bm\al}|=\OO( \poly(\ell)),
\end{align*}
and
\begin{align*}
    |\cE|\lesssim \sum_\al (|\wt G_{c_\al c}^{(\bT b)}-G_{c_\al c}^{(b_\al b)}|+|\wt G_{c_\al c_\al}^{(\bT b)}-G^{(b_\al )}_{c_\al c_\al}|) +\sum_{\al \neq \beta}|\wt G_{c_\al c_\beta}^{(\bT b)}-G^{(b_\al b_\beta)}_{c_\al c_\beta}|+N^{-2}.
\end{align*}
\item $\wt G_{cc}^{(b)}$ can be rewritten as
\begin{align*}
    \wt G_{cc}^{(b)}=G_{cc}^{(b)}+\cE. 
\end{align*}
where 
\begin{align*}
    |\cE|\lesssim |\wt G^{(\bT b)}_{cc}-G^{( b)}_{cc}|
    +(d-1)^\ell\sum_\al |\wt G_{cc_\al}^{(\bT b)}|^2.
\end{align*}
\item  $\wt G_{cc'}^{(bb')}$ can be rewritten as
\begin{align*}
    \wt G_{cc'}^{(bb')}=G_{cc'}^{(bb')}+\cE. 
\end{align*}
where 
\begin{align*}
    |\cE|\lesssim |\wt G^{(\bT bb')}_{cc'}-G^{( bb')}_{cc'}|
    +(d-1)^\ell\sum_\al (|\wt G_{cc_\al}^{(\bT bb')}|^2+|\wt G_{c'c_\al}^{(\bT bb')}|^2).
\end{align*}
\end{enumerate}
\end{lemma}

\begin{proof}
The proof largely emulates the proof of Lemma \ref{lem:diaglem}, with some slight differences. We expand $\wt G_{oc}^{(ib)}$ using the off-diagonal Schur complement formula as
\begin{align*}
\wt G^{(ib)}_{oc}
&=-\frac{1}{\sqrt{d-1}}\sum_{\alpha} \wt G^{(ib)}_{o l_\alpha}\wt G^{(\bT b)}_{c_\alpha c}.
\end{align*}

We can then replace $\wt G_{c_\al c}^{(\bT b)}$ with $G_{c_\al c}^{(b_\al b)}$, as
\begin{align}\begin{split}\label{eq:offdiagonalexp}
\wt G^{(ib)}_{oc}
&=-\frac{1}{\sqrt{d-1}}\sum_{\alpha} \wt  G^{(ib)}_{o l_\alpha}G^{(b_\alpha b)}_{c_\alpha c}+\cE_1\\
&=-\frac{1}{\sqrt{d-1}}\sum_{\alpha} (\wt G^{(ib)}_{o l_\alpha}-P^{(i)}_{ol_\alpha})G^{(b_\alpha b)}_{c_\alpha c}+ P^{(i)}_{ol_\alpha} G^{(b_\alpha b)}_{c_\alpha c}+\cE_1. 
\end{split}\end{align} 
where the error term is
\begin{align*}
    |\cE_1|\lesssim \sum_\al |\wt G_{c_\al c}^{(\bT b)}-G_{c_\al c}^{(b_\al b)}|.
\end{align*}

For the second term on the righthand side of \eqref{eq:offdiagonalexp}, we recall that if $\dist_\cT(i, l_\al)=\ell+1$, $|P^{(i)}_{ol_\al}|\asymp (d-1)^{-\dist_{\cT}(o,\ell)/2}=(d-1)^{-\ell/2}$;
otherwise $|P^{(i)}_{ol_\al}|=0$. Thus we can rewrite it as
\begin{align*}
  \frac{\fc_1}{(d-1)^{\ell/2}}\sum_{\al\in\sfA_i}G^{(b_\alpha b)}_{c_\alpha c},  
\end{align*}
with $\fc_1=\OO(1)$.

For the first term on the righthand side of \eqref{eq:offdiagonalexp}, we do the same expansion as in \eqref{eq:XQdecomp} from the proof of Lemma \ref{lem:diaglem}, writing
\begin{align*}
   \sum_\al (\wt G^{(ib)}_{o l_\alpha}-P^{(i)}_{ol_\alpha})G^{(b_\al b)}_{c_\al c}= \sum_\al \left(P^{(i)}\sum_{k=1}^\fb \left(\widetilde B^\top (\widetilde G^{(\bT b)}-\msc\mathbb I)\widetilde B P^{(i)}\right)^k\right)_{ol_\al}G^{(b_\al b)}_{c_\al c}+\OO(N^{-2}).
\end{align*}

By the same argument as in the proof of \Cref{lem:diaglem}, we can replace $\wt G_{c_\al c_\al}^{(\bT b)}$ and $ \wt G_{c_\al c_\beta}^{(\bT b)}$ by $G^{(b_\al)}_{c_\al c_\al}$ and $G^{(b_\al b_\beta)}_{c_\al c_\beta}$ respectively, and
the terms corresponding to $k\geq 1$ are given as a sum 
\begin{align}\label{eq:PUP2}
   \sum_{\al_1, \al_2,\cdots,\al_{2k+1}\in \qq{\mu}}P^{(i)}_{o l_{\al_1}} V_{c_{\al_1} c_{\al_2}} P^{(i)}_{l_{\al_2} l_{\al_3}} V_{c_{\al_3} c_{\al_4}} 
   P^{(i)}_{l_{\al_4} l_{\al_5}}\cdots V_{c_{\al_{2k-1} }c_{\al_{2k}}}P^{(i)}_{l_{\al_{2k}} l_{\al_{2k+1}}} G_{c_{\al_{2k+1}}c}^{(b_{\al_{k+1}}b)}.
\end{align}
where $ V_{c_{\al_{2j-1}}c_{\al_{2j}}}$ is one of the three terms $(G_{c_\al c_\al}^{(b_\al)}-Q)$, $(Q-\msc)$,  or $ G_{c_\al c_\beta}^{(b_\al b_\beta)}$.

The contribution from $P^{(i)}$ terms in \eqref{e:PUP} depends only on the distances between $l_{\al_{2i}}$ and $l_{\al_{2i+1}}$
\begin{align*}
    f(\bm\al):=P^{(i)}_{o l_{\al_1}} P^{(i)}_{l_{\al_2} l_{\al_3}} 
   \cdots P^{(i)}_{l_{\al_{2k}} l_{\al_{2k+1}}},\quad |f(\bm\al)|\lesssim (d-1)^{-\ell/2}\prod_{j=1}^{k} (d-1)^{-\dist_{\cT}(l_{\alpha_{2j}},l_{\alpha_{2j+1}})/2}.
\end{align*}
We can reorganize \eqref{eq:PUP2} in the following way 
\begin{align*}\begin{split}
    &\phantom{{}={}}\sum_{\bm\al}f(\bm\al)V_{c_{\al_1}c_{\al_2}} V_{c_{\al_3}c_{\al_4}}
   \cdots V_{c_{\al_{2k-1}}c_{\al_{2k}}}G_{c_{\al_{2k+1}}c}^{(b_{\al_{k+1}}b)}\\
   &=\sum_{\bm\al}f(\bm\al) (d-1)^{-2k\ell}((d-1)^{2\ell}V_{c_{\al_1}c_{\al_2}}) ((d-1)^{2\ell}V_{c_{\al_3}c_{\al_4}})
   \cdots ((d-1)^{2\ell} V_{c_{\al_{2k-1}}c_{\al_{2k}}})G_{c_{\al_{2k+1}}c}^{(b_{\al_{k+1}}b)}\\
   &=:\sum_{\bm\al}\fc_{\bm\al}\widehat U_{{\al_1}{\al_2}}\widehat U_{{\al_3}{\al_4}}
   \cdots  \widehat U_{{\al_{2k-1}}{\al_{2k}}}G_{c_{\al_{2k+1}}c}^{(b_{\al_{k+1}}b)},
\end{split}\end{align*}
where $\fc_{\bm\al}=f(\bm\al) (d-1)^{-2k\ell}$ and $\widehat U_{\al_{2j-1}\al_{2j}}:=(d-1)^{2\ell} V_{c_{\al_{2j-1}}c_{\al_{2j}}}$ is one of the three terms $(G_{c_\al c_\al}^{(b_\al)}-Q)$, $(Q-\msc)$,  or $ G_{c_\al c_\beta}^{(b_\al b_\beta)}$.
Finally for the summation over the weights:
\begin{align*}\begin{split}
   \sum_{\bm\al}\fc_{\bm\al}
   &= \sum_{\bm\al}|f(\bm\al)| (d-1)^{-2k\ell}\lesssim \sum_{\bm\al}(d-1)^{-(k+1)\ell}\prod_{j=1}^{k} (d-1)^{-\dist_{\cT}(l_{\alpha_{2j}},l_{\alpha_{2j+1}})/2}\\
   &=\sum_{\al_1}(d-1)^{-\ell}
   \prod_{j=1}^{k}\sum_{\al_{2j}, \al_{2j+1}}(d-1)^{-\ell}(d-1)^{-\dist_{\cT}(l_{\alpha_{2j}},l_{\alpha_{2j+1}})/2}=\OO(\poly(\ell)). 
\end{split}\end{align*}
where in the last inequality, we used that the total number of choices for $\al_1$ is $\OO((d-1)^{\ell})$ and \eqref{e:sumal_beta}. This finishes the first statement in \Cref{lem:offdiagswitch}, and the error term $\cE$ is given by 
\begin{align*}
    |\cE|\lesssim \sum_\al (|\wt G_{c_\al c}^{(\bT b)}-G_{c_\al c}^{(b_\al b)}|+|\wt G_{c_\al c_\al}^{(\bT b)}-G^{(b_\al )}_{c_\al c_\al}|) +\sum_{\al \neq \beta}|\wt G_{c_\al c_\beta}^{(\bT b)}-G^{(b_\al b_\beta)}_{c_\al c_\beta}|+N^{-2}.
\end{align*}

For the second statement, we expand $\wt G_{cc}^{(b)}$ using the Schur complement formula as
\begin{align*}
\wt G^{(b)}_{cc}
=\wt G^{(\bT b)}_{cc}
+\frac{1}{d-1}\sum_{\al,\beta}\wt G^{(\bT b)}_{cc_\al}\widetilde G^{( b)}_{l_\al l_\beta }\wt G^{(\bT b)}_{c_\beta c}
=\wt G^{(b)}_{cc}+\cE,
\end{align*}
where 
\begin{align*}
    |\cE|\lesssim |\wt G^{(\bT b)}_{cc}-G^{( b)}_{cc}|
    +(d-1)^\ell\sum_\al |\wt G_{cc_\al}^{(\bT b)}|^2.
\end{align*}

For the last statement, we expand $\wt G_{cc'}^{(bb')}$ using the Schur complement formula as
\begin{align*}
\wt G^{(bb')}_{cc'}
=\wt G^{(\bT bb')}_{cc'}
+\frac{1}{d-1}\sum_{\al,\beta}\wt G^{(\bT bb')}_{cc_\al}\widetilde G^{( bb')}_{l_\al l_\beta }\wt G^{(\bT bb')}_{c_\beta c'}
=G^{( bb')}_{cc'}+\cE,
\end{align*}
where 
\begin{align*}
    |\cE|\lesssim |\wt G^{(\bT bb')}_{cc'}-G^{( bb')}_{cc'}|
    +(d-1)^\ell\sum_\al (|\wt G_{cc_\al}^{(\bT bb')}|^2+|\wt G_{c'c_\al}^{(\bT bb')}|^2).
\end{align*}

\end{proof}

\subsection{An Expansion using the Woodbury formula}\label{sec:fanalysis}
In this section, we propose a novel expansion using the Woodbury formula. This is an important part of our analysis.
For a graph $\GG$ with Green's function $G$, we consider the Green's function after the switch $\widetilde G:=T_\bfS(G)$ around some vertex $o\in \qq{N}$, and denote its normalized adjacency matrix as $\wt H$.

We compare the normalized adjacency matrix of the switched graph to that of the original graph, $\wt H-H$. We denote the rank of this difference as $r$, and rewrite $\wt H-H=UCU^\top$, where $U$ is an $N\times r$ matrix, and $C$ is $r\times r$. Then, the Woodbury formula gives us 
\begin{align}\label{e:tGG}
\tG-G=(H-z+UCU^\top)^{-1}-(H-z)^{-1}=-GU(C^{-1}+U^\top G U)^{-1}U^\top G.
\end{align}

We denote $ \cF$ the set of edges involved in the local resampling, and $\wt \cF$ the graph after local resampling
\begin{align*}
  \cF:=\cB_{\ell}(o,\GG)\cup\bigcup_{\al=1}^\mu \{(l_\alpha,a_\alpha), (b_\alpha, c_\alpha)\},\quad   \wt \cF:=\cB_{\ell}(o,\GG)\cup\bigcup_{\al=1}^\mu \{(l_\alpha,c_\alpha), (a_\alpha, b_\alpha)\}.
\end{align*}

We will analyze this using the matrix $P:=P(\cF,z,m_{sc})$, as was define in Definition \ref{def:pdef}. Along with comparing $\tG-G$. we can similarly compare $\tP-P$, where $\tP:=P(\widetilde \cF,z,m_{sc})$.

Notice that $\widetilde P^{-1}-P^{-1}=\widetilde H-H$ when restricted to the vertex set $\mathbb F$ of $\cF$, and it is given by $UCU^\top$. We can use the Woodbury formula on $P$ as well, giving
\begin{align}\label{e:tP-P}
    \tP-P=-PU(C^{-1}+U^\top P U)^{-1}U^\top P.
\end{align}

Our next lemma attempts to expand $\tG-G$ in terms of $\tP-P$. We denote the adjacency matrices of our switching as
\begin{align*}
    \xi_\al:=\frac{1}{\sqrt{d-1}}\left(\Delta_{l_\al a_\al}
    +\Delta_{b_\al c_\al}
    -\Delta_{l_\al c_\al}-\Delta_{a_\al b_\al}\right).
\end{align*}
Then 
\begin{align}\label{e:sumxi}
    UCU^\top=\wt H-H=\widetilde P^{-1}-P^{-1}=\sum_{\al }\xi_\al.
\end{align}

We also introduce the following matrix $F$, which is nonzero on the vertex set $\mathbb F$,
\begin{align*}
F:=\sum_{\al \in \qq{\mu}} \xi_{\al}+\sum_{\al, \beta\in \qq{\mu}} \xi_{\al}\tP \xi_{\beta}.
\end{align*}

\begin{lemma}\label{lem:woodbury} 
Fix $z\in {\mathbf D}$, a $d$-regular graph $\cG\in \Omega(z)$, and edges $(i,o)\in \cG$. 
We denote $T_{\bf S}$ the local resampling  with resampling data ${\bf S}=\{(l_\al, a_\al), (b_\al, c_\al)\}_{\al\in\qq{\mu}}$ around  $o$, and $\widetilde G=T_{\bf S}G$. Conditioned on that $\tcG\in \Omega(z)$ and $I(\{(i,o), (b,c), (b',c')\}\cup \{(b_\al, c_\al)\}_{\al\in\qq{\mu}},\cG)=1$ (recall from \Cref{def:indicator}), the following holds  
\begin{align*}
\tG-G=\sum_{k\geq 0} GF((G|_{\mathbb F}-P)F)^{k}G.
\end{align*} 
\end{lemma}

\begin{proof}

The nonzero rows of $U$ are parametrized by $\{l_\alpha, a_\alpha, b_\alpha, c_\alpha\}_{\al\in \qq{\mu}}$.
By rearranging the above expression \eqref{e:tP-P}, we get
\begin{align}\label{e:Fz}
P^{-1}\tP P^{-1}-P^{-1}=-U(C^{-1}+U^\top P U)^{-1}U^\top.
\end{align}
We recall that $\widetilde P^{-1}-P^{-1}=\widetilde H-H=UCU^\top$. 
We can reorganize \eqref{e:Fz} as
\begin{align}\begin{split}\label{e:defF2}
    &\phantom{{}={}}P^{-1}\tP P^{-1}-P^{-1}=P^{-1}\tP \tP^{-1}+P^{-1}\tP (P^{-1}-\tP^{-1})-P^{-1}\\
    &=P^{-1}\tP (P^{-1}-\tP^{-1})
    =(P^{-1}-\tP^{-1})\tP (P^{-1}-\tP^{-1})+\tP^{-1}\tP (P^{-1}-\tP^{-1})\\
     &=(P^{-1}-\tP^{-1})+(P^{-1}-\tP^{-1})\tP (P^{-1}-\tP^{-1}) \\
    &=\sum_{\al \in \qq{\mu}} \xi_{\al}+\sum_{\al, \beta\in \qq{\mu}} \xi_{\al}\tP \xi_{\beta}=F,
\end{split}\end{align}
where in the last line we used \eqref{e:sumxi}.

By plugging \eqref{e:Fz} and \eqref{e:defF2} into \eqref{e:tGG}, and use the resolvent identity, we conclude that
\begin{align*}
\tG-G
&=-GU(C^{-1}+U^\top G U)^{-1}U^\top G
=-GU(C^{-1}+U^\top P U+U^\top (G-P) U)^{-1}U^\top G\\
&
=-GU\left((C^{-1}+U^\top P U)^{-1}\sum_{k\geq 0}(-1)^k(U^\top (G-P) U (C^{-1}+U^\top P U)^{-1})^k \right)U^\top G\\
&
=\sum_{k\geq 0}(-1)^{k+1} GU(C^{-1}+U^\top P U)^{-1}(U^\top (G-P) U (C^{-1}+U^\top P U)^{-1})^k U^\top G\\
&=\sum_{k\geq 0} GF((G|_{\mathbb F}-P)F)^{k}G
.
\end{align*}

\end{proof}

\subsection{Switching using 
the Woodbury Identity}\label{sec:woodbury_lemmas}

In this section we will use the expansion formula \Cref{lem:woodbury} to study the Green's function after local resampling. If vertex $o$ has a tree neighborhood, the following lemma gives a simple formula for the average of the Green's function over the radius $\ell$ ball. It will be used later to simplify the expression.

\begin{lemma}\label{lem:boundaryreduction}
Fix $z\in {\mathbf D}$, a $d$-regular graph $\cG\in \Omega(z)$, and edges $(i,o)\in \cG$. 
We denote $\cT=\cB_\ell(o, \cG)$, $T_{\bf S}$ the local resampling  with resampling data ${\bf S}=\{(l_\al, a_\al), (b_\al, c_\al)\}_{\al\in\qq{\mu}}$ around  $o$, and $\widetilde G=T_{\bf S}G$. We denote $\sfA_i:=\{\al\in \qq{1, \mu}: \dist_{\cT}(i, l_\al)=\ell+1\}$ and $\sfA_i^\complement=\qq{\mu}\setminus \sfA_i$. Condition on that  $I(\{(i,o),\cG)=1$ (recall from \Cref{def:indicator}), then for $\sfL\in \{l,a\}$ and any vertex $v\in \qq{1,N}$ such that $\dist_\cG(v, o)\geq \fR/2$, 
\begin{align}
\label{e:aveall}
(d-1)^{-\ell/2}\sum_{\alpha=1}^\mu G_{v \sfL_\alpha}=\ft_1(\sfL) G_{vo},
\end{align}
and 
\begin{align}\label{e:subset}
(d-1)^{-\ell/2}\sum_{\alpha \in\sfA_i} G_{v \sfL_\alpha}=\ft_2(\sfL)G_{vi}+\ft_3(\sfL) G_{vo}.
\end{align}
\end{lemma}
\begin{proof}[Proof of Lemma \ref{lem:boundaryreduction}]
We will only prove \eqref{e:subset}, and the proof for \eqref{e:aveall} follows from essentially the same argument. 

By our assumption that $I(\{(i,o)\}, \cG)=1$, vertex $o$ has radius $\fR/2$ tree neighborhood. We denote $\cS_0=\{o\}$, and $\cS_k=\{y\in \qq{N}: \dist_\cT(y,o)=k, \dist_\cT(y,i)=k+1\}$ for $0\leq k\leq \ell$. $\cS_k$ is the collection of vertices in $\cT^{(i)}$, which are distance $k$ to the root vertex $o$. 
By using the equation $(G(H-z))_{vo}=\delta_{vo}=0$,
\begin{align*}
\frac1{\sqrt{d-1}}\sum_{y\in \cS_1} G_{vy}=zG_{vo}-\frac1{\sqrt{d-1}}G_{vi}.
\end{align*}
More generally, by summing over $(G(H-z))_{vy}=0$ for $y\in \cS_k$ we get
\begin{align*}
    0&=\sum_{y\in \cS_k}(G(H-z))_{vy}=-z\sum_{y\in \cS_k}G_{vy}+\sum_{y\in \cS_k}(GH)_{vy}\\
    &=-z\sum_{y\in \cS_k}G_{vy}+\sqrt{d-1}\sum_{y\in \cS_{k-1}}G_{vy}+\frac{1}{\sqrt{d-1}}\sum_{y\in \cS_{k+1}}G_{vy}.
\end{align*}
We can convert this into an equation of the sum $\sum_{y\in \cS_k}G_{v y}$, for $k\geq 1$,
\begin{align*}
&\phantom{{}={}}\sum_{y\in \cS_{k+1}}\left(\frac{1}{\sqrt{d-1}}\right)^{k+1}G_{v y}=z\sum_{y\in \cS_{k}}\left(\frac{1}{\sqrt{d-1}}\right)^{k}G_{vy}-\sum_{y\in \cS_{k-1}}\left(\frac{1}{\sqrt{d-1}}\right)^{k-1}G_{vy}.
\end{align*}

 The recursive equation gives
\begin{align*}
\left(
\begin{array}{c}
(d-1)^{-\ell/2}\sum_{y\in \cS_\ell}G_{v y}\\
(d-1)^{-(\ell+1)/2}\sum_{y\in \cS_{\ell+1}} G_{vy}
\end{array}\right)
&=
\left(
\begin{array}{cc}
0&1\\
-1&z
\end{array}
\right) 
\left(
\begin{array}{c}
(d-1)^{-(\ell-1)/2}\sum_{y\in \cS_{\ell-1}}G_{v y}\\
(d-1)^{-\ell/2}\sum_{y\in \cS_{\ell}} G_{vy}
\end{array}\right)\\
&=
\left(\begin{array}{cc}
0&1\\
-1&z
\end{array}
\right)^{\ell+1} \left(
\begin{array}{c}
G_{vi}/\sqrt{d-1}\\
 G_{vo}
\end{array}
\right).
\end{align*}

The eigenvalues of the transfer matrix are given by $\msc$ and $1/\msc$. Therefore $(d-1)^{-\ell/2}\sum_{y\in \cS_\ell}G_{v y}$, 
$(d-1)^{-(\ell+1)/2}\sum_{y\in \cS_{\ell+1}} G_{vy}$ can be written as linear combination of $G_{vi}, G_{vo}$, with bounded coefficients, using that $|\msc|^\ell\asymp 1$.
The claim \eqref{e:subset} follows by noticing that 
\begin{align*}
    &\sum_{\alpha \in\sfA_i}(d-1)^{-\ell/2} G_{v l_\alpha}
    =(d-1)^{-(\ell-2)/2}\sum_{y\in \cS_{\ell}}G_{v y}.\\
    &\sum_{\alpha \in\sfA_i}(d-1)^{-\ell/2} G_{v a_\alpha}
    =(d-1)^{-\ell/2}\sum_{y\in \cS_{\ell+1}}G_{v y}.
\end{align*}

\end{proof}
In the following lemma, we gather estimates on the difference in Green's functions before and after local resampling, using \Cref{lem:woodbury}.
\begin{lemma}\label{lem:generalQlemma}
Fix $z\in {\mathbf D}$, a $d$-regular graph $\cG\in \Omega(z)$, and edges $(i,o), (b,c), (b',c')\in \cG$. 
We denote $\cT=\cB_\ell(o, \cG)$, $T_{\bf S}$ the local resampling  with resampling data ${\bf S}=\{(l_\al, a_\al), (b_\al, c_\al)\}_{\al\in\qq{\mu}}$ around  $o$, and $\widetilde G=T_{\bf S}G$. Condition on that $\tcG\in \Omega(z)$ and $I(\{(i,o)\}\cup \{(b_\al, c_\al)\}_{\al\in \qq{\mu}},\cG)=1$ (recall from \Cref{def:indicator}), the following holds  
\begin{enumerate}
\item  Fix any vertices $s,w$ in $ \cG$,
$\widetilde G_{sw}-G_{sw}$ can be rewritten as a weighted sum 
\begin{align*}
    \widetilde G_{sw}-G_{sw}=\sum_{\bms}\fc_\bms U_{\bms}+\cE,
\end{align*}
the summation is over indices $\bms=(s_1, s_2,\cdots, s_{2k})\in (\{l_\al, a_\al, b_\al, c_\al\}_{\al\in \qq{\mu}})^{2k}$ for some $1\leq k\leq \fb$;
\begin{align*}
    U_{\bms}=(d-1)^{2\ell}G_{ss_1} \widehat U_{s_2 s_3} \widehat U_{s_4 s_5}
   \cdots \widehat U_{s_{2k-2}s_{2k-1}}G_{s_{2k} w},
\end{align*}
where  $\widehat U_{s_{2j}s_{2j+1}}=(d-1)^{2\ell} (G_{s_{2j} s_{2j+1}}-P_{s_{2j} s_{2j-1}})$; and the total weights are bounded by $\poly(\ell)$: \begin{align*}
    \sum_{\bms}|\fc_{\bms}|=\OO( \poly(\ell)).
\end{align*}

\item 
For $s\in \{o,i\}$ and any vertex $w$ in $\cG$, 
$\widetilde G_{sw}-G_{sw}$ can be rewritten as a weighted sum 
\begin{align}\begin{split}\label{e:tG-Gexp}
  \widetilde G_{sw}-G_{sw}
  &=\fc_1  G_{ow} +\fc_2 G_{iw}+\frac{\fc_3 \sum_{\al\in \sfA_i} G_{b_\al w}+\fc_4 \sum_{\al\in \sfA_i} G_{c_\al w}}{(d-1)^{\ell/2}}+\sum_{\bms_1}\fc_{\bms_1} U_{\bms_1}\\
  & + \frac{\fc_5\sum_{\al\in\sfA^\complement_i} G_{b_\al w}+ \fc_6\sum_{\al\in\sfA^\complement_i} G_{c_\al w}}{(d-1)^{\ell/2}} +\sum_{\bms_2}\fc_{\bms_2} U_{\bms_2}+\cE,
\end{split}\end{align}
where $\sfA_i:=\{\al\in \qq{1, \mu}: \dist_{\cT}(i, l_\al)=\ell+1\}$ and $\sfA_i^\complement :=\qq{\mu}\setminus \sfA_i$. The summation in the first line is over indices $\bms_1=(s_1, s_2,\cdots, s_{2k})\in (\{l_\al, a_\al, b_\al, c_\al\}_{\al\in \qq{\mu}})^{2k}$ for some $1\leq k\leq \fb$;
\begin{align*}
    U_{\bms_1}= \widehat U_{ss_1}\widehat U_{s_2 s_3} \widehat U_{s_4 s_5}
   \cdots \widehat U_{s_{2k-2}s_{2k-1}}G_{s_{2k}w},
\end{align*}
where $\widehat U_{ss_1}=(d-1)^{2\ell}(G_{s s_1}-P_{s s_1})$ and $\widehat U_{s_{2j}s_{2j+1}}=(d-1)^{2\ell} (G_{s_{2j} s_{2j+1}}-P_{s_{2j} s_{2j-1}})$; the summation in the second line is over indices $\bms_2=(s_1, s_2,\cdots, s_{2k+1})\in (\{l_\al, a_\al, b_\al, c_\al\}_{\al\in \qq{\mu}})^{2k+1}$ for some $1\leq k\leq \fb$;
\begin{align*}
    U_{\bms_2}= \widehat U_{s_1 s_2} \widehat U_{s_3 s_4}
   \cdots \widehat U_{s_{2k-1}s_{2k}}G_{s_{2k+1}w},
\end{align*}
where $\widehat U_{s_{2j-1}s_{2j}}=(d-1)^{2\ell} (G_{s_{2j-1} s_{2j}}-P_{s_{2j-1} s_{2j}})$;
and the total weights are bounded by $\poly(\ell)$: \begin{align*}
    |\fc_1|+|\fc_2|+|\fc_3|+|\fc_4|+|\fc_5|+|\fc_6|+\sum_{\bms_1}|\fc_{\bms_1}|+\sum_{\bms_2}|\fc_{\bms_2}|=\OO( \poly(\ell)).
\end{align*}

\item    For any vertex $w$ in $ \cG$, $(1/\widetilde G_{ww})-(1/G_{ww})$ can be rewritten as a weighted sum 
\begin{align*}
   (1/\widetilde G_{ww})-(1/G_{ww})=\sum_{1\leq j\leq \fb} \frac{1}{G_{ww}^{j+1}}\sum_{\bms_j}\fc_{\bms_j} U_{\bms_j}+\cE,
\end{align*}
where for the $j$-th term, the summation is over indices 
$\bms_j=(s_1, s_2,\cdots, s_{2k})\in (\{l_\al, a_\al, b_\al, c_\al\}_{\al\in \qq{\mu}})^{2k}$ for some $j\leq k\leq j\fb$;
\begin{align*}
    U_{\bms_j}=
    \prod_{1\leq r\leq j} (d-1)^{2\ell} G_{s_{2r-1} w}G_{s_{2r} w}
    \prod_{j+1\leq r\leq k}\widehat U_{s_{2r-1} s_{2r}}, 
\end{align*}
where  $\widehat U_{s_{2j}s_{2j+1}}=(d-1)^{2\ell} (G_{s_{2j} s_{2j+1}}-P_{s_{2j} s_{2j-1}})$; and the total weights are bounded by $\OO(\poly(\ell))$.
\end{enumerate}
In all cases, the error $\cE$ satisfies
\begin{align*}
  \cE\leq N^{-2}.
\end{align*}
\end{lemma}

\begin{proof}[Proof of Lemma \ref{lem:generalQlemma}]
We first prove the second term on the righthand side of \eqref{eq:fexpansion3}. The first statement is easier, and can be proven in the same way. 
We recall the notations from \Cref{lem:woodbury}, in particular $P:=P(\cT,z, m_{sc})$. Thus, by Lemma \ref{lem:woodbury}, there is some finite $\fb$ such that, 
\begin{align}\label{e:G-Gexp}
\widetilde G-G
=
\sum_{k=1}^\fb GF\left( (G-P)F\right)^{k-1}G+\OO(N^{-2}),
\end{align}
where
\begin{align*}
F:=\sum_{\al \in \qq{\mu}} \xi_{\al}+\sum_{\al, \beta\in \qq{\mu}} \xi_{\al}\tP \xi_{\beta},
\end{align*}
which has nonzero entries only on the vertices $\{l_\al, a_\al, b_\al, c_\al\}_{\al\in \qq{\mu}}$, and for any $\al, \al'\in \qq{\mu}$ 
\begin{align}\label{e:Fbound}
    |F_{\sfL_\al \sfL'_{\al'}}|\lesssim (d-1)^{-\dist_{\cT}(l_\al, l_{\al'})},
    \quad \sfL, \sfL'\in \{l,a,b,c\}.
\end{align}

Thus we can rewrite the $k$-th term in \eqref{e:G-Gexp} as
\begin{align}\begin{split}\label{eq:fexpansion3}
&\phantom{{}={}}\left(G F ((G-P)F)^{k-1}G\right)_{sw}\\
&=\sum_{\bms}
(G_{ss_1}-P_{ss_1})
F_{s_1s_2}
(G_{s_2s_3}-P_{s_2s_3})
F_{s_3s_4}
\cdots
(G_{s_{2k-2}s_{2k-1}}-P_{s_{2k-2}s_{2k-1}})
F_{s_{2k-1}s_{2k}}G_{s_{2k}w}\\
&+\sum_{\bms}
P_{ss_1}
F_{s_1s_2}
(G_{s_2s_3}-P_{s_2s_3})
F_{s_3s_4}
\cdots
(G_{s_{2k-2}s_{2k-1}}-P_{s_{2k-2}s_{2k-1}})
F_{s_{2k-1}s_{2k}}G_{s_{2k}w},
\end{split}\end{align}
where the summation is over $\bms=(s_1,s_2,\cdots,s_{2k})\in (\{l_\al, a_\al, b_\al, c_\al\}_{\al\in \qq{\mu}})^{2k}$.
Roughly speaking, we have 
\begin{align*}\begin{split}
    \widetilde G-G
    &=GFG+\text{higher order terms}\\
    &=PFG+(G-P)FG+\text{higher order terms}
\end{split}\end{align*}
The terms linear in $G$ on the righthand side of \eqref{e:tG-Gexp} corresponds to $PFG$.

We can reorganize the first term on the righthand side of \eqref{eq:fexpansion3} in the following way
\begin{align}\label{e:biaodayi}
    \sum_{\bms}\fc_\bms\widehat U_{ss_1}\widehat U_{s_2s_3}\widehat U_{s_4 s_5}\cdots \widehat U_{s_{2k-2}s_{2k-1}} G_{s_{2k}w},\quad \fc_\bms:=f(\bms)(d-1)^{-2k\ell},
\end{align}
where $\widehat U_{ss_1}=(d-1)^{2\ell}(G_{s s_1}-P_{s s_1})$ and $\widehat U_{s_{2j}s_{2j+1}}=(d-1)^{2\ell} (G_{s_{2j} s_{2j+1}}-P_{s_{2j} s_{2j-1}})$,
and for $\bm s=(\sfL^{(1)}_{\al_1}, \sfL^{(2)}_{\al_2},\cdots, \sfL^{(2k)}_{\al_{2k}})$ and $ \sfL^{(1)}, \sfL^{(2)},\cdots, \sfL^{(2k)}\in \{l,a,b,c\}$, 
\begin{align*}
    f(\bms)=F_{s_1s_2}F_{s_3 s_4}\cdots F_{s_{2k-1}s_{2k}},\quad |f(\bms)|\leq \prod_{j=1}^k (d-1)^{-\dist_\cT(l_{\al_{2j-1}}, l_{\al_{2j}})},
\end{align*}
where we used \eqref{e:Fbound}. For the summation over the weights:
\begin{align*}\begin{split}
   \sum_{\bms}|\fc_{\bms}|
   &= \sum_{\bms}|f(\bms)| (d-1)^{-2k\ell}\lesssim \sum_{\bm\al}4^k(d-1)^{-2k\ell}\prod_{j=1}^{k} (d-1)^{-\dist_{\cT}(l_{\alpha_{2j-1}},l_{\alpha_{2j}})/2}\\
   &\leq 4^k
   \prod_{j=1}^{k}\sum_{\al_{2j-1}, \al_{2j}}(d-1)^{-\ell}(d-1)^{-\dist_{\cT}(l_{\alpha_{2j-1}},l_{\alpha_{2j}})/2}= \OO(\poly(\ell)). 
\end{split}\end{align*}
where in the last inequality, we used \eqref{e:sumal_beta}. These terms in \eqref{e:biaodayi} give the summation over $\bms_1$ in the first line of \eqref{e:tG-Gexp}.

For the second term on the righthand side of \eqref{eq:fexpansion3}, with $k\geq 2$, we reorganize them as
\begin{align}\label{e:biaodaer}
    \sum_{\bms}\fc_\bms\widehat U_{s_2s_3}\widehat U_{s_4 s_5}\cdots \widehat U_{s_{2k-2}s_{2k-1}} G_{s_{2k}w}, \quad \fc_\bms=f(\bms)(d-1)^{-2(k-1)\ell},
\end{align}
where $\widehat U_{s_{2j}s_{2j+1}}=(d-1)^{2\ell} (G_{s_{2j} s_{2j+1}}-P_{s_{2j} s_{2j-1}})$,
and for $\bm s=(\sfL^{(1)}_{\al_1}, \sfL^{(2)}_{\al_2},\cdots, \sfL^{(2k)}_{\al_{2k}})$ and $ \sfL^{(1)}, \sfL^{(2)},\cdots, \sfL^{(2k)}\in \{l,a,b,c\}$, 
\begin{align*}
    f(\bms)=P_{ss_1}F_{s_1s_2}F_{s_3 s_4}\cdots F_{s_{2k-1}s_{2k}},\quad |f(\bms)|\leq (d-1)^{-\ell/2}\prod_{j=1}^k (d-1)^{-\dist_\cT(l_{\al_{2j-1}}, l_{\al_{2j}})},
\end{align*}
where we used \eqref{e:Fbound}. The summation over the weights:
\begin{align*}\begin{split}
   \sum_{\bms}|\fc_{\bms}|
   &= \sum_{\bms}|f(\bms)| (d-1)^{-2(k-1)\ell}\lesssim \sum_{\bm\al}4^k(d-1)^{-k\ell}\prod_{j=1}^{k} (d-1)^{-\dist_{\cT}(l_{\alpha_{2j-1}},l_{\alpha_{2j}})/2}\\
   &=4^k
   \prod_{j=1}^{k}\sum_{\al_{2j-1}, \al_{2j}}(d-1)^{-\ell}(d-1)^{-\dist_{\cT}(l_{\alpha_{2j-1}},l_{\alpha_{2j}})/2}= \OO(\poly(\ell)). 
\end{split}\end{align*}
where in the last inequality, we used \eqref{e:sumal_beta}. 
These terms in \eqref{e:biaodaer} give the summation over $\bms_2$ in the first line of \eqref{e:tG-Gexp}.

For the second term on the righthand side of \eqref{eq:fexpansion3}, with $k=1$ (i.e. the term $(PFG)_{sw}$),
we can sum over $s_1$ first. Since $s\in\{o,i\}$, $P_{ss_1}\neq 0$ only if $s_1\in \{l_\al, a_\al\}_{\al\in \qq{\mu}}$. Therefore fix $s_2=\sfL_{\al_2}$ for some $\sfL\in \{l,a,b,c\}$ and $1\leq \alpha_2\leq \mu$, we have
\begin{align}\label{e:coefficient}
    \sum_{s_1}P_{ss_1}F_{s_1 s_2}=\sum_{{\al_1}\in \qq{\mu}}\left(P_{sl_{\al_1}}F_{l_{\al_1} \sfL_{\al_2}}+P_{sa_{\al_1}}F_{a_{\al_1} \sfL_{\al_2}}\right)
\end{align}
Recall the set $\sfA_i=\{\al\in \qq{\mu}: \dist_\cT(i,l_\al)=\ell+1\}$. First we notice that the above summation depends only on if $\al_2\in\sfA_i$ or $\al_2\notin\sfA_i$. Moreover, thanks to \eqref{e:Fbound} and  \eqref{e:sumal_beta}.
\begin{align*}
    \left|\sum_{s_1}P_{ss_1}F_{s_1 s_2}\right|\lesssim \sum_\al (d-1)^{-\ell/2}(d-1)^{-\dist_{\cT}(l_\al, l_{\al_2})}\lesssim (d-1)^{-\ell/2}.
\end{align*}
From the discussion above we conclude that \eqref{e:coefficient} is in the following form
\begin{align*}
    \sum_{s_1}P_{ss_1}F_{s_1 s_2}=(d-1)^{-\ell/2}(\widehat\fc_1(\sfL)\bm1(\al_2\in\sfA_i)+ \widehat\fc_2(\sfL)\bm1(\al_2\not\in\sfA_i)),
\end{align*}
where the two coefficients $|\widehat\fc_1(\sfL)|,|\widehat\fc_2(\sfL)|\leq \poly(\ell)$. Separating $\sfL_{\al_2}$ in \eqref{e:coefficient} into two cases: $\sfL\in \{a,l\}$ or $\sfL\in \{b,c\}$, and using Lemma \ref{lem:boundaryreduction} for the case $\sfL\in \{a,l\}$, we get 
\begin{align*}\begin{split}
    &\phantom{{}={}}\sum_{s_1,s_2}P_{ss_1}F_{s_1 s_2}G_{s_2 w}
    =\sum_{s_2=\sfL_{\al_2}}(\widehat\fc_1(\sfL)\bm1(\al_2\in\sfA_i)+ \widehat\fc_2(\sfL)\bm1(\al_2\not\in\sfA_i))G_{s_2 w}\\
    &=\fc_1  G_{ow} +\fc_2 G_{iw}+\frac{\fc_3 \sum_{\al\in \sfA_i} G_{b_\al w}+\fc_4 \sum_{\al\in \sfA_i} G_{c_\al w}+\fc_5\sum_{\al\in\sfA^\complement_i} G_{b_\al w}+ \fc_6\sum_{\al\in\sfA^\complement_i} G_{c_\al w}}{(d-1)^{\ell/2}}, 
\end{split}\end{align*}
where the coefficients $|\fc_j|\leq \OO(\poly(\ell))$ for $1\leq j\leq 6$ depending on $s\in\{o,i\}$.
This finishes the proof of the second statement in \Cref{lem:generalQlemma}.

For the first statement in \Cref{lem:generalQlemma}, we have exactly the same expansion as in \eqref{eq:fexpansion3}, and the first statement in \Cref{lem:generalQlemma} follows from the same analysis as the first term on the righthand side of \eqref{eq:fexpansion3}.

For the last statement, by Taylor expansion 
\begin{align}\label{e:texp}
    \frac{1}{\widetilde G_{ww}}-  \frac{1}{ G_{ww}}
    =\sum_{1\leq j\leq \fb} \frac{(G_{ww}-\wt G_{ww})^j}{G_{ww}^{j+1}}+\OO(N^{-2}).
\end{align}
Then the last statement in \Cref{lem:generalQlemma} follows from the first statement.

\end{proof}

In the following lemma, we gather estimates on the difference of $Q-Y_\ell(Q)$ before and after local resampling, using \Cref{lem:woodbury}.
\begin{lemma}\label{lem:Qlemma}
Fix $z\in {\mathbf D}$, a $d$-regular graph $\cG\in \Omega(z)$, and edges $(i,o), (b,c), (b',c')\in \cG$. 
We denote $\cT=\cB_\ell(o, \cG)$, $T_{\bf S}$ the local resampling  with resampling data ${\bf S}=\{(l_\al, a_\al), (b_\al, c_\al)\}_{\al\in\qq{\mu}}$ around  $o$, and $\widetilde G=T_{\bf S}G$. Condition on that $\tcG\in \Omega(z)$ and $I(\{(i,o)\}\cup \{(b_\al, c_\al)\}_{\al\in \qq{\mu}},\cG)=1$ (recall from \Cref{def:indicator}), then $(\wt Q-Y_\ell(\wt Q))-(Q-Y_\ell(Q))$ can be rewritten as
\begin{align}\begin{split}\label{e:tQ-Qexp}
  \frac{1-Y'_\ell(Q)}{Nd}\sum_{u,v}A_{uv}\left(\sum_{\bms_1}\fc_{\bms_1} U^{(u)}_{\bms_1}+\sum_{0\leq j\leq \fb}\frac{G_{uv}}{G_{uu}^{j+1}}\sum_{\bms_2}\fc_{\bms_2} U^{(u,v)}_{\bms_2}+\sum_{0\leq j\leq \fb}\frac{1}{G_{uu}^{j+1}}\sum_{\bms_3}\fc_{\bms_3} U^{(u,v)}_{\bms_3}\right)+\cE,
\end{split}\end{align}
where 
the first summation is over indices $\bms_1=(s_1, s_2,\cdots, s_{2k})\in (\{l_\al, a_\al, b_\al, c_\al\}_{\al\in \qq{\mu}})^{2k}$ for some $1\leq k\leq \fb$;
\begin{align*}
    U^{(u,v)}_{\bms_1}=(d-1)^{2\ell} G_{vs_1} (d-1)^{2\ell} G_{s_{2} v} \widehat U_{s_3 s_4} \widehat U_{s_5 s_6}
   \cdots \widehat U_{s_{2k-1}s_{2k}};
\end{align*}
the second summation is over indices $\bms_2=(s_1, s_2,\cdots, s_{2k})\in (\{l_\al, a_\al, b_\al, c_\al\}_{\al\in \qq{\mu}})^{2k}$ for some $j+1\leq k\leq (j+1)\fb$;
\begin{align*}
    U^{(u,v)}_{\bms_2}=(d-1)^{2\ell} G_{us_1} (d-1)^{2\ell} G_{s_{2} v}
    \prod_{2\leq r\leq j+1} (d-1)^{2\ell} G_{s_{2r-1} u}(d-1)^{2\ell}G_{s_{2r} u}
    \prod_{j+2\leq r'\leq k}\widehat U_{s_{2r'-1} s_{2r'}}; 
\end{align*}
the third summation is over indices $\bms_3=(s_1, s_2,\cdots, s_{2k})\in (\{l_\al, a_\al, b_\al, c_\al\}_{\al\in \qq{\mu}})^{2k}$ for some $j+2\leq k\leq (j+2)\fb$;
\begin{align*}
    U^{(u,v)}_{\bms_3}
    &=\prod_{1\leq r\leq 2}(d-1)^{2\ell} G_{us_{2r-1}} (d-1)^{2\ell} G_{s_{2r} v}\\
    &\times
    \prod_{3\leq r'\leq j+2} (d-1)^{2\ell} G_{s_{2r'-1} u}(d-1)^{2\ell}G_{s_{2r'} u}
    \prod_{j+3\leq r''\leq k}\widehat U_{s_{2r''-1} s_{2r''}} .
\end{align*}
In all the above cases $\widehat U_{s_{2k-1}s_{2k}}=(d-1)^{2\ell} (G_{s_{2k-1} s_{2k}}-P_{s_{2k-1} s_{2k}})$; the total summation of weights $|\fc_\bms|$ is bounded by $\OO(\poly(\ell))$; and the error $\cE$ satisfies
\begin{align*}
   \cE\leq N^{-2} +\Phi^2.
\end{align*}
\end{lemma}

\begin{proof}[Proof of Lemma \ref{lem:Qlemma}]
By \eqref{eq:Yprime},
\begin{align}\label{e:wtQ-Q}
    (\wt Q-Y_\ell(\wt Q))-(Q-Y_\ell(Q))
    =(1-Y'_\ell(Q))(\wt Q-Q)+\OO(|\wt Q-Q|^2).
\end{align}

Next, we investigate the difference $\wt Q-Q$. 
By the Schur complement formula, we can write 
\begin{align}\label{e:schurexp}
    G_{vv}^{(u)}=G_{vv}-\frac{G_{uv}^2}{G_{uu}}.
\end{align}
By
\Cref{lem:generalQlemma}, for $x,y\in \{u,v\}$,
$\widetilde G_{xy}-G_{xy}$ can be rewritten as a weighted sum 
\begin{align}\label{e:diftG-G}
    \widetilde G_{xy}-G_{xy}=\sum_{\bms}\fc_\bms U_{\bms}+\cE,
\end{align}
the summation is over indices $\bms=(s_1, s_2,\cdots, s_{2k})\in (\{l_\al, a_\al, b_\al, c_\al\}_{\al\in \qq{\mu}})^{2k}$ for some $1\leq k\leq \fb$;
\begin{align*}
    U_{\bms}=(d-1)^{2\ell} G_{xs_1} (d-1)^{2\ell} G_{s_{2k} y} \widehat U_{s_2 s_3} \widehat U_{s_4 s_5}
   \cdots \widehat U_{s_{2k-2}s_{2k-1}},
\end{align*}
where $\widehat U_{s_{2j}s_{2j+1}}=(d-1)^{2\ell} (G_{s_{2j} s_{2j+1}}-P_{s_{2j} s_{2j-1}})$; and the total weights are bounded by $\OO(\poly(\ell))$: \begin{align*}
    \sum_{\bms}|\fc_{\bms}|=\OO( \poly(\ell)).
\end{align*}

For the difference $\widetilde G_{vv}^{(u)}- G_{vv}^{(u)}$, using \eqref{e:schurexp} and \eqref{e:texp}, we can rewrite it as
\begin{align*}
  \widetilde G_{vv}^{(u)}- G_{vv}^{(u)}
  &=  (\wt G_{vv}-G_{vv})
  +\frac{2G_{uv}(\wt G_{uv} -G_{uv})}{G_{uu}}\sum_{0\leq j\leq \fb }\left(\frac{ G_{uu}-\wt G_{uu}}{G_{uu}}\right)^j\\
  &+\frac{(\wt G_{uv} -G_{uv})^2}{G_{uu}}\sum_{0\leq j\leq \fb }\left(\frac{ G_{uu}-\wt G_{uu}}{G_{uu}}\right)^j+\OO(N^{-2}).
\end{align*}
Then we can replace the differences $\wt G_{vv}-G_{vv}$, $\wt G_{uu}-G_{uu}$ and $\wt G_{uv}-G_{uv}$ by \eqref{e:diftG-G}, giving
\begin{align}\label{e:tG-Gexpa}
   \widetilde G_{vv}^{(u)}- G_{vv}^{(u)}= \sum_{\bms_1}\fc_\bms U^{(v)}_{\bms_1}+\sum_{0\leq j\leq \fb}\frac{G_{uv}}{G_{uu}^{j+1}}\sum_{\bms_2}\fc_{\bms_2} U^{(u,v)}_{\bms_2}+\sum_{0\leq j\leq \fb}\frac{1}{G_{uu}^{j+1}}\sum_{\bms_3}\fc_{\bms_3} U^{(u,v)}_{\bms_3}+\cE
\end{align}
where the first summation is over indices $\bms_1=(s_1, s_2,\cdots, s_{2{k_1}})\in (\{l_\al, a_\al, b_\al, c_\al\}_{\al\in \qq{\mu}})^{2{k_1}}$ for some $1\leq {k_1}\leq \fb$;
\begin{align*}
    U^{(v)}_{\bms_1}=(d-1)^{2\ell} G_{vs_1} (d-1)^{2\ell} G_{s_{2} v} \widehat U_{s_3 s_4} \widehat U_{s_5 s_6}
   \cdots \widehat U_{s_{2{k_1}-1}s_{2{k_1}}};
\end{align*}
the second summation is over indices $\bms_2=(s_1, s_2,\cdots, s_{2{k_2}})\in (\{l_\al, a_\al, b_\al, c_\al\}_{\al\in \qq{\mu}})^{2{k_2}}$ for some $j+1\leq {k_2}\leq (j+1)\fb$;
\begin{align*}
    U^{(u,v)}_{\bms_2}=(d-1)^{2\ell} G_{us_1} (d-1)^{2\ell} G_{s_{2} v}
    \prod_{2\leq r\leq j+1} (d-1)^{2\ell} G_{s_{2r-1} u}(d-1)^{2\ell}G_{s_{2r} u}
    \prod_{j+2\leq r'\leq {k_2}}\widehat U_{s_{2r'-1} s_{2r'}} ;
\end{align*}
the third summation is over indices $\bms_3=(s_1, s_2,\cdots, s_{2{k_3}})\in (\{l_\al, a_\al, b_\al, c_\al\}_{\al\in \qq{\mu}})^{2{k_3}}$ for some $j+2\leq {k_3}\leq (j+2)\fb$;
\begin{align*}
    U^{(u,v)}_{\bms_3}
    &=\prod_{1\leq r\leq 2}(d-1)^{2\ell} G_{us_{2r-1}} (d-1)^{2\ell} G_{s_{2r} v}\\
    &\times \prod_{3\leq r'\leq j+2} (d-1)^{2\ell} G_{s_{2r'-1} u}(d-1)^{2\ell}G_{s_{2r'} u}
    \prod_{j+3\leq r''\leq {k_3}}\widehat U_{s_{2r''-1} s_{2r''}} ;
\end{align*}
In all the above cases $\widehat U_{s_{2k}s_{2k+1}}=(d-1)^{2\ell} (G_{s_{2k} s_{2k+1}}-P_{s_{2k} s_{2k-1}})$; and the total summation of weights $|\fc_\bms|$ is bounded by $\OO(\poly(\ell))$.

If we average over the edges $(u,v)\in \cG$ in \eqref{e:tG-Gexpa}, by the Ward identity, we conclude that 
\begin{align}\label{e:wtQ-Q2}
    |\wt Q-Q|\lesssim \Phi.
\end{align}
The claim of \Cref{lem:Qlemma} follows from plugging \eqref{e:tG-Gexpa} and \eqref{e:wtQ-Q2} into \eqref{e:wtQ-Q}.

\end{proof}

\section{Bounds on Error Terms}\label{e:error_term}
In the previous section, we collect various expansions of the differences in Green's functions before and after local resampling. In this section, we show that the error terms in these differences are in fact small. Fix $z\in {\mathbf D}$, a $d$-regular graph $\cG\in \Omega(z)$, we introduce the error parameter
\begin{align}\label{e:defUG}
    \Pi_p(z):=\bm1(\cG\in \Omega(z))\left(|Q(z;\cG)-Y_\ell(Q(z;\cG))|+\Phi(z)\right)^{2p-1},
\end{align}
and recall $\Phi(z)=N^\fc \cdot \Im[m(z)]/(N\eta)$ from \Cref{def:phidef}.

First, we will show that if we delete a common neighbor of two vertices, the expected value of the Green's function is as small as it would be if the two vertices were chosen at random.
\begin{proposition}
    \label{lem:deletedalmostrandom}
    Fix $z\in {\mathbf D}$ and a $d$-regular graph $\cG\in \Omega(z)$. Let $\Pi_p(z)$ be as in \eqref{e:defUG}.  We denote the indicator function $I(o,\cG)=1$ if for any $v\in \cB_\ell(o, \cG)$, $v$ has a radius $\fR$ tree neighborhood. We denote $T_{\bf S}$ the local resampling around $o$ with resampling data $\bf S$, and $\widetilde G=T_{\bf S}G$. Then for any neighbor vertices $i,j$ of $o$, we have
    \begin{align}\label{e:sameasdisconnect}
        \bE[I(o,\cG) |G^{(o)}_{ij}(z)|^2 \Pi_p(z)]\lesssim \bE[\Psi_p(z)].
    \end{align}
\end{proposition}

We need the following estimates for the proof of \Cref{lem:deletedalmostrandom}.

\begin{lemma}
    \label{c:expectationbound}
   Fix $z\in {\mathbf D}$ and a $d$-regular graph $\cG\in \Omega(z)$. We denote $T_{\bf S}$ the local resampling around $o$ with resampling data $\bf S$, $\bT$ the vertex set of $\cT=\cB_\ell(o,\cG)$ and the set $\bW=\{b_\al\}_{\al \in \qq{\mu}}$.
   Given indices $\al\neq \beta$, and a vertex $y\sim c_\beta$ (where the adjacency $\sim$ is with respect to the graph $\cG$) or $y\in \{a_\gamma\}_{\gamma\in\qq{\mu}}$, the following holds
\begin{align}\label{e:Gcxbound}
 \bE_\bfS[|(G^{(\bT)}(G^{(\bT)}|_{\bW})^{-1}G^{(\bT)})_{c_\al c_\beta}|^2],
 \quad 
\bE_\bfS[|G_{c_\al c_\beta}^{(\bT\bW)}|^2],
\quad 
\bE_\bfS[|G_{c_\al y}^{(\bT\bW)}|^2]\lesssim \Phi+N^{-1+\fc}.
\end{align}
\end{lemma}

\begin{proof}[Proof of \Cref{c:expectationbound}]
    We only proof the first inequality in \eqref{e:Gcxbound} as the other proofs are similar.
    We recall the set of resampling data $\sfF(\cG)\subset \sfS(\cG)$ from \Cref{lem:goodresamplingdata}. Then 
    \Cref{lem:configuration} and \Cref{thm:prevthm} together imply that
    \begin{align*}
       \bE_\bfS[|(G^{(\bT)}(G^{(\bT)}|_{\bW})^{-1}G^{(\bT)})_{c_\al c_\beta}|^2]
       =\bE_\bfS[\bm1(\bfS\in \sfF(\cG))|(G^{(\bT)}(G^{(\bT)}|_{\bW})^{-1}G^{(\bT)})_{c_\al c_\beta}|^2]+\OO(N^{-1+\fc}).
    \end{align*}
    
    In the rest, we restrict to the event that $\bfS\in \sfF(\cG)$. We denote the diagonal matrix $D=\diag(G^{(\bT)}|_{\bW})$, then we can Taylor expand the inverse as
\begin{align*}
    (G^{(\bT)}|_{\bW})^{-1}=\sum_{k=0}^\fb D^{-1}\left((G^{(\bT)}|_{\bW}-D)D^{-1}\right)^k+\OO(N^{-2}).
\end{align*}

Excluding some negligible part, we can rewrite $(G^{(\bT)}(G^{(\bT)}|_{\bW})^{-1}G^{(\bT)})_{c_\al c_\beta}$ as
\begin{align}\begin{split}\label{eq:firsterror}
&\phantom{{}={}}\bigg[G^{(\bT)}
\sum_{k=0}^\fb D^{-1}\left((G^{(\bT)}|_{\bW}-D)D^{-1}\right)^k
G^{(\bT)}\bigg]_{c_\alpha c_\beta}\\
&=
\sum_{k=0}^\fb\sum_{\al_1, \al_2,\cdots,\al_{k+1}\in \qq{\mu}}G^{(\bT)}_{c_\al b_{\al_1}}D^{-1}_{b_{\al_1}b_{\al_1}} V_{b_{\al_1} b_{\al_2}} D^{-1}_{b_{\al_2}b_{\al_2}}V_{b_{\al_2} b_{\al_3}} D^{-1}_{b_{\al_3}b_{\al_3}}\cdots D^{-1}_{b_{\al_{k+1} }b_{\al_{k+1}}}G^{(\bT)}_{b_{\al_{k+1}}c_\beta}.
\end{split}\end{align}
For each summand in \eqref{eq:firsterror}: if $k=0$, since $\al\neq \beta$, either $\al\neq \al_1$ or $\beta\neq \al_1$; if $k\geq 1$, $V_{b_{\al_1} b_{\al_2}}$ is nonzero only if $\al_1\neq \al_2$, in which case $V_{b_{\al_1} b_{\al_2}}=G^{(\bT)}_{b_{\al_1}b_{\al_2}}$. Therefore, by Theorem \ref{thm:prevthm}, 
\begin{align*}
    \sum_\al |V_{b_\al b_\beta}|\leq (d-1)^\ell\varepsilon\leq (d-1)^{-\ell}.
\end{align*} 
Therefore, as we know that to get from $\alpha$ to $\beta$ we must change indices, we can then bound the sum \eqref{eq:firsterror} as
\begin{align*}
    |\eqref{eq:firsterror}|\leq \frac{1}{(d-1)^\ell}\sum_{\gamma\neq \al}|G^{(\bT)}_{c_\al b_\gamma}|+\frac{1}{(d-1)^\ell}\sum_{\gamma\neq \beta}|G^{(\bT)}_{c_\beta b_\gamma}|
    +\frac{1}{(d-1)^{2\ell}}\sum_{\gamma\neq \gamma'}|G^{(\bT)}_{b_\gamma b_{\gamma'}}|.
\end{align*}
Taking square and expectation over $\bfS$, thanks to the Ward identity, we conclude
\begin{align*}
    \bE_{\bfS}[|\eqref{eq:firsterror}|^2]
    \lesssim 
    \sum_{\gamma\neq \al}\frac{\bE_\bfS[|G^{(\bT)}_{c_\al b_\gamma}|^2]}{(d-1)^\ell} 
    +\sum_{\gamma\neq \beta}\frac{\bE_\bfS[|G^{(\bT)}_{c_\beta b_\gamma}|^2]}{(d-1)^\ell}
    +\sum_{\gamma\neq \gamma'}
    \frac{\bE_\bfS[|G^{(\bT)}_{b_\gamma b_{\gamma'}}|^2]}{(d-1)^{2\ell}}\lesssim \Phi. 
\end{align*}
This finishes the proof for the first statement in \eqref{e:Gcxbound}.

\end{proof}

\begin{proof}[Proof of \Cref{lem:deletedalmostrandom}]
We will do a local resampling around the vertex $o$. We split the lefthand side of \eqref{e:sameasdisconnect} into two terms
\begin{align}\begin{split}\label{e:GUmain0}
&\phantom{{}={}}\bE\left[I(o,\cG) |G^{(o)}_{ij}|^2 \Pi_p\right] \\ 
&=\bE\left[I(o,\cG)I(o,\wt\cG)\bm1(\tcG\in \oOmega) |G^{(o)}_{ij}|^2 \Pi_p\right]  +\bE\left[I(o,\cG)\bE_\bfS[1-I(o,\wt\cG)\bm1(\tcG\in \oOmega) ] |G^{(o)}_{ij}|^2 \Pi_p\right] \\
&=\bE\left[I(o,\cG)I(o,\wt\cG)\bm1(\tcG\in \oOmega)  |G^{(o)}_{ij}|^2 \Pi_p\right]  +\OO\left(\frac{1}{N^{1-2\fc}}\bE\left[I(o,\cG) |G^{(o)}_{ij}|^2 \Pi_p\right] \right)\\
&=\bE\left[I(o,\cG)I(o,\wt\cG)\bm1(\tcG\in \oOmega)  |G^{(o)}_{ij}|^2 \Pi_p\right]  +\OO\left(\frac{1}{N^{1-2\fc}}\bE\left[ \Pi_p\right] \right),\\
\end{split}\end{align}
where in the third line we used \Cref{lem:wellbehavedswitch}, that $\bE_\bfS[1-I(o,\tcG)\bm1(\tcG\in \oOmega) ]\leq N^{-1+2\fc}$; in the last line we used that for $\cG\in \Omega$, $|G_{ij}^{(o)}|\lesssim 1$. 

Next we analyze the first term in the last line of \eqref{e:GUmain0}.
\begin{align}\begin{split}\label{e:GUmaint1}
&\phantom{{}={}}\bE\left[I(o,\cG)I(o,\wt\cG)\bm1(\tcG\in \oOmega) |G^{(o)}_{ij}|^2 \Pi_p\right]\\
&=
\bE\left[I(o,\cG)I(o,\wt\cG)\bm1(\tcG\in \Omega) |G^{(o)}_{ij}|^2 \Pi_p\right]  +\bE\left[I(o,\cG)I(o,\wt\cG)\bm1(\tcG\in \oOmega\setminus \Omega)  |G^{(o)}_{ij}|^2 \Pi_p\right].
\end{split}\end{align}
For the last term in \eqref{e:GUmaint1}, we have that
\begin{align}\begin{split}\label{e:ffbbcopy}
&\phantom{{}={}}\bE\left[I(o,\cG)I(o,\wt\cG)\bm1(\tcG\in \oOmega\setminus \Omega)  |G^{(o)}_{ij}|^2 \Pi_p\right]
\lesssim \bE\left[\bm1(\tcG\in \oOmega\setminus \Omega)  \Pi_p\right]\\
&\lesssim \bE\left[\Pi_p^{2p/(2p-1)}\right]^{1-1/2p}\bE\left[\bm1( \tcG\in \oOmega\setminus\Omega)\right]^{1/2p}\\
&\lesssim  N^{-(\fd-3)/2p} \bE\left[\Pi_p^{2p/(2p-1)}\right]^{1-1/2p},
\end{split}\end{align}
where in the first line we used that $|G_{ij}^{(o)}|\lesssim 1$ for $\cG\in \Omega$; in the second line we used H\"older's inequality; in the last line we used Jensen's inequality and $\bP(\oOmega\setminus \Omega)\lesssim N^{-\fd+3}$, which follows  from Proposition \ref{thm:prevthm}.
Recall the expression of $\Pi_p$ from \eqref{e:defUG}, we conclude that 
\begin{align*}
|\eqref{e:ffbbcopy}|\lesssim  \bE[\Psi_p],
\end{align*}
provided we take that $\fd\geq 2p+3$.

By Lemma \ref{lem:exchangeablepair}, we can rewrite the first term on the righthand side of \eqref{e:GUmaint1}
\begin{align}\begin{split}\label{e:IIGU}
    \bE\left[I(o,\cG)I(o,\wt\cG){\bm1(\tcG\in \Omega)}|G^{(o)}_{ij}|^2 \Pi_p\right]
    &=\bE\left[I(o,\cG)I(o,\wt\cG){\bm1(\cG\in \Omega)}|\wt G^{(o)}_{ij}|^2  \wt\Pi_p\right]\\
    &\lesssim \bE\left[I(o,\cG) |\wt G^{(o)}_{ij}|^2  \Pi_p\right],
\end{split}\end{align}
where with overwhelmingly high probability,
\begin{align*}
    \bm1(\cG\in \Omega) \wt \Pi_p
    =\bm1(\cG\in \Omega) \bm1(\tcG\in \Omega)(|Q(z;\tcG)-Y(Q(z;\tcG))|+\wt \Phi(z))^{2p-1}\lesssim \Pi_p.
\end{align*}

In the following we estimate the righthand side of \eqref{e:IIGU}. We denote $\cT=\cB_\ell(o,\cG)$ and its vertex set
$\bT$, and $P:=P(\cT,z,\msc)$. We notice that since $i,j$ are distinct neighboring vertices of $o$, $i,j$ are in different connected components of $\cT^{(o)}$ Thus $P^{(o)}_{ij}=0$, and $\wt G^{(o)}_{ij}=\wt G^{(o)}_{ij}-P^{(o)}_{ij}$. We will use the same argument as in the proof of \Cref{lem:diaglem}.    In the rest, we restrict to the event that $\bfS\in \sfF(\cG)$ from \Cref{lem:goodresamplingdata}, then
\begin{align}\label{e:GooY2}
    \wt G^{(o)}_{ij}=\left(P^{(o)}\sum_{k=1}^\fb \left(\widetilde B^\top (\widetilde G^{(\bT)}-\msc\mathbb I)\widetilde B P^{(o)}\right)^k\right)_{ij}+\OO(N^{-2}).
\end{align}
For any $1\leq k\leq \fb$, the above expression is a sum of terms in the following form
\begin{align}\label{e:PUP2}
   \sum_{\al_1, \al_2,\cdots,\al_{2k}\in \qq{\mu}}P^{(o)}_{i l_{\al_1}} V_{c_{\al_1} c_{\al_2}} P^{(o)}_{l_{\al_2} l_{\al_3}} V_{c_{\al_3} c_{\al_4}} 
   P^{(o)}_{l_{\al_4} l_{\al_5}}\cdots V_{c_{\al_{2k-1} }c_{\al_{2k}}}P^{(o)}_{l_{\al_{2k}} j},
\end{align}
where $ V_{c_{\al_{2j-1}}c_{\al_{2j}}}$ is one of the two terms $(\wt G_{c_\al c_\al}^{(\bT)}-\msc)$ or $\wt G_{c_\al c_\beta}^{(\bT)}$.

Next, we show that the summand in \eqref{e:PUP2} is zero, unless there exists some $\al_{2j-1}\neq \al_{2j}$. Otherwise the summand simplifies to 
\begin{align}\label{e:oneterma}
    P^{(o)}_{i l_{\al_1}} V_{c_{\al_1} c_{\al_1}} P^{(o)}_{l_{\al_1} l_{\al_3}} V_{c_{\al_3} c_{\al_3}} 
   P^{(o)}_{l_{\al_3} l_{\al_5}}\cdots V_{c_{\al_{2k-1} }c_{\al_{2k-1}}}P^{(o)}_{l_{\al_{2k-1}} j}.
\end{align}
For the sequence of indices $i, l_{\al_1}, l_{\al_3},\cdots, l_{\al_{2k-1}}, j$, there exists some adjacent pairs, say $(l_{\al_{2j-1}}, l_{\al_{2j+1}})$, which are in different connected components of $\cT^{(o)}$. Then $P^{(o)}_{l_{\al_{2j-1}} l_{\al_{2j+1}}}=0$, and the term in \eqref{e:oneterma} vanishes.   Consequently, for the summation \eqref{e:PUP2}, we can restrict to the case that there exists some $\al_{2j-1}\neq \al_{2j}$. By the same argument as in the proof of \Cref{lem:diaglem}, using \eqref{e:sumal_beta} and Theorem \ref{thm:prevthm}, we can bound \eqref{e:PUP2} as
\begin{align}\label{e:firstbound}
    |\eqref{e:PUP2}|\lesssim \frac{1}{(d-1)^\ell}\sum_{\al\neq \beta}|\wt G_{c_\al c_\beta}^{(\bT)}|,
\end{align}
and by plugging \eqref{e:GooY2},\eqref{e:PUP2} and \eqref{e:firstbound} back into \eqref{e:IIGU} we conclude that
\begin{align}\label{e:IIGU2}
    \bE\left[I(o,\cG) |\wt G^{(o)}_{ij}|^2  \Pi_p\right]\lesssim \frac{1}{(d-1)^\ell}\sum_{\al\neq \beta}\bE\left[I(o,\cG) |\wt G^{(\bT)}_{c_\al c_\beta}|^2  \Pi_p\right].
\end{align}
Next, we will rewrite $|\wt G_{c_\al c_\beta}^{(\bT)}|$ back into the Green's function of the graph $\cG$. As we do so, any term that we can bound as $\OO(\Phi+N^{-1+\fc})$ is negligible, as the resulting term $\bE\left[{\bm1(\cG\in \Omega)}(\Phi+N^{-1+\fc})  \Pi_p\right]=\OO(\bE[\Psi_p])$.
We denote $\bW:=\{b_\alpha\}_{\alpha\in \qq{\mu}}$. We can reduce $\wt G^{(\bT)}$ to a function of $G$ by deleting $\bW$. This is because as $\bT$ is already deleted, deleting $\bW$ will remove all other edges affected by the switch. To see the effect, we define a new matrix $\widetilde  B:=\widetilde H_{\bW \bW^\complement}$, which is the adjacency matrix of edges from $\bW$ to $\bW^{\complement}$ in the switched graph $\wt\cG$. We can then do the decomposition using Schur complement formula
\begin{align}\begin{split}\label{eq:tildeexpansion1}
\wt G^{(\bT)}|_{\bW^\complement}&=\wt G^{(\bT \bW)}+\wt G^{(\bT \bW)}\widetilde {B}\wt G^{(\bT)}|_\bW \widetilde {B}^\top  \wt G^{(\bT \bW)}\\
&=G^{(\bT)}|_{\bW^\complement}-G^{(\bT)}(G^{(\bT)}|_{\bW})^{-1}G^{(\bT)}+G^{(\bT \bW)}\widetilde {B}\wt G^{(\bT)}|_\bW \widetilde {B}^\top G^{(\bT \bW)}.
\end{split}\end{align}

Thanks to \Cref{c:expectationbound}, the contributions from the first two terms in \eqref{eq:tildeexpansion1}, are bounded by $\OO(\Phi)$. 

For the last term in \eqref{eq:tildeexpansion1} we write it explicitly as
\begin{align}\label{e:GBGBG}
(G^{(\bT \bW)}\widetilde {B}\wt G^{(\bT)}|_\bW \widetilde {B}^\top G^{(\bT \bW)})_{c_\alpha c_\beta}
=\sum_{\gamma,\gamma'\in \qq{\mu}}(G^{(\bT \bW)}\widetilde B)_{c_\alpha b_\gamma}(\wt G^{(\bT)})_{b_{\gamma}b_{\gamma'}} (G^{(\bT \bW)}\widetilde B)_{c_\beta b_{\gamma'}},
\end{align}

For the first term in the summand of \eqref{e:GBGBG}, it is given by 
\begin{align}\label{e:GtBexp}
    (G^{(\bT \bW)}\widetilde B)_{c_\alpha b_\gamma}
    =\frac{1}{\sqrt{d-1}} \left(G^{(\bT \bW)}_{c_\al a_\gamma}+\sum_{x\sim b_\gamma\atop x\neq c_\gamma} G^{(\bT \bW)}_{c_\al x}\right)
\end{align}
where the adjacency relation $\sim$ is with respect to the graph $\cG$. Plugging \eqref{e:GtBexp} into \eqref{e:GBGBG},  and using \Cref{c:expectationbound}, we have 
\begin{align*}
    \bE_{\bfS}\left[\left|\sum_{\gamma,\gamma'\in \qq{\mu}\atop \gamma\neq \al}(G^{(\bT \bW)}\widetilde B)_{c_\alpha b_\gamma}(\wt G^{(\bT)})_{b_{\gamma}b_{\gamma'}} (G^{(\bT \bW)}\widetilde B)_{c_\beta b_{\gamma'}}\right|^2\right]\lesssim \Phi.
\end{align*}
A similar estimate holds for $\gamma'\neq \beta$. Thus we can reduce \eqref{e:GBGBG} to the sum where $\al=\gamma, \beta=\gamma'$ 
\begin{align}
    \label{e:GtGG}
\frac1{d-1}\left(G_{c_\al a_\al}^{(\bT \bW)}+\sum_{
x\sim b_\alpha\atop
x\neq c_\alpha} G^{(\bT \bW)}_{c_\alpha x}\right)\wt G^{(\bT)}_{b_\alpha b_\beta}
\left(G_{c_\beta a_\beta}^{(\bT \bW)}+
\sum_{
 y\sim b_\beta\atop
y\neq c_\beta} G^{(\bT \bW)}_{yc_\beta}\right).
\end{align} 
Next, we show that we can replace $G_{c_\al x}^{(\bT \bW)},G_{y c_\beta}^{(\bT \bW)}$ in the above expression by $G_{c_\al x}^{(b_\al)}, G_{yc_\beta}^{(b_\beta)}$, and further reduce \eqref{e:GtGG} to 
\begin{align}\begin{split}\label{e:GtG2}
&\phantom{{}={}}\bE_\bfS\left|\left(G_{c_\al a_\al}^{(\bT \bW)}+\sum_{
x\sim b_\alpha\atop
x\neq c_\alpha} G^{(\bT \bW)}_{c_\alpha x}\right)\wt G^{(\bT)}_{b_\alpha b_\beta}
\left(G_{c_\beta a_\beta}^{(\bT \bW)}+
\sum_{
 y\sim b_\beta\atop
y\neq c_\beta} G^{(\bT \bW)}_{yc_\beta}\right)\right|^2\\
&\lesssim \bE_\bfS\left|\sum_{{
x\sim b_\alpha, y\sim b_\beta\atop
x\neq a_\alpha, y\neq a_\beta}} G^{(b_\alpha)}_{c_\alpha x}\wt G^{(\bT)}_{b_\alpha b_\beta}G^{(b_\beta)}_{yc_\beta}\right|^2+\OO(\Phi),
\end{split}\end{align}
where the two terms $G_{c_\al a_\al}^{(\bT \bW)}$ and $G_{c_\beta a_\beta}^{(\bT \bW)}$ are negligible thanks to \Cref{c:expectationbound}. To replace $G_{c_\al x}^{(\bT \bW)}$ by $G_{c_\al x}^{(b_\al)}$, we use the following Schur complement formula
\begin{align*}
G^{(\bT \bW)}_{c_\alpha x}=G^{(b_\alpha)}_{c_\alpha x}-\sum_{u,v\in \bT \bW\backslash b_\alpha}G^{(b_\alpha)}_{c_\alpha u}(G|_{\bT \bW}^{(b_\alpha)})^{-1}_{uv}G^{(b_\alpha)}_{v x},
\end{align*}
then the last term can be bounded using the Ward identity
\begin{align}\label{e:warduse}
    \bE_\bfS[|G^{(b_\alpha)}_{c_\alpha u}|^2]
    \lesssim 
    \bE_\bfS\left[|G_{c_\alpha u}|^2+\frac{|G_{c_\al b_\al}G_{b_\al u}|^2}{|G_{b_\al b_\al}|^2}\right]\lesssim \Phi.
\end{align}

Finally for \eqref{e:GtG2}, we notice that $c_\alpha,x$ are distinct neighbors of $b_\alpha$; and $y,c_\beta$ are distinct neighbors of $b_\beta$. We use the upper bound $\varepsilon$ from Theorem \ref{thm:prevthm} on $G_{yc_\beta}^{(b_\beta)}$ and the loose bound of $|\widetilde G_{b_\alpha b_\beta}^{(\bT)}|\lesssim 1$, which gives
\begin{align*}
&\phantom{{}={}}\bE  I(o,\cG)\bE_\bfS\left[\left|\sum_{\substack{
x\sim b_\alpha, y\sim b_\beta\\
x\neq a_\alpha, y\neq a_\beta}} G^{(b_\alpha)}_{c_\alpha x}\wt G^{(\bT)}_{b_\alpha b_\beta}G^{(b_\beta)}_{yc_\beta}\right|^2 \Pi_p\right]\\
&\lesssim \sum_{
x\sim b_\alpha
x\neq a_\alpha}\varepsilon^2 \bE[ I(o,\cG)\bE_\bfS[|G^{(b_\al)}_{c_\al x}|^2]\Pi_p]\\
&\lesssim \sum_{
x\sim b_\alpha
x\neq a_\alpha}\varepsilon^2 \bE[ \bE_\bfS[I(c_\al,\cG)|G^{(b_\al)}_{c_\al x}|^2]\Pi_p] +N^{-1+\fc}\bE[\Pi_p]\\
&\lesssim \varepsilon^2 \bE[ I(o,\cG)|G^{(o)}_{ij}|^2\Pi_p] +\bE[\Psi_p]
\end{align*}
where in the third line we used that $\bE_\bfS[1-I(c_\al, \cG)]\leq N^{-1+\fc}$; in the last line, we used the permutation invariance of the vertices, so that $G_{ij}^{(o)}$ and $G_{c_\al x}^{(b_\al)}$ have the same distribution. 

Thus combining the discussion above, we arrive at the following bound 
    \begin{align}
\bE[I(o, \cG)|G_{ij}^{(o)}|^2\Pi_p]\lesssim \bE[\Psi_p]+ \varepsilon^2\bE[I(o,\cG)|G_{ij}^{(o)}|^2\Pi_p],
\end{align}
and the claim \eqref{e:sameasdisconnect} follows from rearranging.  

\end{proof}

As a consequence of \Cref{lem:deletedalmostrandom}, the following proposition states that during the local resampling, the errors $\cE$ from \Cref{lem:diaglem} and \Cref{lem:offdiagswitch} (after averaging) are negligible. 

\begin{proposition}
\label{lem:task2}
   Fix $z\in {\mathbf D}$, a $d$-regular graph $\cG\in \Omega(z)$ and $\bT$ the vertex set of $\cT=\cB_\ell(o,\cG)$. Let $\Pi_p(z)$ be as in \eqref{e:defUG}.   We denote the indicator function $I(o,\cG)=1$ if for any $v\in \cB_\ell(o, \cG)$, $v$ has a radius $\fR$ tree neighborhood. We denote $T_{\bf S}$ the local resampling around $o$ with resampling data $\bf S$, and $\widetilde G=T_{\bf S}G$. Then for any $\al, \beta\in \qq{\mu}$, the following holds
\begin{align}\begin{split}\label{eq:task2}
    &\frac{1}{N}\sum_{u\not\in \bT}\bE\bE_\bfS[I(o,\cG)|\widetilde G^{(\bT\cup \mathbb X)}_{c_\alpha c_\beta}-G_{c_\alpha c_\beta}^{(b_\alpha b_\beta)}|\Pi_p]\lesssim\bE[\Psi_p],\quad \mathbb X=\emptyset \text{ or } \{u\},\\
    &\frac{1}{N}\sum_{u\sim v\not\in \bT}\bE\bE_\bfS[I(o,\cG)|\widetilde G^{(\bT\cup \mathbb X)}_{c_\alpha v}-G_{c_\alpha v}^{(b_\alpha\cup \mathbb X)}|\Pi_p]\lesssim \bE[\Psi_p],\quad \mathbb X=\emptyset \text{ or } \{u\},\\
   &\frac{1}{N^2}\sum_{ u\sim v\not\in \bT}\sum_{ u'\sim v'\not\in \bT}\bE\bE_\bfS[I(o,\cG)|\widetilde G^{(\bT u u')}_{v v'}-G^{u u'}_{v v'}|\Pi_p]\lesssim \bE[\Psi_p],
   \end{split}
    \end{align} 
    where the adjacency relation $\sim$ is with respect to the graph $\cG$. 
\end{proposition}

\begin{proof}[Proof of Proposition \ref{lem:task2}]
We will only prove the first statement in \eqref{eq:task2} with $\al\neq \beta$ and $\mathbb X=\emptyset$ as the others are similar. 
Thanks to the Schur complement formula, similar to \eqref{eq:tildeexpansion1}, we have
\begin{align}\label{e:tGt}
\wt G^{(\bT)}-G^{(b_\alpha b_\beta)}=-G^{(b_\alpha b_\beta)}(G^{(b_\alpha b_\beta)}|_{\bT \bW\backslash \{b_\alpha b_\beta\}})^{-1}G^{( b_\alpha b_\beta)}+G^{(\bT \bW)}\widetilde{B}\wt G^{(\bT)}|_\bW \widetilde{B}^\top G^{(\bT \bW)}.
\end{align}
where $\bW:=\{b_\alpha\}_{\alpha\in \qq{\mu}}$ and  $\widetilde  B:=\widetilde H_{\bW \bW^\complement}$.
Therefore, for $\cG\in \Omega$, we have
\begin{align}\begin{split}\label{e:decomp}
&\phantom{{}={}}\bE_\bfS[I(o,\cG)|\wt G^{(\bT)}_{c_\alpha c_\beta}-G_{c_\alpha c_\beta}^{(b_\alpha b_\beta)}|]\\
&\leq\bE_\bfS[I(o,\cG)|(G^{(b_\alpha b_\beta)}(G^{(b_\alpha b_\beta)}|_{\bT \bW\backslash \{b_\alpha b_\beta\}})^{-1}G^{( b_\alpha b_\beta)})_{c_\alpha c_\beta}|]\\
&+\bE_\bfS[I(o,\cG)|(G^{(\bT \bW)}\widetilde{B}\wt G^{(\bT)}|_\bW \widetilde{B}^\top G^{(\bT \bW)})_{c_\alpha c_\beta}|].
\end{split}\end{align}
We consider the two error terms separately.  By a Cauchy-Schwarz inequality, 
\begin{align}\begin{split}\label{e:t1}
&\phantom{{}={}}\bE_\bfS\bigg[I(o,\cG)|(G^{(b_\alpha b_\beta)}(G^{(b_\alpha b_\beta)}|_{\bT \bW\backslash \{b_\alpha b_\beta\}})^{-1}G^{( b_\alpha b_\beta)})_{c_\alpha c_\beta}|\bigg]\\
&\lesssim \sum_{x,y\in \bT\bW\setminus\{b_\al, b_\beta\}}\bE_\bfS[I(o,\cG)|G^{(b_\alpha b_\beta)}_{c_\alpha x}|^2|(G^{(b_\alpha b_\beta)}|_{\bT \bW\setminus \{b_\alpha b_\beta\}})_{xy}^{-1}|]\\
&\phantom{{}\lesssim \sum_{x,y\in \bT\bW\setminus\{b_\al, b_\beta\}}{}}+\bE_\bfS[I(o,\cG)|G^{(b_\al b_\beta )}_{y c_\beta}|^2(G^{(b_\al b_\beta)}|_{\bT \bW\setminus \{b_\alpha b_\beta\}})_{xy}^{-1}|]\\
&\lesssim 
\sum_{x,y\in \bT\bW\setminus\{b_\al, b_\beta\}}\bE_\bfS[I(o,\cG)|(G^{(b_\alpha b_\beta)}_{c_\alpha x}|^2 +|G^{(b_\al b_\beta )}_{y c_\beta}|^2)]\lesssim \Phi,
\end{split}\end{align}
where in the last line we used that $\cG\in \Omega$, and the row and column sums of $(G^{(b_\alpha b_\beta)}|_{\bT W\backslash b_\alpha b_\beta})^{-1}$ are of order $\OO(1)$.  Then, using a Ward identity as in \eqref{e:warduse}, the above sum is bounded by  $\OO(\Phi)$ as desired.

For the second error term in \eqref{e:tGt}, we need to bound
\begin{align}\label{e:lastterm}
\bE_\bfS\left[I(o,\cG)|(G^{(\bT \bW)}\widetilde{B}\wt G^{(\bT )}|_{\bW } \widetilde{B}^\top G^{(\bT \bW)})_{c_\alpha c_\beta}|\right].
\end{align}
By the same analysis as for \eqref{e:GBGBG}, we can expand this term out and use the Cauchy-Schwarz inequality
\begin{align*}
\bE_\bfS[|(G^{(\bT \bW)}\widetilde{B}\wt G^{(\bT )}|_{\bW } \widetilde{B}^\top G^{(\bT \bW)})_{c_\alpha c_\beta}]\lesssim \sum_{\gamma\in \qq{\mu}}\bE_\bfS[|(G^{(\bT \bW)}\widetilde{B})_{c_\al b_\gamma}|^2] 
+\bE_\bfS[|(G^{(\bT \bW)}\widetilde{B})_{c_\beta b_\gamma}|^2].
\end{align*}
The two terms above can be studied in the same way, in the following we will focus on the first term. If $\gamma\neq \al$, thanks to \Cref{c:expectationbound}, we have $\bE_\bfS[|(G^{(\bT \bW)}\widetilde{B})_{c_\al b_\gamma}|^2] \lesssim \Phi$. For $\gamma=\al$, then by the same argument as for \eqref{e:GtG2}, we can show that
\begin{align*}
    \bE_\bfS[|(G^{(\bT \bW)}\widetilde{B})_{c_\al b_\al}|^2]\lesssim \sum_{x\sim b_{\alpha}\atop x\neq b_\al}\bE_\bfS[|G_{c_\alpha x}^{(b_\alpha)}|^2]+\OO(\Phi).
\end{align*}
Combining the above estimates, and taking expectation on both sides of \eqref{e:lastterm}, we conclude that
\begin{align}\begin{split}\label{e:lastterm1}
    &\phantom{{}={}}\bE\left[I(o,\cG)\bE_\bfS[|(G^{(\bT \bW)}\widetilde{B}\wt G^{(\bT )}|_{\bW } \widetilde{B}^\top G^{(\bT \bW)})_{c_\alpha c_\beta}]|\Pi_p\right]\\
    &\lesssim \sum_{
x\sim b_{\alpha}\atop x\neq b_\al}\bE[\bE_\bfS[|G_{c_\alpha x}^{(b_\alpha)}|^2]\Pi_p]+\bE[\Phi \Pi_p]\\
&\lesssim \sum_{
x\sim b_{\alpha}\atop x\neq b_\al}\bE[\bE_\bfS[I(c_\al, \cG)|G_{c_\alpha x}^{(b_\alpha)}|^2]\Pi_p]+\bE[\Phi \Pi_p]\lesssim \bE[\Psi_p],
\end{split}\end{align}
where in the third inequality, we used that with probability at least $1-N^{\fc-1}$, $I(c_\al, \cG)=1$; in the last inequality we used \Cref{lem:deletedalmostrandom} and the permutation invariance of the vertices.

The claim \eqref{eq:task2} follows from combining  \eqref{e:decomp}, \eqref{e:t1} and \eqref{e:lastterm1}.

\end{proof}

We now give the general ways of bounding the expectation of generic terms (recall from \Cref{def:pgen}.

\begin{proposition}\label{p:small_Ri}
    Given a generic term $R_\bfi$ as in \Cref{def:pgen}, and recall $\Pi_p(z)$ from \eqref{e:defUG}, then for any function $V(Q,m,z)$ of $Q(z), m(z), z$, the following holds
    \begin{align}\label{e:rough_bound}
         \frac{1}{Z_\cF}\sum_{\bfi}\bE[\bm1(\cG\in \Omega)|V(Q,m,z)||R_\bfi| ]\lesssim \bE[ |V(Q,m,z)| \Pi_p].
    \end{align}
    Moreover, if $R_\bfi$ satisfies one of the following two conditions
    \begin{enumerate}
        \item \label{i1:small}$R_\bfi$ contains two terms in the form: $(d-1)^{2\ell}G_{cc'}^{(bb')}$ with $(c,b)\neq (c',b')\in \cC^\circ$;
       \item \label{i1:smallcase2}$R_\bfi$ contains a term $(d-1)^{2\ell}G_{cc'}^{(bb')}$ with 
     $(c,b)\neq (c',b')\in \cC^\circ$, and a Green's function term $(d-1)^{2\ell}(G_{ss'}-P_{ss'})$ with $s,s'\in \cK$ in different connected components of $\cF$;
        \item \label{i2:small}$R_\bfi$ contains a factor $(d-1)^{2\ell} G_{cc'}^{(bb')}$ with $(c,b)\neq (c',b')\in \cC^\circ$, and $\cV\neq\emptyset$.
    \end{enumerate}
    then 
    \begin{align}\label{e:refined_bound}
        \frac{1}{Z_\cF}\sum_{\bfi}\bE[I(\cF, \cG)\bm1(\cG\in \Omega)R_\bfi ]=\OO(\bE[\Psi_{p}]).
    \end{align}
\end{proposition}

\begin{proof}

Based on \Cref{def:pgen}, assume that there are $k$ special edges $\cV=\{(u_i, v_i)\}_{1\leq i\leq k}$, for $0\leq k\leq 2p-1$. Then we can write
\[
R_\bfi=(Q-Y_\ell(Q))^{2p-1-k}(1-Y'_\ell(Q))^{k}\left(\prod_{i=1}^{k} G_{s_iw_i}G_{s'_iw_i'}\right) W,
\]
for some $s_i,s_i'\in \cK$ and $w_i,w_i'\in \{u_i,v_i\}$ corresponding to the special edge $(u_i,v_i)$, and $W$ is the product of other factors. For $\cG\in \Omega$, thanks to \Cref{thm:prevthm} we have $|W|\lesssim 1$.  Then we have
\begin{align}\begin{split}\label{e:Eexp}
&\phantom{{}={}}\frac{1}{Z_\cF}\sum_{\bfi}\bE[\bm1(\cG\in \Omega)|V||R_\bfi| ]\\
&=\frac{1}{Z_\cF}\sum_{\bfi}\bE\left[\bm1(\cG\in \Omega)|V|\left|(Q-Y_\ell(Q))^{2p-1-k}(1-Y'_\ell(Q))^{k}\prod_{i=1}^{k} G_{s_iw_i}G_{s'_iw_i'}W\right|\right]\\
&\lesssim\frac{1}{Z_\cF}\sum_{\bfi}\bE[\bm1(\cG\in \Omega)|V||Q-Y_\ell(Q)|^{2p-1-k}\prod_{i=1}^{k} (|G_{s_iw_i}|^2+|G_{s'_iw_i'}|^2)] .
\end{split}\end{align}
We can then first sum over the special edges, thanks to the Ward identity, each special edge gives a factor $N^\fc\Im[m]/(N\eta)=\Phi$. Then we sum over the remaining indices in $\bfi$. This gives \eqref{e:rough_bound}
    \begin{align*}
        \frac{1}{N}\frac{1}{Z_\cF}\sum_{\bfi}\bE[\bm1(\cG\in \Omega)|V||R_\bfi| ]\lesssim \frac{1}{N}\bE\left[\bm1(\cG\in \Omega)|V|\left(|Q-Y_\ell(Q)|+\Phi\right)^{2p-1}\right].
    \end{align*}

Next we prove \eqref{e:refined_bound}  under assumption \Cref{i1:small}. We start with the Schur complement formula
\begin{align}\label{e:scf}
 G_{cc'}^{(bb')}=G_{cc'}+(G (G|_{\{bb'\}})^{-1} G)_{cc'}.
\end{align}
For the last term in \eqref{e:scf}, it consists of four terms
\begin{align}\label{e:scf2}
   G_{cb} (G|_{\{bb'\}})_{bb'}^{-1} G_{b'c'} 
   +G_{cb} (G|_{\{bb'\}})_{bb}^{-1} G_{bc'} 
   +G_{cb'} (G|_{\{bb'\}})_{b'b}^{-1} G_{bc'} 
   +G_{cb'} (G|_{\{bb'\}})_{b'b'}^{-1} G_{b'c'} 
\end{align}
For $\cG\in \Omega$, thanks to \Cref{thm:prevthm}, we conclude that
\begin{align*}
|G_{cc'}^{(bb')}|\lesssim |G_{bb'}|+|G_{bc'}|+|G_{b'c}|+|G_{cc'}|,
\end{align*}
and the Ward identity leads to 
\begin{align}\begin{split}\label{e:ccbb}
   \frac{1}{(Nd)^2}\bm1(\cG\in \Omega) \sum_{c\sim b\atop c'\sim b'}|G_{cc'}^{(bb')}|^2
   &\lesssim 
   \frac{1}{(Nd)^2}\bm1(\cG\in \Omega) \sum_{c\sim b\atop c'\sim b'}(|G_{bb'}|^2+|G_{bc'}|^2+|G_{b'c}|^2+|G_{cc'}|^2)\\
   &\lesssim \bm1(\cG\in \Omega)\Phi.
\end{split}\end{align}
Under assumption \Cref{i1:small} that $R_\bfi$ contains two Green's function terms $G_{c_1c'_1}^{(b_1b'_1)}$, $G_{c_2c'_2}^{(b_2b'_2)}$, where $(b_1,c_1)$, $(b_1',c_1')$, $(b_2,c_2),(b_2',c_2')\in \cK$. The same as in \eqref{e:Eexp}, we have 
\begin{align}\begin{split}\label{e:Eexp2}
&\phantom{{}={}}\frac{1}{Z_\cF}\left|\sum_{\bfi}\bE[I(\cF, \cG)\bm1(\cG\in \Omega)R_\bfi ]\right|\\
&\lesssim\frac{1}{Z_\cF}\left|\sum_{\bfi}\bE\left[I(\cF, \cG)\bm1(\cG\in \Omega)|Q-Y_\ell(Q)|^{2p-1-k}(|G_{c_1c'_1}^{(b_1b'_1)}|^2+|G_{c_2c'_2}^{(b_2b'_2)}|^2)\prod_{i=1}^{k} (|G_{s_iw_i}|^2+|G_{s'_iw_i'}|^2)\right] \right|.
\end{split}\end{align}
Then can then first sum over the special edges, which gives a factor $\Phi^k$. Then thanks to \eqref{e:Eexp2},  the average over $b_1,b_1', c_1, c_1, b_2, b_2', c_2, c_2'$ gives another factor of $\Phi$. Finally we sum over the remaining indices in $\bfi$. This gives \eqref{e:refined_bound}:
\begin{align*}
        \frac{1}{Z_\cF}\sum_{\bfi}\bE[I(\cF, \cG)\bm1(\cG\in \Omega)|R_\bfi| ]\lesssim \bE\left[\bm1(\cG\in \Omega)\Phi\left( \Phi+|Q-Y_\ell(Q)|\right)^{2p-1}\right]=\OO(\bE[\Psi_{p}]).
    \end{align*}

The statement under assumption \Cref{i1:smallcase2} can be proven in the same way as above, so we omit.

Next we prove \eqref{e:refined_bound}  under assumption \Cref{i2:small} that $R_\bfi$ contains a term $(d-1)^{2\ell}G^{(b_1 b_2)}_{c_1 c_2}$ with $(b_1,c_1)\neq (b_2,c_2)\in \cC^\circ$. We define $\cX_1=(b_1, c_1)$ and  $\cX_2=(b_2,c_2)$. They are two distinct connected components of $\cF$. If there is another Green's function term $(d-1)^{2\ell}(G_{ss'}-P_{ss'})$ with $s\in \cX_1,s'\in \cX_2$, or another copy $(d-1)^{2\ell}G^{(b_1b_2)}_{c_1c_2}$, then the same argument as in \Cref{i1:small} and \Cref{i1:smallcase2} gives \eqref{e:refined_bound}.

We then consider the case that $(d-1)^{2\ell}G^{(b_1b_2)}_{c_1c_2}$ is the unique term which contains indices in both $\cX_1$ and $\cX_2$.  We will do this by categorizing factors of $R_\bfi$ based on their relationship with $\cX_1$ and $\cX_2$.  To see how to categorize factors, we consider $U, V$ as vectors with entries $U(b,c)$ and $V(b',c')$, 
\begin{align}
\begin{split}\label{eq:spectralbound}
\left|\frac1{N^2d^2}\sum_{b\sim c\atop b'\sim c'}U(b,c) G_{cc}^{(bb')}V(b',c')\right|&=
\left|\frac1{N^2d^2}\sum_{b\sim c\atop b'\sim c'}U(b,c) (G_{cc'}+(G (G|_{\{bb'\}})^{-1} G)_{cc'})V(b',c')\right|,
\end{split}
\end{align}
where the adjacency $\sim$ is with respect to the graph $\cG$.
Because $\|G(E+\ri \eta)\|_{\spec}\leq\eta^{-1}$, then 
\begin{align}\label{e:UGV}
    \left|\frac1{N^2d^2}\sum_{b\sim c\atop b'\sim c'}U(b,c) G_{cc'}V(b',c')\right|
    \lesssim 
    \frac{1}{N^2\eta} \sqrt{\sum_{b\sim c}|U(b,c)|^2 \sum_{b'\sim c'} |V(b',c'))|^2}.
\end{align}
Using the expression \eqref{e:scf2} and \Cref{thm:prevthm}, the summation in \eqref{eq:spectralbound} involving the term $(G (G|_{\{bb'\}})^{-1} G)_{cc'}$ gives the same bound as in \eqref{e:UGV}. And we conclude that
\begin{align}\label{e:normbound}
    \left|\frac1{N^2d^2}\sum_{b\sim c\atop b'\sim c'}U(b,c) G_{cc}^{(bb')}V(b',c')\right|\lesssim \frac{1}{N^2\eta} \sqrt{\sum_{b\sim c}|U(b,c)|^2 \sum_{b'\sim c'} |V(b',c'))|^2}.
\end{align}

Denote the forest $\widehat \cF$ from $\cF$ by removing $\{(b_1,c_1), (b_2,c_2)\}$ and special edges $\cV$,
\begin{align*}
    \widehat \cF=(\widehat \bfi, \widehat E)=\cF\setminus \left( \{(b_1,c_1), (b_2,c_2)\}\cup\{ (u_i, v_i)\}_{1\leq i\leq k}\right).
\end{align*}
Next we show that we can replace the indicator function $I(\cF,\cG)$ by $I(\widehat\cF, \cG)A_{b_1c_1}A_{b_2c_2}\prod_{i=1}^k A_{u_i v_i}$, and the error is negligible: 
\begin{align}\begin{split}\label{e:relaxbc}
    &\phantom{{}={}}\frac{1}{Z_\cF}\sum_{\bfi}\bE[I(\cF, \cG)\bm1(\cG\in \Omega)R_\bfi ]\\
    &= \frac{1}{Z_\cF}\sum_{\bfi}\bE\left[\bm1(\cG\in \Omega)I(\widehat\cF, \cG)A_{b_1c_1}A_{b_2c_2}\prod_{i=1}^k A_{u_i v_i} R_\bfi \right]+\OO(\bE[\Psi_p]).
\end{split}\end{align}
We denote that
\begin{align*}
\Delta I&:= I(\widehat\cF, \cG)A_{b_1c_1}A_{b_2c_2}\prod_{i=1}^k A_{u_i v_i} -I(\cF,\cG)\\
&=\left(I(\widehat\cF, \cG)A_{b_1c_1}A_{b_1c_2}- I(\widehat \cF', \cG)\right)\prod_{i=1}^k A_{u_i v_i},
\end{align*}
where $\widehat \cF'=\cF\setminus \{(u_i,v_i)\}_{i\in \qq{k}}$.
To show \eqref{e:relaxbc}, we first notice that  
\begin{align}\begin{split}\label{e:relaxbc2}
    &\phantom{{}={}}\frac{1}{Z_\cF}\sum_{\bfi}\bE\left[|\Delta I|\bm1(\cG\in \Omega)|R_\bfi| \right]\\
    &\lesssim 
    \frac{1}{Z_\cF}\sum_{\bfi}\bE\left[|\Delta I|\bm1(\cG\in \Omega)|Q-Y_\ell(Q)|^{2p-1-k}\prod_{i=1}^{k} (|G_{s_iw_i}|^2+|G_{s'_iw_i'}|^2) \right]\\
    &\lesssim 
    \frac{1}{Z_{\widehat \cF}}\sum_{\widehat\bfi}\frac{1}{(Nd)^2}\sum_{b_1,c_1, b_2, c_2}\bE\left[\left|I(\widehat\cF, \cG)A_{b_1c_1}A_{b_1c_2}- I(\widehat \cF', \cG)\right|\bm1(\cG\in \Omega)|Q-Y_\ell(Q)|^{2p-1-k}\Phi^k \right],
\end{split}\end{align}
where $Z_\cF=Z_{\widehat \cF}(Nd)^{2+k}$.
For the indicator function $|I(\widehat\cF, \cG)A_{b_1c_1}A_{b_1c_2}- I(\widehat \cF', \cG)|$ in the last line of \eqref{e:relaxbc2}, it equals one if there exists $v\in \cB_\ell(c_1, \cG)\cup \cB_\ell(c_2, \cG)$, such that $v$ does not have a radius $\fR$ tree neighborhood; or $c_1, c_2$ are within distance $3\fR$ to each other, or other core edges. By the same argument as in \Cref{lem:configuration}, we have
\begin{align}\label{e:indicator}
     \frac{1}{(Nd)^2}\sum_{b_1,c_1, b_2, c_2}|I(\widehat\cF, \cG)A_{b_1c_1}A_{b_1c_2}- I(\widehat \cF', \cG)|\bm1(\cG\in \Omega)
     \lesssim \frac{1}{N^{1-2\fc}}.
\end{align}
By plugging \eqref{e:indicator} into \eqref{e:relaxbc2}, and averaging over the indices $\widehat \bfi$, we conclude that
\begin{align*}
    \frac{1}{Z_\cF}\sum_{\bfi}\bE\left[|\Delta I|\bm1(\cG\in \Omega)|R_\bfi| \right]\lesssim \frac{1}{N^{1-2\fc}}\bE[{\bf1}(\cG\in \Omega) |Q-Y_\ell(Q)|^{2p-1-k}\Phi^k]\lesssim \bE[\Psi_p].
\end{align*}

Next, we study \eqref{e:relaxbc}. We define $U(b,c), V(b',c')$ as follows such that 
\begin{align}\label{e:bfi}
    R_{\bfi}=U(b_1,c_1)V(b_2,c_2)G^{(b_1 b_2)}_{c_1 c_2}R'_{\widehat\bfi},
\end{align}
where $R'_{\widehat\bfi}$ only depends on the indices $\widehat \bfi$. In particular it does not depends on indices $b_1, c_1, b_2, c_2$ and vertices in special edges.

For each special edge $(u_i, v_i)$, there are at least two factors in the form $(d-1)^{2\ell}G_{s s'}$ with $s\in \cK$ and $s'\in \{u_i, v_i\}$, and possibly some factors in the form $G_{u_i u_i},G_{u_i v_i}, 1/G_{v_i v_i}$. We divide the set $\qq{k}=\bV_0\cup\bV_1\cup \bV_2$:

\begin{enumerate}
    \item If all the factors $(d-1)^{2\ell}G_{s s'}$ for $s'\in \{u_i, v_i\}$ satisfies $s\not\in \{b_2, c_2\}$, we put $i\in \bV_1$, and all the Green's function terms associated with $(u_i, v_i)$ in $\widehat U(b_1, c_1)$.
    \item If $i\notin \bV_1$, and all the factors $(d-1)^{2\ell}G_{s s'}$ for $s'\in \{u_i, v_i\}$ satisfies $s\not\in \{b_1, c_1\}$, we put $i\in \bV_2$, and all the Green's function terms associated with $(u_i, v_i)$ in $\widehat U(b_2, c_2)$.
    \item For the remaining indices $i\in \bV_0$, there are some factors $(d-1)^{2\ell}G_{s s'}$ with $s\in \{b_1, c_1\}, s'\in \{u_i, v_i\}$, we put them in $\widehat U'(b_1,c_1)$; and some factors $(d-1)^{2\ell}G_{s s'}$ with $s\in \{b_2, c_2\}, s'\in \{u_i, v_i\}$ we put them in $\widehat V'(b_2,c_2)$. We also put all the factors  $G_{u_i u_i},G_{u_i v_i}, 1/G_{v_i v_i}$ in $\widehat U'(b_1,c_1)$.
\end{enumerate}
For terms involving only vertices $b_1,c_1$, i.e. $(d-1)^{2\ell}(G_{c_1c_1}^{(b_1)}-Q)$ and $(d-1)^{2\ell}(G_{b_1c_1}-P_{b_1c_1})$, we put them in $\widehat U(b_1,c_1)$; and terms involving only vertices $b_2,c_2$, we put them in $\widehat V(b_2,c_2)$. In this way all the factors involving the indices $b_1, c_1, b_2, c_2$ and vertices in special edges are collected in $U(b_1,c_1)=\widehat U(b_1,c_1) \widehat U'(b_1,c_1)$, and $V(b_1,c_1)=\widehat V(b_1,c_1) \widehat V'(b_1,c_1)$. We collect remaining terms in $R'_{\widehat \bfi}$, and this leads to \eqref{e:bfi}. 

Thanks to the Ward identity, and \Cref{thm:prevthm} we have 
\begin{align}\begin{split}\label{e:ward}
    &\sum_{u_i\sim v_i: i\in \bV_1} |\widehat U(b_1,c_1)|\lesssim (Nd \Phi)^{|\bV_1|},
    \quad 
    \sum_{u_i\sim v_i: i\in \bV_2} |\widehat V(b_2, c_2)|\lesssim (Nd \Phi)^{|\bV_2|},\\
    &\sum_{u_i\sim v_i: i\in \bV_0} |\widehat U'(b_1,c_1)|^2\lesssim (Nd \Phi)^{|\bV_0|},
    \quad 
    \sum_{u_i\sim v_i: i\in \bV_0} |\widehat V'(b_2, c_2)|^2\lesssim (Nd\Phi)^{|\bV_0|}.
\end{split}\end{align}

With the above notations, we can rewrite the righthand side of \eqref{e:relaxbc} as 
\begin{align}\begin{split}\label{e:sumbound1}
&\phantom{{}={}}\frac{1}{Z_\cF}\sum_{\bfi}\bE\left[\bm1(\cG\in \Omega)I(\widehat\cF, \cG)A_{b_1c_1}A_{b_2c_2}\prod_{i=1}^k A_{u_i v_i} R_\bfi \right]\\
  &= \frac{1}{Z_{\widehat\cF}}\sum_{\widehat\bfi}\bE\left[\frac{\bm1(\cG\in \Omega)I(\widehat\cF, \cG)}{(Nd)^{k+2}} \sum_{u_i\sim v_i:i\in \bV_0} R_{\widehat \bfi}\sum_{b_1\sim c_1\atop b_2\sim c_2}\left(\sum_{u_i\sim v_i:i\in \bV_1} U(b_1,c_1)\right)G_{c_1 c_2}^{(b_1b_2)}\left(\sum_{u_i\sim v_i:i\in \bV_2}  V(b_2,c_2)\right) \right].
\end{split}\end{align}
We can first sum over the special edges and $b_1, c_1, b_2, c_2$ using \eqref{e:normbound} and \eqref{e:ward}
\begin{align}\begin{split}\label{e:sumbound2}
&\phantom{{}={}}\frac{\bm1(\cG\in \Omega) }{(Nd)^{k+2}}\sum_{u_i\sim v_i: i\in \bV_0} \sum_{b_1\sim c_1\atop b_2\sim c_2}\left(\sum_{u_i\sim v_i: i\in \bV_1}U(b_1,c_1)\right)G_{c_1 c_2}^{(b_1b_2)}\left(\sum_{u_i\sim v_i: i\in \bV_2}V(b_2,c_2)\right)\\
 &\leq    \frac{\bm1(\cG\in \Omega) }{(Nd)^{k+2}}\sum_{u_i\sim v_i: i\in \bV_0}\frac{1}{\eta}\sqrt{\sum_{b_1\sim c_1}\left|\sum_{u_i\sim v_i: i\in \bV_1}U(b_1,c_1)\right|^2 \sum_{b_2\sim c_2} \left|\sum_{u_i\sim v_i: i\in \bV_2}V(b_2,c_2)\right|^2}\\
  &    \lesssim\frac{\bm1(\cG\in \Omega) }{(Nd)^{k+2}} \frac{1}{\eta}(Nd \Phi)^{|\bV_1|+|\bV_2|} \sum_{u_i\sim v_i: i\in \bV_0}  \left(\sum_{b_1\sim c_1}|\widehat U(b_1,c_1)|^2+\sum_{b_1\sim c_1}|\widehat U(b_1,c_1)|^2\right)\\
    &\lesssim \frac{\bm1(\cG\in \Omega) }{(Nd)^{k+2}}  \frac{1}{\eta}(Nd \Phi)^{|\bV_0|+|\bV_1|+|\bV_2|} = \frac{\bm1(\cG\in \Omega) }{\eta (Nd)^{2}}  \Phi^k.
\end{split}\end{align}
Then we can plug \eqref{e:sumbound2} into \eqref{e:sumbound1} to conclude
\begin{align}\begin{split}\label{e:sumbound3}
   &\phantom{{}={}}\frac{1}{Z_\cF}\sum_{\bfi}\bE\left[\bm1(\cG\in \Omega)I(\widehat\cF, \cG)A_{b_1c_1}A_{b_2c_2}\prod_{i=1}^k A_{u_i v_i} R_\bfi \right]\\
   &\lesssim 
    \frac{1}{N^2\eta}\frac{1}{Z_{\widehat \cF}}\sum_{\widehat \bfi}\bE\left[\bm1(\cG\in \Omega)I(\widehat\cF, \cG) |R_{\widehat \bfi}| \Phi^k \right]\\
    &\lesssim \frac{1}{N^2\eta}\frac{1}{Z_{\widehat \cF}}\sum_{\widehat \bfi}\bE\left[\bm1(\cG\in \Omega)I(\widehat\cF, \cG) |Q-Y_\ell(Q)|^{2p-1-k}|1-Y'_\ell(Q)|^k \Phi^k \right]\\
    &\lesssim \frac{1}{N^2\eta}\bE\left[\bm1(\cG\in \Omega)|Q-Y_\ell(Q)|^{2p-1-k}|1-Y'_\ell(Q)|^k\Phi^k \right]\\
    &\lesssim\bE\left[\bm1(\cG\in \Omega)(|Q-Y_\ell(Q)|+\Phi)^{2p-2}\Phi\frac{|1-Y'_\ell(Q)|\Phi}{N\eta} \right]=\OO(\bE[\Psi_p]).
\end{split}\end{align}
The two estimates \eqref{e:sumbound1} and \eqref{e:sumbound3} together lead to \eqref{e:refined_bound}  under assumption \Cref{i2:small}.
\end{proof}

\section{Proof of the Iteration Scheme}\label{s:proof_main}
In this section we prove \Cref{p:add_indicator_function} and \Cref{p:iteration}.
\subsection{Proof of \Cref{p:add_indicator_function}}
The following proposition is crucial for proving \Cref{p:add_indicator_function}. It states that if an index $\al\in \qq{\mu}$ appears exactly once in the expression, then its expectation over the randomness of the choice of the resampling data $\bfS$ is small. 
\begin{proposition}
\label{p:upbb}
Fix $z\in {\mathbf D}$, a $d$-regular graph $\cG\in \Omega(z)$, and an edge $(i,o)\in \cG$. 
We denote $T_{\bf S}$ the local resampling  with resampling data ${\bf S}=\{(l_\al, a_\al), (b_\al, c_\al)\}_{\al\in\qq{\mu}}$ around  $o$, and $\widetilde G=T_{\bf S}G$.
 \begin{align}\label{e:single_index_sum}
        \bE_\bfS[G_{c_\al c_\al}^{(b_\al)}-Q], 
        \quad 
        \bE_\bfS[G_{b_\al s}]
        \quad 
        \bE_\bfS[G_{c_\al s}]=\OO\left(\frac{1}{N^{1-\fc}}\right),
    \end{align}
    and 
    \begin{align}\begin{split}\label{e:Gccbb}
    &\bE_\bfS[G_{c_\al c}^{(b_\al b)}]=\frac{d}{\sqrt{d-1}}\bE_\bfS\left[ (G_{c_\al c_\al}^{(b_\al)}-Q)G_{b_\al c}^{(b)}\right]+\OO\left(\frac{1}{N^{1-\fc}}+\bE_\bfS\left[\sum_{v\neq v'\sim b_\al}(|G_{vv'}^{(b_\al)}|^2+|G_{b_\al c}^{(b)}|^2)\right]\right).
    \end{split}\end{align}
    For any vector $w\in \bR^N$, and a term $U^{(u,v)}$, which is a product of $G_{vv}, G_{vv}, 1/G_{uu}$, the following holds 
    \begin{align}\label{e:Wtermexp}
      \quad  \frac{1}{N^2} \sum_{s, u,v} w_s \bE_{\bfS}[A_{uv}G_{\sfL_\al u}G_{sv}U^{(u,v)}], \frac{1}{N^2} \sum_{s, u,v} w_s \bE_{\bfS}[A_{uv}G_{\sfL_\al u}G_{su}U^{(u,v)}] =\OO\left( \frac{\|w\|_2}{N^{5/2-\fc}\Im[z]}\right),
    \end{align}
    for any $\sfL\in \{l,a,b,c\}$.
\end{proposition}

\begin{proof}[Proof of \Cref{p:upbb}]
To prove \eqref{e:single_index_sum} and \eqref{e:Wtermexp}, we first notice that the Green's function $G$ satisfies 
    \begin{align}\label{e:Ghao}
        \sum_x G_{xs}=\sum_s G_{xs}=\frac{1}{d/\sqrt{d-1} -z}\lesssim 1,
    \end{align}
    for $z\in \bf D$ from \eqref{e:D}. We recall from \Cref{s:local_resampling}, the directed edge $(b_\al, c_\al)$ is uniform randomly selected in $\cG^{(\bT)}$, where $\bT$ is the vertex set of $\cB_\ell(o,\cG)$. Thus
    \begin{align*}
        \bE_\bfS[G_{c_\al c_\al}^{(b_\al)}-Q]
        =\bE_\bfS[G_{c_\al c_\al}^{(b_\al)}]-Q
        =\frac{1}{Nd-\OO(|\bT|)}\sum_{u\sim v}G_{vv}^{(u)} -Q+\OO\left(\frac{|\bT|}{Nd}\right)
        =\OO\left(\frac{1}{N^{1-\fc}}\right),
    \end{align*}
    where the sum is over all pairs of adjacent vertices $u\sim v$, and we used the definition \eqref{e:Qsum} of $Q$.
    Moreover, the expectation $\bE_\bfS[\cdot]$ over $b_\al$ or $c_\al$ is uniform over $\qq{N}\setminus \bX$ where $\bX=\bT\cup\{a_\al\}_{\al\in \qq{\mu}}$,
    \begin{align}\label{e:Gbs}
        \bE_\bfS[G_{b_\al s}]
        =\frac{1}{N-\OO(|\bX|)}\sum_v G_{vs}+\OO\left(\frac{|\bX|}{N}\right)
          =\OO\left(\frac{1}{N^{1-\fc}}\right).
    \end{align}
where we used \eqref{e:Ghao}. The same estimate holds for $\bE_\bfS[G_{c_\al s}]$. 

For \eqref{e:Gccbb}, we first notice that by the same argument as for \eqref{e:Gbs} 
\begin{align}\label{e:Gbs2}
    \bE_\bfS[G_{b_\al c}^{(b)}]=\OO\left(\frac{1}{N^{1-\fc}}\right).
\end{align}
Moreover, for any vertex $u\in \qq{N}$, the Schur complement formula gives
\begin{align*}
    -\frac{G_{uu}}{\sqrt{d-1}}\sum_{v:v\sim u}G_{vc}^{(u b)}
    =G_{uc}^{(b)},\quad -\frac{1}{G_{uu}}=z+\frac{1}{d-1}\sum_{v,v'\sim u}G_{vv'}^{(u)}.
\end{align*}
By taking $(u,v)=(b_\al, c_\al)$ in the above expression, and taking expectation, we get
\begin{align*}
  \bE_\bfS[G_{c_\al c}^{(b_\al b)}]
  &=\bE_\bfS\left[\left(z\sqrt{d-1}+\frac{1}{\sqrt{d-1}}\sum_{v,v'\sim b_\al} G_{vv'}^{(b_\al)}\right)G_{b_\al c}^{(b)}\right]+\OO\left(\frac{1}{N^{1-\fc}}\right)\\
  &=\frac{d}{\sqrt{d-1}}\bE_\bfS\left[ (G_{c_\al c_\al}^{(b_\al)}-Q)G_{b_\al c}^{(b)}\right]+\cE^{(b,c)}+\OO\left(\frac{1}{N^{1-\fc}}\right),
\end{align*}
where we used \eqref{e:Gbs2}, and the error term $\cE^{(b,c)}$ is given by
\begin{align*}
    \cE^{(b,c)}
    =\frac{1}{\sqrt{d-1}}\bE_\bfS\left[\sum_{v\neq v'\sim b_\al}G_{vv'}^{(b_\al)} G_{b_\al c}^{(b)}\right].
\end{align*}
This gives \eqref{e:Gccbb}.

    In the following we prove \eqref{e:Wtermexp}. We will only prove the first statement, the second one can be proven in the same way. We have 
    \begin{align*}
       &\phantom{{}={}}\frac{1}{N} \sum_{ u,v} w_s \bE_{\bfS}[A_{uv}G_{\sfL_\al u}G_{sv}U^{(u,v)}]\\
        &=\frac{\OO(1)}{N^2}\sum_{x\in \qq{N}}\sum_{u,v}w_s A_{uv}G_{x u}G_{sv}U^{(u,v)}+\OO\left( \frac{1}{N^2}\sum_{x\in \bX}\sum_{u,v} |w_s|A_{uv}|G_{x u}G_{sv}U^{(u,v)}|\right)\\
        &= \frac{\OO(1)}{N^2}\sum_{u,v}w_s G_{sv} A_{uv}U^{(u,v)} +\OO\left(\frac{|w_s|N^\fc}{N^2\Im[z]}\right),
    \end{align*}
    where in the last line we used \eqref{e:Ghao} and the Ward identity.
If we further average over index $s$, using that $\|G\|_{\spec}\leq 1/\Im[z]$, we conclude that
\begin{align*}
    &\phantom{{}={}}\frac{\OO(1)}{N^3}\sum_{s,v}w_s G_{sv} \sum_u A_{uv}U^{(u,v)}+\OO\left(\sum_s\frac{|w_s|N^\fc}{N^3\Im[z]}\right)\\
    &\lesssim \frac{1}{N^3\Im[z]}\|w\|_2 \sqrt{\sum_{v}   \left|\sum_u U^{(u,v)}A_{uv}\right|^2}+\frac{N^\fc}{N^3\Im[z]} (\sqrt N\|w\|_2)
    \lesssim \frac{\|w\|_2 }{N^{5/2-\fc}\Im[z]}.
\end{align*}
This gives \eqref{e:Wtermexp}.

\end{proof}

\begin{proof}[Proof of \Cref{p:add_indicator_function}]
Fix an embedding of $\cF$ in to $\cG$, such that $I(\cF, \cG)=1$. We pick a switching edge from last step, and denote it as $(i,o)$. We will do a local resampling around the vertex $o$. We split the lefthand side of \eqref{e:maint} into two terms
\begin{align}\begin{split}\label{e:maint0}
&\phantom{{}={}}\bE\left[I(\cF,\cG){\bm1(\cG\in \Omega)}R_\bfi(\cG)\right] \\ 
&=\bE\left[I(\cF,\cG)I(\cF,\wt\cG){\bm1(\cG\in \Omega)}\bm1(\tcG\in \oOmega)R_\bfi(\cG)\right]  +\bE\left[I(\cF,\cG)\bE_\bfS[1-I(\cF,\wt\cG)\bm1(\tcG\in \oOmega)]{\bm1(\cG\in \Omega)}R_\bfi(\cG)\right] \\
&=\bE\left[I(\cF,\cG)I(\cF,\wt\cG){\bm1(\cG\in \Omega)\bm1(\tcG\in \oOmega)}R_\bfi(\cG)\right]  +\OO\left(N^{-1+2\fc}\bE\left[I(\cF,\cG){\bm1(\cG\in \Omega)}|R_\bfi(\cG)|\right] \right),
\end{split}\end{align}
where in the third line we used \Cref{lem:wellbehavedswitch}, that $\bE_\bfS[1-I(\cF,\tcG)\bm1(\tcG\in \oOmega) ]\leq N^{-1+2\fc}$. After averaging over $\bfi$, thanks to \Cref{p:small_Ri} the second term on the righthand side of \eqref{e:maint0} is bounded as
\begin{align*}
    \frac{1}{Z_\cF}\sum_{\bfi}N^{-1+2\fc}\bE\left[I(\cF,\cG){\bm1(\cG\in \Omega)}|R_\bfi(\cG)|\right]=\OO(\bE[\Psi_p]).
\end{align*}

Next we analyze the first term in the last line of \eqref{e:maint0}.
\begin{align}\begin{split}\label{e:maint1}
&\phantom{{}={}}\bE\left[I(\cF,\cG)I(\cF,\wt\cG)\bm1(\cG\in \Omega) \bm1(\tcG\in \oOmega) R_\bfi(\cG)\right]\\
&=
\bE\left[I(\cF,\cG)I(\cF,\wt\cG){\bm1(\cG,\tcG\in \Omega)}R_\bfi(\cG)\right]  +\bE\left[I(\cF,\cG)I(\cF,\wt\cG){\bm1(\cG\in\Omega)\bm1(\tcG\in \oOmega\setminus \Omega)} R_\bfi(\cG)\right].
\end{split}\end{align}
For $\cG\in \Omega$, we can bound the second term on the righthand side of \eqref{e:maint1} by the same argument as for \eqref{e:rough_bound}, and then using \eqref{e:ffbbcopy}
\begin{align}\begin{split}\label{e:maint2}
&\phantom{{}={}}\frac{1}{Z_\cF}\sum_{\bfi}\bE\left[I(\cF,\cG)I(\cF,\wt\cG){\bm1(\cG\in\Omega)\bm1(\tcG\in \oOmega\setminus \Omega)} |R_\bfi(\cG)|\right]\\
&\lesssim\bE\left[(|Q-Y_\ell(Q)|+\Phi)^{2p-1}{\bm1(\cG\in\Omega)\bm1(\tcG\in \oOmega\setminus \Omega)}\right]\lesssim \bE[\Psi_p]
\end{split}\end{align}

The main term is the first term on the righthand side of \eqref{e:maint1}. Recall that $R_\bfi(\cG)=(G_{oo}^{(i)}-Y_\ell(Q))R_\bfi'(\cG)$, we have
\begin{align}\begin{split}\label{e:towrite}
    &\phantom{{}={}}\frac{1}{Z_\cF}\sum_{\bfi}\bE\left[I(\cF,\cG)I(\cF,\wt\cG)\bm1(\cG,\tcG\in \Omega) R_\bfi(\cG)\right]\\
    &= \frac{1}{Z_\cF}\sum_{\bfi}\bE\left[I(\cF,\cG)I(\cF,\wt\cG)\bm1(\cG,\tcG\in \Omega)G_{oo}^{(i)} R'_\bfi(\cG)\right]\\
    &-\frac{1}{Z_\cF}\sum_{\bfi}\bE\left[I(\cF,\cG)I(\cF,\wt\cG)\bm1(\cG,\tcG\in \Omega)Y_\ell(Q) R'_\bfi(\cG)\right].
\end{split}\end{align}
For the first term on the righthand side of \eqref{e:towrite}, thanks to \Cref{lem:exchangeablepair}, $\cG$ and $\tcG=T_{\bf S}(\cG)$ forms an exchangeable pair, we have
\begin{align}\begin{split}\label{e:towrite2}
    &\phantom{{}={}}\frac{1}{Z_\cF}\sum_{\bfi}\bE\left[I(\cF,\cG)I(\cF,\wt\cG)\bm1(\cG,\tcG\in \Omega)G_{oo}^{(i)} R'_\bfi(\cG)\right]\\
    &=\frac{1}{Z_\cF}\sum_{\bfi}\bE\left[I(\cF,\cG)I(\cF,\wt\cG)\bm1(\cG,\tcG\in \Omega)\wt G_{oo}^{(i)} R'_\bfi(\wt\cG)\right],
\end{split}\end{align}
and 
\begin{align}\begin{split}\label{e:towrite3}
    &\phantom{{}={}}\frac{1}{Z_\cF}\sum_{\bfi}\bE\left[I(\cF,\cG)I(\cF,\wt\cG)\bm1(\cG,\tcG\in \Omega)Y_\ell(Q) R'_\bfi(\cG)\right]\\
    &=\frac{1}{Z_\cF}\sum_{\bfi}\bE\left[I(\cF,\cG)I(\cF,\wt\cG)\bm1(\cG,\tcG\in \Omega)Y_\ell(\wt Q) R'_\bfi(\wt \cG)\right]\\
    &=\frac{1}{Z_\cF}\sum_{\bfi}\bE\left[I(\cF,\cG)I(\cF,\wt\cG)\bm1(\cG,\tcG\in \Omega)Y_\ell( Q) R'_\bfi(\wt \cG)\right]\\
    &+\frac{1}{Z_\cF}\sum_{\bfi}\bE\left[I(\cF,\cG)I(\cF,\wt\cG)\bm1(\cG,\tcG\in \Omega)(Y_\ell(Q)-Y_\ell(\widetilde Q)) R'_\bfi(\cG)\right].
\end{split}\end{align}
For the last term on the righthand side of \eqref{e:towrite3}, we can bound it as
\begin{align}\begin{split}\label{e:towrite4}
    &\phantom{{}={}}\left|\frac{1}{Z_\cF}\sum_{\bfi}\bE\left[I(\cF,\cG)I(\cF,\wt\cG)\bm1(\cG,\tcG\in \Omega)(Y_\ell(Q)-Y_\ell(\widetilde Q)) R'_\bfi(\cG)\right]\right|\\
    &\lesssim
    \frac{1}{Z_\cF}\sum_{\bfi}\bE\left[\bm1(\cG,\tcG\in \Omega)|\widetilde Q-Q| |R'_\bfi(\cG)|\right]\lesssim \bE[\Psi_p],
\end{split}\end{align}
where in the second line we used that $|Y'|\lesssim 1$ from \eqref{eq:Yprime}; in the last line we used that $|\widetilde Q-Q| \lesssim \Phi$ from \Cref{lem:Qlemma} and \eqref{e:rough_bound}.
The estimates \eqref{e:towrite}, \eqref{e:towrite2} and \eqref{e:towrite3} and \eqref{e:towrite4} together finishes the proof of \Cref{p:add_indicator_function}.
\end{proof}

\subsection{Proof of \Cref{p:iteration}}
 We denote $T_{\bf S}$ the local resampling  with resampling data ${\bf S}=\{(l_\al, a_\al), (b_\al, c_\al)\}_{\al\in\qq{\mu}}$ around  $o$, and $\widetilde G=T_{\bf S}G$.
We expand the forest $\cF$ by including the resampling data:
\begin{align*}
    &\widehat \cF=(\widehat \bfi, \widehat E)=\cF\bigcup \cT_\ell(o)\bigcup_{\al\in\qq{\mu}} \{(l_\al, a_\al), (b_\al, c_\al)\},\quad \cK^+=\widehat \cK=\cK\bigcup_{\al\in\qq{\mu}} \{(l_\al, a_\al), (b_\al, c_\al)\}\\
    &
      \cC^+=\widehat \cC=\cC\bigcup_{\al\in\qq{\mu}} \{(b_\al, c_\al)\},\quad
      (\cC^+)^\circ=\widehat \cC^\circ=\cC^+\setminus\{(i,o)\}, 
      \quad 
      P=P(\widehat \cF, z,\msc).
\end{align*}
By the same argument as in \eqref{e:maint0}, we can remove the indicator function $I(\cF, \wt \cG)$; and thanks to \Cref{lem:configuration}, we can further replace $\cF$ by $\widehat \cF$ as 
    \begin{align}\begin{split}\label{e:hatF}
    &\phantom{{}={}}\frac{1}{Z_\cF}\sum_{ \bfi}\bE\left[I(\cF,\cG)I(\cF,\wt\cG)\bm1(\cG,\tcG\in \Omega) (\wt G_{oo}^{(i)}-Y_\ell(Q)) R'_\bfi(\wt\cG)\right]\\
    &= \frac{1}{Z_{\widehat\cF}}\sum_{\widehat \bfi}\bE\left[I(\widehat\cF,\cG)\bm1(\cG,\tcG\in \Omega) (\wt G_{oo}^{(i)}-Y_\ell(Q)) R'_\bfi(\wt\cG)\right]+\OO(\bE[\Psi_p]).
 \end{split}\end{align}

In order to use our various lemmas from Section \ref{sec:expansions} that allow us to reduce $R_\bfi(\wt \cG)$ to terms of the unswitched graph $\cG$, we need to classify the factors of $R'_\bfi$ based on their dependence on $o,i$, where $(i,o)\in \cC^\circ$. We write
\begin{align}
    R'_\bfi=\prod_{j\in J}W_j,\quad W_0:=(d-1)^{2\ell}(\wt G_{oo}^{(i)}-Y_\ell(Q)) 
\end{align}
where $W_j$ is a factor of a generic function, as defined in Definition \ref{def:pgen}. Examining the definition, we see that factors are only of the form 
\begin{align*}
&\{(d-1)^{2\ell}(G_{cc}^{(b)}-Q)\}_{(b,c)\in \cC^\circ},\quad \{(d-1)^{2\ell}(G_{ss'}-P_{ss'})\}_{s,s'\in \{o,i\}},\quad 
\{(d-1)^{2\ell}G_{oc}^{(ib)}\}_{(i,o)\neq (b,c)\in \cC^\circ},\\
&\{(d-1)^{2\ell}(G_{ss'}-P_{ss'}) \}_{s\in \{o,i\}, s'\in \cK\setminus\{o,i\}},\quad
\{(d-1)^{2\ell} G_{cc'}^{(bb')}\}_{(i,o)\neq (b,c), (b',c')\in \cC^\circ},\\
&\{(d-1)^{2\ell}(G_{ss'}-P_{ss'})\}_{ s,s'\in \cK\setminus\{o,i\}},
 \quad
\{(d-1)^{2\ell}G_{ss'}\}_{ s\in \cK, s'\in \cV}, \quad
\{G_{uv}, G_{vv}, 1/G_{uu}\}_{(u,v)\in \cV},
\\
& (d-1)^{2\ell}(Q-\msc),\quad  Q-Y_\ell(Q),\quad  1-Y'_\ell(Q).
\end{align*}

We will next demonstrate the effect of the switch on each of the factors mentioned above, using the results from \Cref{sec:expansions}. Before that, we introduce some notations for the lower-order terms created when performing a switch. These terms are of two types:  $\cD_{r}^{S}$ for $r\geq 0$, which is a weighted sum of products involving at least $r$ factors of the form 
\begin{align}\label{def:DS}
   (d-1)^{2\ell}(G_{c_\al c_\al}^{(b_\al)}-Q), \quad (d-1)^{2\ell} G_{c_\al c_\beta}^{(b_\al b_\beta)},\quad  (d-1)^{2\ell}(Q-\msc). 
\end{align}
The total weight is $\OO(\poly(\ell))$. As established in \Cref{lem:diaglem} and \Cref{lem:offdiagswitch}, these factors are the additional terms generated by lower-order analysis when utilizing the Schur complement formulas.

The other type is $\cD_{r}^{W}$, which is a weighted sum of products involving at least $r$ factors of the form 
\begin{align}
    \label{def:DW}
   \{(d-1)^{2\ell}(G_{ss'}-P_{ss'})\}_{s,s'\in \{l_\alpha,a_\alpha,b_\alpha,c_\alpha\}_{\alpha\in \qq{\mu}}}.
\end{align}
The total weight is $\OO(\poly(\ell))$.
As established in \Cref{lem:generalQlemma} and \Cref{lem:Qlemma}, these factors are the additional terms generated by lower-order analysis when utilizing the Woodbury formula. The exact functions can change line to line.

The advantage of introducing these error terms is that, according to \Cref{thm:prevthm}, each of these factors is bounded by 
 $\varepsilon$, thus an increase of $r$ gives us a demonstrably smaller term. For example $(d-1)^{2\ell}(G_{c_\alpha c_\alpha}^{(b_\alpha)}-Q)(d-1)^{2\ell}G_{c_\alpha c_\beta}^{(b_\alpha b_\beta)}\in \cD^S_{2}$, and $|(d-1)^{2\ell}(G_{c_\alpha c_\alpha}^{(b_\alpha)}-Q)(d-1)^{2\ell}G_{c_\alpha c_\beta}^{(b_\alpha b_\beta)}|\leq ((d-1)^{2\ell}\varepsilon)^2$ for $\GG\in \Omega$.

For any $(b,c)\neq (b',c')\in \cC^\circ$, we introduce the following error functions (they are error terms in \Cref{sec:expansions}) that satisfy
\begin{align}\begin{split}\label{eq:cEerror1} 
   &|\cE_{1}|= \sum_{\al} |\wt G^{(\bT)}_{c_\al c_\al}-G^{(b_\al)}_{c_\al c_\al}|+\sum_{\al\neq \beta} |\wt G^{(\bT)}_{c_\al c_\beta}-G^{(b_\al b_\beta)}_{c_\al c_\beta}|\\
   & |\cE_{1}^{(b,c)} |= \sum_\al |\wt G_{c_\al c_\al}^{(\bT b)}-G^{(b_\al )}_{c_\al c_\al}|) +\sum_{\al \neq \beta}|\wt G_{c_\al c_\beta}^{(\bT b)}-G^{(b_\al b_\beta)}_{c_\al c_\beta}|+N^{-2}\\
&|\cE_{1}^{(b,c),(b',c')} |= \sum_\al (|\wt G_{c_\al c}^{(\bT b)}-G_{c_\al c}^{(b_\al b)}|+|\wt G^{(\bT bb')}_{cc'}-G^{( bb')}_{cc'}|)
    +(d-1)^\ell\sum_\al (|\wt G_{cc_\al}^{(\bT bb')}|^2+|\wt G_{c'c_\al}^{(\bT bb')}|^2) +N^{-2},\\
&|\cE_{2}|= N^{-2}, \quad |\cE_{3}|= \Phi.
\end{split}\end{align}
Thanks to \Cref{lem:task2}, the important property of the first three errors is that 
\begin{align*}
\bE[\bE_\bfS[\cE_{1}\Pi_p]],\bE[\bE_\bfS[\cE_{1}^{(b,c)}\Pi_p]],\bE[\bE_\bfS[\cE_{1}^{(b,c),(b',c')} \Pi_p]]=\OO(\bE[\Psi_p]).
\end{align*}
Consequently, these error terms are small enough that any term that incorporates these error terms as a factor is negligible. 

Through the lemmas in Section \ref{sec:expansions}, we see that the leading order term of a generic function on a switched graph replaces the indices of the core edge $(i,o)$ with an average over the edges $\{(b_\alpha,c_\alpha)\}_{\al \in \qq{\mu}}$ from the resampling data $\bfS$. The other, lower-order terms are demonstrably smaller. We must confirm this for each of the factors listed above.  
In the following, $\fc_j$ for $j\geq 1$ will represent constants bounded by $\OO(\poly(\ell))$ that can change line to line. Given the factor $W_j$, we will consider the corresponding term after the local resampling, denoted as $\wt W_j$. Then we approximate $\wt W_j=\wh W_j+\cE$, where $\wh W_j$ is a new function of the original Green's function and $\cE$ is an error term listed above in \eqref{eq:cEerror1}. Finally we write $\wh W_j$ as a leading order term, plus a higher order term of the form containing factors $\cD_{r}^S$ or $\cD_{r}^W$, depending on the expansion used. 
\begin{enumerate}

\item \label{i1}
By Lemma \ref{lem:diaglem}, $\wt W_0:=W_0=(d-1)^{2\ell}(\wt G_{oo}^{(i)}-Y_\ell(Q)$, can be written as 
\begin{align}\label{e:case1W0}
    &\wt W_{0}=\widehat W_0+\cE_1,\quad \widehat W_0:=\fc_1(d-1)^\ell\sum_{\al\in\sfA_i}(G^{(b_\alpha)}_{c_\alpha c_\alpha}-Q)+\fc_2(d-1)^\ell\sum_{\al\neq \beta\in\sfA_i} G^{(b_\alpha b_\beta)}_{c_\alpha c_\beta}
   +\cD_{2}^S.
\end{align}
where $\cD_2^S, \cE_1$ are as in \eqref{def:DS} and \eqref{eq:cEerror1}.

    \item \label{i2}
If $W_j=(d-1)^{2\ell}(G_{oo}^{(i)}-Q)$  we can write 
\begin{align*}
    (d-1)^{2\ell}(\wt G_{oo}^{(i)}-\wt Q)=(d-1)^{2\ell}(\wt G_{oo}^{(i)}-Y_\ell( Q))+(d-1)^{2\ell}(Y_\ell( Q)-\wt Q),
\end{align*}
then expand according to Lemma \ref{lem:diaglem}. This gives that $\wt W_j=\wh W_j+\cE_1$, where
\begin{align*}
\wh W_j&=\fc_1(d-1)^{\ell}\sum_{\al\in \sfA_i} ( G_{c_\al c_\al}^{(b_\al)}-Q) + \fc_2 (d-1)^{\ell}\sum_{\al\neq \beta\in \sfA_i}  G_{c_\al c_\beta}^{(b_\al b_\beta)}\\ 
&+\fc_3 (d-1)^{2\ell}(Y_\ell(Q)-Q) +\fc_4 
    (d-1)^{2\ell}(Q-\wt Q)+\cD_{2}^S.
\end{align*}

A similar expansion is possible for $W_j\in \{(d-1)^{2\ell}(G_{ss'}-P_{ss'})\}_{s,s'\in \{o,i\}}$,
\begin{align*}
   \wh W_j
    &={\fc_1} {(d-1)^\ell}\sum_{\al\in\sfA_i}(G^{(b_\alpha)}_{c_\alpha c_\alpha}-Q)
    +{\fc_2} {(d-1)^\ell}\sum_{\al\notin\sfA_i}(G^{(b_\alpha)}_{c_\alpha c_\alpha}-Q)
    \\
   &+{\fc_3} {(d-1)^\ell}\sum_{\al\neq \beta\in \sfA_i} G^{(b_\alpha b_\beta)}_{c_\alpha c_\beta}
   +{\fc_4} {(d-1)^\ell}\sum_{\al\neq \beta\in\sfA^\complement_i} G^{(b_\alpha b_\beta)}_{c_\alpha c_\beta}
   \\
   &+{\fc_5} {(d-1)^\ell}\sum_{\al\in\sfA_i\atop \beta\in \sfA^\complement_i} G^{(b_\alpha b_\beta)}_{c_\alpha c_\beta}+\fc_6 (d-1)^{2\ell}(Q-\msc)+\cD_2^S.
\end{align*}

\item \label{i3}
For $W_j\in\{(d-1)^{2\ell}G_{oc}^{(ib)}\}_{(i,o)\neq (b,c)\in \cC^\circ}$, by  \Cref{lem:offdiagswitch}, we have $\wt W_j=\wh W_j+\cE_{1}^{(b,c)}$, where
\begin{align}
  \wh W_j=\fc_1(d-1)^{3\ell/2}\sum_{\al\in \sfA_i} G_{c_\al c}^{(b_\al b)}+\frac{1}{(d-1)^{\ell}}\sum_{\al\in\qq{\mu}}\cD_{1}^S\times  (d-1)^{2\ell}G_{c_\al c}^{(b_\al b)},
\end{align}
where $\cD_1^S, \cE_1^{(b,c)}$ are as in \eqref{def:DS} and \eqref{eq:cEerror1}.
\item \label{i4}
For $W_j\in \{(d-1)^{2\ell} G_{cc}^{(b)}\}_{(b,c)\in \cC^\circ}$ or $\{(d-1)^{2\ell} G_{cc'}^{(bb')}\}_{(b,c)\neq  (b',c')\in \cC^\circ}$, by \Cref{lem:offdiagswitch}, we can replace it as $\widetilde W_j=\widehat W_j+\cE_{1}^{(b,c)}$ or $\widetilde W_j=\widehat W_j+\cE_{1}^{(b,c),(b',c')}$, where $\widehat W_j=W_j$, and $\cE_{1}^{(b,c)}, \cE_{1}^{(b,c),(b',c')}$ are as in \eqref{eq:cEerror1}.

\item \label{i5}
For $W_j\in\{(d-1)^{2\ell} (G_{ss'}-P_{ss'})\}_{s\in \{o,i\}, s'\in \cK\setminus\{o,i\}}$ or $\{(d-1)^{2\ell} G_{ss'}\}_{s\in \{o,i\}, s'\in \cV}$, by \Cref{lem:generalQlemma}, we have $\wt W_j=\widehat W_j+\cE_{2} $, where 
\begin{align*}
    \wh W_j&=\fc_1 (d-1)^{2\ell} G_{os'} +\fc_2(d-1)^{2\ell} G_{is'}+\\
    &+(d-1)^{3\ell/2}\left(\fc_3 \sum_{\al\in \sfA_i} G_{b_\al s'}+\fc_4 \sum_{\al\in \sfA_i} G_{c_\al s'}
    +\fc_5\sum_{\al\in\sfA^\complement_i} G_{b_\al s'}+ \fc_6\sum_{\al\in\sfA^\complement_i} G_{c_\al s'}\right)\\
    &+\frac{1}{(d-1)^{\ell}}\sum_{\al\in\qq{\mu}\atop \sfL\in\{l,a,b,c\}} \cD_{1}^{W}\times (d-1)^{2\ell} G_{\sfL_\al s'}\\
    &+\frac{1}{(d-1)^{2\ell}}\sum_{\al,\beta\in\qq{\mu}\atop \sfL, \sfL'\in\{l,a,b,c\}} (d-1)^{2\ell} G_{s\sfL_\al }\times \cD_{0}^{W}\times (d-1)^{2\ell} G_{\sfL_\beta s'},
\end{align*}
 where $\cD_0^W, \cD_1^W, \cE_2$ are as in \eqref{def:DW} and \eqref{eq:cEerror1}.

\item \label{i6}
For $W_j\in\{(d-1)^{2\ell}(G_{ss'}-P_{ss'})\}_{s,s'\in \cK\setminus \{o,i\}}$ or $\{(d-1)^{2\ell} G_{ss'}\}_{s\in \cK\setminus \{o,i\}, s'\in \cV}$, by \Cref{lem:generalQlemma}, we have $\wt W_j=\wh W_j+\cE_{2} $, where
\begin{align*}
    \wh W_j &= W_j+\frac{1}{(d-1)^{\ell}}\sum_{\al\in\qq{\mu}\atop \sfL\in\{l,a,b,c\}} \cD_{1}^{W}\times (d-1)^{2\ell} G_{\sfL_\al s'}\\
    &+\frac{1}{(d-1)^{2\ell}}\sum_{\al,\beta\in\qq{\mu}\atop \sfL, \sfL'\in\{l,a,b,c\}} (d-1)^{2\ell} G_{s\sfL_\al }\times \cD_{0}^{W}\times (d-1)^{2\ell} G_{\sfL_\beta s'}.
\end{align*}
\item \label{i7}
For $W_j\in\{G_{uv}, G_{vv}, 1/G_{uu}\}_{(u,v)\in \cV}$, by \Cref{lem:generalQlemma}, we have, for $w\in\{u,v\}$, 
\begin{align*}
     \wt W_j=\widehat W_j +\cD_{0}^{W}\times U^{(u,v)}+\cE_{2},\quad \widehat W_j:=W_j
\end{align*}
where $U^{(u,v)}$ consists of at least two Green's function factors in the form $(d-1)^{2\ell}G_{\sfL_\al u}, (d-1)^{2\ell}G_{\sfL_\al v}$, and some factors in the form $G_{vv}, G_{uv}, 1/G_{uu}$. 
Here $\cD_0^W, \cE_2$ are as in \eqref{def:DW} and \eqref{eq:cEerror1}.
\item \label{i8}
For $W_j=Q-Y_\ell(Q)$, by \Cref{lem:Qlemma}, then $\wt W_j=\wh W_j+\cE_{2}$ if $\wh W_j$ is a linear combination of terms in the form
\begin{align*}
    Q-Y_\ell(Q),\quad \frac{1}{N}\sum_{u,v}A_{uv}(1-Y'_\ell(Q)) \cD_{0}^{W} \times U^{(u,v)},
\end{align*}
where $U^{(u,v)}$ consists of at least two Green's function factors in the form $G_{\sfL_\al u}, G_{\sfL_\al v}$, and some factors in the form $G_{vv}, G_{uv}, 1/G_{uu}$; Here $\cD_0^W, \cE_2$ are as in \eqref{def:DW} and \eqref{eq:cEerror1}.

\item \label{i9}
For $W_j\in \{Q-\msc,1-Y'_\ell(Q)\}$, we have by Lemma \ref{lem:generalQlemma}
\[
\wt W_j=\widehat W_j+\cE_{3},\quad \widehat W_j:=W_j,
\]
where $\cE_3$ is as defined in \eqref{eq:cEerror1}.
\end{enumerate}
In all the above cases, we have that $|\widetilde W_j|, |\widehat W_j|\lesssim 1$ provided that $\cG, \widetilde \cG\in \Omega(z)$. In \Cref{i5} and \Cref{i6} with $s'\in \cV$, each term of $\widehat W_j$ contains a term $(d-1)^{2\ell}G_{\sfL_\al s'}$. In particular the number of terms in the form $(d-1)^{2\ell}G_{ss'}$ with $s\in \cK^+, s'\in \cV$ does not decrease. Thus the total number of Green's function factors associated with each special edge does not decrease. 
 In \Cref{i9}, each term $U^{(u,v)}$ contains at least two terms in the form $(d-1)^{2\ell}G_{s s'}$ for $s\in \cK^+, s'\in \{u,v\}$. 

Next we show that we can replace $\widetilde W_j$ by $\widehat W_j$, and the errors are negligible. Starting from \eqref{e:hatF}, we have
\begin{align}\begin{split}\label{e:replace}
    &\phantom{{}={}}\frac{1}{Z_{\widehat\cF}}\sum_{\widehat \bfi}\bE\left[I(\widehat\cF,\cG)\bm1(\cG,\tcG\in \Omega)(\wt G_{oo}^{(i)}-Y_\ell(Q)) R'_\bfi(\wt\cG)\right]\\
    &=\frac{1}{(d-1)^{2\ell}Z_{\widehat\cF}}\sum_{\widehat \bfi}\bE\left[I(\widehat\cF,\cG)\bm1(\cG,\tcG\in \Omega)\widehat W_0 \prod_j \widehat W_j\right]\\
    &+\frac{1}{(d-1)^{2\ell}Z_{\widehat\cF}}\sum_{\widehat \bfi}\bE\left[I(\widehat\cF,\cG)\bm1(\cG,\tcG\in \Omega) \sum_k\prod_{j<k} \widetilde W_j(\widetilde W_k-\widehat W_k)\prod_{j>k} \widehat W_j\right].
    \end{split}
\end{align}
To show that \eqref{e:replace} is negligible, we prove that  for each $k$,
\be\label{e:error_bound}
\frac{1}{(d-1)^{2\ell}Z_{\widehat\cF}}\sum_{\widehat \bfi}\bE\left[I(\widehat\cF,\cG)\bm1(\cG,\tcG\in \Omega) \left|\prod_{j<k} \widetilde W_j(\widetilde W_k-\widehat W_k)\prod_{j>k} \widehat W_j\right|\right]=\bE[\Psi_p].
\ee

If $W_k$ is as in \Cref{i1} or \Cref{i2}, we can first sum over the indices in $\widehat\bfi$ corresponding to the special edges. For any special edge $(u,v)\in \cV$, there are at least two factors of the form $(d-1)^{2\ell}\widetilde G_{sw}$ or $(d-1)^{2\ell} G_{sw}$  where $s\in \cK, w\in \{u,v\}$. The summation over $u,v$ in \eqref{e:error_bound}, by the Ward identity, we have
\begin{align*}
    \frac{(d-1)^{2\ell}}{N}\sum_{w\in\qq{N}}|G_{sw}|^2, \frac{(d-1)^{2\ell}}{N}\sum_{w\in\qq{N}}|\widetilde G_{sw}|^2\lesssim \Phi.
\end{align*}

After summing over the special edges, we then sum over the indices $\widehat \bfi\setminus\{o,i\}$. This leads to the overall bound of
\begin{align}\label{e:error_bound1}
    \frac{1}{N}\sum_{o,i}\bE\left(\bE_\bfS\left[I(\{o,i\},\cG)\bm1(\cG,\tcG\in \Omega)|\cE_1| \left(|Q-Y_\ell(Q)|+\Phi\right)^{2p-1}\right]\right), 
\end{align}
where $\cE_1$ is as defined in \eqref{eq:cEerror1}.
By \Cref{lem:task2}, when taking the expectation, we are left with $\OO(\bE[\Psi_p])$. If $W_k$ is as in \Cref{i3} or \Cref{i4}, a similar argument implies that the total error is negligible.

If $W_k$ is as in \Cref{i5}, \Cref{i6},  or \Cref{i7}, similarly we can first sum over the special edges, then the remaining indices in $\widehat \bfi$. We notice that if the term $W_k$ is the Green's function involving the special edges, then the summation has one less factor of $\Phi$. But in these cases we have a better bound $|\cE_2|= N^{-2}$ (recall from \eqref{eq:cEerror1}), and the final bound is given as 
\begin{align}\label{e:error_bound2}
    \frac{1}{N^2}\left[\bm1(\cG,\tcG\in \Omega) \left(|Q-Y_\ell(Q)|+\Phi\right)^{2p-2}\right]=\OO(\bE[\Psi_p]). 
\end{align}

If $W_k$ is as in \Cref{i8}, similarly we can first sum over the special edges, and then sum over the remaining indices in $\widehat \bfi$. In this case, we have one less copy of $Q-Y_\ell(Q)$ factor, but $|\cE_2|= N^{-2}$ (recall from \eqref{eq:cEerror1})
\begin{align}\label{e:error_bound4}
        \frac{1}{N^2}\left[\bm1(\cG,\tcG\in \Omega) \left(|Q-Y_\ell(Q)|+\Phi\right)^{2p-2}\right]=\OO(\bE[\Psi_p]).  
\end{align}

If $W_k$ is as in \Cref{i9}, similarly we can first sum over the special edges and then over the remaining indices in $\widehat \bfi$, together with $|\cE_3|=\Phi$ from \eqref{eq:cEerror1}
\begin{align}\label{e:error_bound3}
    \bE\left[\bm1(\cG,\tcG\in \Omega)\Phi \left(|Q-Y_\ell(Q)|+\Phi\right)^{2p-1}\right]=\OO(\bE[\Psi_p]). 
\end{align}

The claim \eqref{e:error_bound} follows from combining \eqref{e:error_bound1}, \eqref{e:error_bound2}, \eqref{e:error_bound3} and \eqref{e:error_bound4}.

We remark that in the above procedure, each time if we replace some $W_j=Q-Y_\ell(Q)$ by 
\begin{align}\label{e:special_term}
    \frac{1}{N}\sum_{u,v}A_{uv}(1-Y'_\ell(Q))\cD_{0}^{W}\times  U^{(u,v)}
\end{align}
by using \Cref{i8}, this gives a new special edge $(u,v)$. We denote the set of these new special edges $\{\widehat e_\al\}_{\al\in\qq{ \widehat \mu}}$, and denote
\begin{align*}
    & \cF^+=(\bfi^+, E^+)=\widehat \cF \bigcup_{\al\in\qq{ \widehat \mu}}\{\widehat e_\al\},
    \quad 
    \cV^+=\cV\cup \{\widehat e_\al\}_{\al\in\qq{ \widehat \mu}}.  
\end{align*}
Each special edge $(u,v)$ is associated with at least two factors in the form $(d-1)^{2\ell}G_{ss'}$ with $s\in \cK^+$ and $s'\in \{u,v\}$, and a factor $1-Y'_\ell(Q)$. Moreover the total number of special edges, plus the factors of the form $(Q-Y_\ell(Q))$ or $\overline{(Q-Y_\ell(Q))}$ is $2p-1$ at every step.

From the above discussion, we can rewrite the second line in \eqref{e:replace} as,
\begin{align}\label{e:hatRi}
    \frac{1}{(d-1)^{2\ell}Z_{\widehat\cF}}\sum_{\widehat \bfi}\bE\left[I(\widehat\cF,\cG)\bm1(\cG,\tcG\in \Omega) \prod_j \widehat W_j\right]=\frac{1}{(d-1)^{2\ell}Z_{\cF^+}}\sum_{ \bfi^+}\bE\left[I(\cF^+,\cG)\bm1(\cG,\tcG\in \Omega) \prod_j \widehat W_j\right],
\end{align}
where we replaced the indicator function $I(\widehat\cF,\cG)$ by $I(\cF^+,\cG)$, which include the special edges, and $Z_{\cF^+}=Z_{\widehat \cF}(Nd)^{\widehat \mu}$. Then we can further rewrite \eqref{e:hatRi} 
as  a weighted sum of terms with the total weight bounded by $\OO(\poly(\ell))$, in the form 
\begin{align}\begin{split}\label{e:oneterm}
&\frac{1}{(d-1)^{2\ell}Z_{\cF^+}}\sum_{\bfi^+}\bE\left[I(\cF^+,\cG)\bm1(\cG,\tcG\in \Omega) \widehat R_{\bfi^+}\right],\\
&\widehat R_{\bfi^+}=\frac{1}{(d-1)^{(k_1 +k_2/2+k_3/2)\ell}} \sum_{\bm\al, \bm\beta, \bm\gamma}   \prod_{j=1}^{k_1}  A_{\al_  j}  \prod_{  j=1}^{k_2}  B_{\beta_j} 
\prod_{j=1}^{k_3/2}  C_{\gamma_{2j-1} \gamma_{2j}}   D(\bm\theta),
\end{split}\end{align}
 Here the summation for $\bm\al, \bm\beta, \bm\gamma$ is over each $\al_j, \beta_j, \gamma_j$ in one of the sets $\qq{\mu}$, $\sfA_i=\{\al\in \qq{\mu}: \dist_\cT(i,o)=\ell+1\}$ and $\qq{\mu}\setminus \sfA_i$;
 and the factors in \eqref{e:oneterm} are given by
\begin{align*}
  &A_\al=  (d-1)^{2\ell} ( G_{c_\al c_\al}^{(b_\al)}-Q), \\
  &B_\beta\in \left\{(d-1)^{2\ell}  G_{c_\beta c}^{(b_\beta b)}, (d-1)^{\ell} G_{b_\beta s'},(d-1)^{\ell}G_{c_\beta s'}\right\},\quad (b,c)\in \cC^\circ,\quad s'\in \cK\cup\cV\setminus\{o,i\}\\
  &  C_{\gamma \gamma'}= (d-1)^{2\ell}  G_{c_\gamma c_{\gamma'}}^{(b_\gamma b_{\gamma'})}  \\
  &  D(\bm\theta) \text{ is a product of other terms},
\end{align*}
where $  D(\bm\theta)$ indicates that it  depends on $\{b_\theta, c_\theta\}_{\theta \in \bm\theta}$ for some ${\bm\theta}\subset\qq{\mu}$.

Next we show that if any index in $\bm\al, \bm\beta$ appears only once among $\bm\al, \bm\beta, \bm \gamma, \bm\theta$, then 
\begin{align}\label{e:single_index_bound}
    \left|\frac{1}{(d-1)^{2\ell}Z_{\cF^+}}\sum_{\bfi^+}\bE\left[I(\cF^+,\cG)\bm1(\cG,\tcG\in \Omega) \widehat R_{\bfi^+}\right]\right|\lesssim \bE[\Psi_p].
\end{align}
There are two cases
\begin{enumerate}
    \item Assume $\widehat R_{\bfi^+}$ contains one of the terms $(d-1)^{2\ell}( G_{c_\al c_\al}^{(b_\al)}-Q)$, $(d-1)^{2\ell}  G_{c_\al c}^{(b_\al b)}$, or $(d-1)^{2\ell}  G_{\sfL_\al s'}$, where $s'$ does not belong to a special edge, and the index $\al$ only appears once. Then we can first average over the resampling edge $\{(b_\al, c_\al)\}$. Thanks to \Cref{p:upbb}, this gives a factor $\OO(N^{-1+\fc})$. Next we can  sum over the special edges, and then sum over the remaining indices in $\bfi^+$, to conclude that
    \begin{align*}
         &\phantom{{}={}}\left|\frac{1}{(d-1)^{2\ell}Z_{\cF^+}}\sum_{\bfi^+}\bE\left[I(\cF^+,\cG)\bm1(\cG,\tcG\in \Omega) \widehat R_{\bfi^+}\right]\right|\\
         &\lesssim \frac{1}{N^{1-\fc}}\bE\left[\bm1(\cG\in \Omega) \left(|Q-Y_\ell(Q)|+\Phi\right)^{2p-1}\right]\lesssim \bE[\Psi_p].
    \end{align*}
    \item If $\widehat R_{\bfi^+}$ contains a term $(d-1)^{2\ell}  G_{\sfL_\al u}$, where $(u,v)$ is a special edge, and the index $\al$ only appears once. Then $\widehat R_{\bfi^+}$ contains at least one other Green's function term in the form $(d-1)^{2\ell}G_{ss'}$ with $s\in \cK^+$ and $s'\in \{u,v\}$. If there are more than one such terms, the same argument as in the first case leads to the desired bound. 
    Otherwise, we can first sum over special edges $\cV^+\setminus\{u,v\}$, and then sum over the remaining indices in $\bfi^+\setminus \{s,u,v\}$, to conclude that
    \begin{align}\begin{split}\label{l:single_index2}
         &\phantom{{}={}}\frac{1}{(d-1)^{2\ell}Z_{\cF^+}}\sum_{\bfi^+}\bE\left[I(\cF^+,\cG)\bm1(\cG,\tcG\in \Omega) \widehat R_{\bfi^+}\right]\\
         &= \frac{1}{N^2}\sum_{s, u,v}\bE\left[\bm1(\cG\in \Omega) I(s, \cG)A_{uv} (1-Y'_\ell(Q))(d-1)^{2\ell}   \bE_\bfS[G_{\sfL_\al u}](d-1)^{2\ell}G_{ss'}w_s U^{(u,v)}\right],
    \end{split}\end{align}
    where for $\cG\in \Omega$, $|U^{(u,v)}|\lesssim 1$ (which is a product of $G_{vv}, G_{uv}, 1/G_{uu}$), and the vector $w\in \bR^N$ satisfies
    \begin{align*}
        |w_s|\lesssim \left(|Q-Y_\ell(Q)|+\Phi\right)^{2p-2},\quad 1\leq s\leq N.
    \end{align*}
    Thanks to \Cref{p:upbb}, we can further upper bound \eqref{l:single_index2} as 
    \begin{align*}
|\eqref{l:single_index2}|\lesssim\bE\left[ \bm1(\cG\in \Omega)\frac{|1-Y'_\ell(Q)|}{N\eta}\Phi\left(|Q-Y_\ell(Q)|+\Phi\right)^{2p-2}\right]\lesssim \bE[\Phi_p],
    \end{align*}
    as desired.
\end{enumerate}

From the discussion above, we can focus on the case that each index in $\bm\al, \bm\beta$ appears at least twice among $\bm\al, \bm\beta, \bm \gamma, \bm\theta$. For the indices $\bm\gamma$, there are three cases:
\begin{enumerate}    
\item If $k_3/2\geq 2$, then for fixed $\bm\gamma, \bm\theta$, the total number of new indices for the summation of $\bm\al, \bm\beta$ is at most $(k_1+k_2)/2$.    Consequently, \eqref{e:oneterm} is a weighted sum of terms with total weights $\OO(1)$, in the form
\begin{align}\label{e:oneterm3}
\frac{1}{(d-1)^{{(k_1-k_3)}\ell/2}}  \prod_{j=1}^{k_1}  A_{\al_i}  \prod_{j=1}^{k_2}  B_{\beta_j} 
\prod_{j=1}^{k_3/2}  C_{\gamma_{2j-1} \gamma_{2j}}   D(\bm\theta)=: \frac{1}{(d-1)^{(k_1-k_3)\ell/2}} R_{\bfi^+},
\end{align}
Since $k_3/2\geq 2$, \Cref{p:small_Ri} implies that
\begin{align*}
  \frac{1}{(d-1)^{2\ell+(k_1-k_3)\ell/2}Z_{\cF^+}}\sum_{\bfi^+}\bE\left[I(\cF^+,\cG)\bm1(\cG,\tcG\in \Omega)  R_{\bfi^+}\right]\lesssim \bE[\Psi_p].
\end{align*}
    \item If $k_3/2\leq 1$ and each index in $\bm\gamma$ appears at least twice among $\bm\al, \bm\beta, \bm \gamma, \bm\theta$. 
Then for fixed $\bm\theta$, the total number of new indices for the summation of $\bm\al, \bm\beta,\bm\gamma$ is at most $(k_1+k_2+k_3)/2$. Consequently, \eqref{e:oneterm} is a weighted sum of terms with total weights $\OO(1)$, in the form
\begin{align}\label{e:oneterm2}
\frac{1}{(d-1)^{(k_1+k)\ell/2}}  \prod_{j=1}^{k_1}  A_{\al_j}  \prod_{j=1}^{k_2}  B_{\beta_j} 
\prod_{j=1}^{k_3/2}  C_{\gamma_{2j-1} \gamma_{2j}}   D(\bm\theta)=: \frac{1}{(d-1)^{(k_1+k)\ell/2}} R_{\bfi^+},
\end{align}
where the total number of new indices in $\bm\al, \bm\beta, \bm \gamma$ is $(k_1+k_2+k_3-k)/2$ with $k\geq 0$, meaning that \eqref{e:hatRi} is a weighted sum of terms (with total weights $\OO(\poly(\ell))$) in form
\begin{align}\label{e:R+term}
  \frac{1}{(d-1)^{2\ell+(k_1+k)\ell/2}Z_{\cF^+}}\sum_{\bfi^+}\bE\left[I(\cF^+,\cG)\bm1(\cG,\tcG\in \Omega)  R_{\bfi^+}\right].
\end{align}

    \item If $k_3/2\leq 1$ and $\widehat R_{\bfi^+}$ contains a term $(d-1)^{2\ell}G_{c_\nu c_{\nu'}}^{(b_\nu b_{\nu'})}$, where $\nu\neq \nu'\in \bm\gamma\cup \bm\theta$ and the index $\nu$ only appears once. 
    Then we can first average over the resampling edge $\{(b_\nu, c_\nu)\}$, using \eqref{e:Gccbb}. In this way, we can replace $G_{c_\nu c_{\nu'}}^{(b_\nu b_{\nu'})}$ by 
    \begin{align}\label{e:GGG}
    (d-1)^{2\ell}(G_{c_\nu c_\nu}^{(b_\nu)}-Q)G_{b_\nu c_{\nu'}}^{(b_{\nu'})}
    =(d-1)^{2\ell}(G_{c_\nu c_\nu}^{(b_\nu)}-Q)\left(G_{b_\nu c_{\nu'}}-\frac{G_{b_\nu b_{\nu'}}G_{b_{\nu'} c_{\nu'}}}{G_{b_{\nu'} b_{\nu'}}}\right),
    \end{align}
    and the error terms are bounded by $\OO(\bE[\Psi_p])$ thanks to \Cref{lem:deletedalmostrandom} and \Cref{p:small_Ri}. Condition on the event that $I(\cF^+,\cG)=1$, we can further expand the last term in \eqref{e:GGG}
    \begin{align}\label{e:GGG2}
        \frac{G_{b_\nu b_{\nu'}}G_{b_{\nu'} c_{\nu'}}}{G_{b_{\nu'} b_{\nu'}}}
        =\frac{G_{b_\nu b_{\nu'}}}{\md}\sum_{k=1}^{\fb}\frac{(-1)^k(G_{b_\nu' b_\nu'}-P_{b_\nu' b_\nu'})^k(\msc+(G_{b_{\nu'} c_{\nu'}}-P_{b_{\nu'} c_{\nu'}}))}{\md^k} +\OO(N^{-2}),
    \end{align}
    for $\fb\geq 1$ large enough. Thanks to \Cref{thm:prevthm}, the error in the above expansion can be made smaller then $\OO(N^{-2})$.
    The same as in the \eqref{e:oneterm3}, using \eqref{e:GGG} and \eqref{e:GGG2}, up to an error $\OO(\bE[\Psi_p])$, we can write \eqref{e:oneterm} as a weighted sum of terms with total weights $\OO(1)$, in the form
\begin{align}\label{e:oneterm4}
 \frac{1}{(d-1)^{(k_1-k_3)\ell/2}} (G_{c_\nu c_\nu}^{(b_\nu)}-Q) R'_{\bfi^+},
\end{align}
In particular, \eqref{e:hatRi} is a weighted sum of terms (with total weights $\OO(\poly(\ell))$) in form
\begin{align}\label{e:R+term2}
  \frac{1}{(d-1)^{(k_1-k_3+4)\ell/2}Z_{\cF^+}}\sum_{\bfi^+}\bE\left[I(\cF^+,\cG)\bm1(\cG,\tcG\in \Omega)  (G_{c_\nu c_\nu}^{(b_\nu)}-Q) R'_{\bfi^+}\right],
\end{align}
where $k_3/2\leq 1$.
\end{enumerate}

\begin{proof}[Proof of \Cref{p:iteration}] There are several cases, for \eqref{e:oneterm4} and \eqref{e:oneterm2} depending on the choice when we replace $\wt G_{oo}^{(i)}-Y_\ell(Q)$ 
by using \Cref{i1}, 

\begin{enumerate}
    \item \label{i:k0}For \eqref{e:oneterm4}, we can further replace $Q$ by $Y_\ell(Q)$, and the error is negligible
    \begin{align}\begin{split}\label{e:finaleq}
       &\phantom{{}={}}\frac{1}{(d-1)^{(k_1-k_3+4)\ell/2}Z_{\cF^+}}\sum_{\bfi^+}\bE\left[I(\cF^+,\cG)\bm1(\cG,\tcG\in \Omega)  (G_{c_\nu c_\nu}^{(b_\nu)}-Q)R_{\bfi^+}'\right]\\
       &=\frac{1}{(d-1)^{(k_1-k_3+4)\ell/2}Z_{\cF^+}}\sum_{\bfi^+}\bE\left[I(\cF^+,\cG)\bm1(\cG,\tcG\in \Omega)  (G_{c_\nu c_\nu}^{(b_\nu)}-Y_\ell(Q))R_{\bfi^+}'\right]\\
       &+\OO\left(\frac{1}{(d-1)^{\ell/2}Z_{\cF^+}}\sum_{\bfi^+}\bE\left[\bm1(\cG,\tcG\in \Omega)  |Y_\ell(Q)-Q||R_{\bfi^+}'|\right]\right),
    \end{split}\end{align}
    where we used that $k_1-k_3+4\geq 1$.
   
    For the righthand side of \eqref{e:finaleq}, by the same argument as in the proof of \Cref{p:add_indicator_function}, up to a negligible error $\OO(\bE[\Phi_p])$, we can replace $\bm1(\cG, \wt \cG\in\Omega)$ back to $\bm1(\cG\in\Omega)$.Then \eqref{e:rough_bound} gives that the last error term in \eqref{e:finaleq} is bounded by $\bE[\Phi_p]$.
    For the first term on the righthand side of \eqref{e:finaleq}, if in $R_{\bfi^+}$ the total number of special edges, $Q-Y_\ell(Q)$ and its complex conjugate is exactly $2p-1$, we get \eqref{e:case1} by noticing that $k_1-k_3+4\geq 1$. Otherwise if the total number is at least $2p$, thanks to \eqref{e:rough_bound} 
    \begin{align*}
        |\eqref{e:R+term}|\lesssim (d-1)^{-\ell}\bE[\bm1(\cG\in \Omega)(|Q-Y_\ell(Q)|+\Phi)^{2p}]\lesssim \bE[\Psi_p].
    \end{align*}
 \item \label{i:k1}If $k_1\geq 1$ in \eqref{e:oneterm2}, we can rewrite \eqref{e:oneterm2} as
    \begin{align*}
        \frac{1}{(d-1)^{(k_1+k)\ell/2}} R_\bfi^+=:
        \frac{1}{(d-1)^{(k_1+k)\ell/2}}  (d-1)^{2\ell} (G_{c_\al c_\al}^{(b_\al)}-Q)R_{\bfi^+}'.
    \end{align*} 
    Then we can process as in the previous case \Cref{i:k0}, which leads to \eqref{e:case1}. 
 
\item \label{i:ignore}If \eqref{e:oneterm2} contains two terms in the form $(d-1)^{2\ell}G_{c c'}^{(b b')}$ with $(b,c)\neq (b',c')\in (\cC^+)^\circ$; or one term in the form $(d-1)^{2\ell}G_{c c'}^{(b b')}$, and one Green's function term $(d-1)^{2\ell}(G_{ss'}-P_{ss'})$ with $s,s'\in \cK^+$ but in different connected components of $\cF^+$; or one term in the form $(d-1)^{2\ell}G_{c c'}^{(b b')}$ and $\cV^+\neq \emptyset$, then \Cref{p:small_Ri} implies that
\begin{align*}
     &\phantom{{}={}}\frac{1}{(d-1)^{2\ell}Z_{\cF^+}}\sum_{\bfi^+}\bE\left[I(\cF^+,\cG)\bm1(\cG,\tcG\in \Omega) \frac{1}{(d-1)^{(k_1+k)\ell/2}} R_{\bfi^+}\right]=\OO(\bE[\Psi_p]).
\end{align*}

\item \label{ii:offab} If $k_3=2$ in \eqref{e:oneterm2}, then \eqref{e:oneterm2} contains $(d-1)^{2\ell}G_{c_\beta c_{\beta'}}^{(b_{\beta} b_{\beta'})}$ where the indices $\beta, \beta'$ appears at least twice among $\bm\al, \bm\beta, \bm\gamma,\bm \theta$ in \eqref{e:oneterm2} (otherwise, we are in the case for \eqref{e:oneterm4}). If $k_1\geq 1$, this is covered in \Cref{i:k1}.  In the following we assume that $k_1=0$, $\cV^+=\emptyset$ and \eqref{e:oneterm2} does not contain other terms in the form $(d-1)^{2\ell}G^{bb'}_{cc'}$ with $(b,c)\neq (b',c')\in (\cC^+)^\circ$, or $(d-1)^{2\ell}(G_{ss'}-P_{ss'})$ with $s,s'\in \cK^+$ but in different connected components of $\cF^+$. Otherwise it has been covered in \Cref{i:k1} and \Cref{i:ignore}. 
Thus in \eqref{e:oneterm2},
\begin{align*}
    k_1=0, \quad k_2=0, \quad k_3=2.
\end{align*}
Moreover,  all the factors in \eqref{e:oneterm2} come from \Cref{i1}, \Cref{i2}, \Cref{i4}, $W_j$ in \Cref{i6},  $W_j$ in  \Cref{i8}, and \Cref{i9}. We recall that the indices $\beta, \beta'$ appear at least twice among $\bm\al, \bm\beta, \bm\gamma,\bm \theta$. It follows that $\sfD(\bm\theta)$ contains $(d-1)^{2\ell}(G_{c_\beta c_\beta}^{(b_\beta)}-Q)$ and $(d-1)^{2\ell}(G^{(b_{\beta'})}_{c_{\beta'} c_{\beta'}}-Q)$.
Thus $k=2$ in \eqref{e:oneterm2}, and by the same argument as in \Cref{i:k0}, \eqref{e:R+term} leads to \eqref{e:case1}.

\item Finally we discuss the case that $\wt W_0$ is replaced by $\cE\in \cD_r^S$ as described in \eqref{def:DS} with $r\geq 2$. 
If $\cE$ contains at least one term in the form $(d-1)^{2\ell}(G_{c_\al c_\al}^{(b_\al)}-Q)$, the same argument as in \Cref{i:k0}, \eqref{e:R+term} leads  to \eqref{e:case2}; 
If $\cE$ contains at least one term in the form $(d-1)^{2\ell}G_{c_\al c_\beta}^{(b_\al b_\beta)}$, by the same argument as in \Cref{ii:offab}, \eqref{e:R+term} leads to \eqref{e:case2}.

If $\cE$ contains at least one term in the form $(d-1)^{2\ell}(Q-Y_\ell(Q))$, then $|\cE|\lesssim (d-1)^{2\ell}|Q-Y_\ell(Q)| (d-1)^{2\ell}\varepsilon$. The total number of special edges, $Q-Y_\ell(Q)$ and its complex conjugate is at least $2p$. It follows 
\begin{align*}
    &\phantom{{}={}}\frac{1}{(d-1)^{2\ell}Z_{\cF^+}}\sum_{\bfi^+}\bE\left[I(\cF^+,\cG)\bm1(\cG,\tcG\in \Omega) \frac{1}{(d-1)^{(k_1+k)\ell/2}} R_{\bfi^+}\right]\\
    &\lesssim \bE[(d-1)^{2\ell}\bm1(\cG\in \Omega)\varepsilon (|Q-Y_\ell(Q)|+\Phi)^{2p} ] \lesssim \bE[\Psi_p].
\end{align*}
In the remaining cases $\cE$ contains at least two terms in the form $(d-1)^{2\ell}(Q-\msc)$, this leads to
\begin{align*}
    &\phantom{{}={}}\frac{1}{(d-1)^{2\ell}Z_{\cF^+}}\sum_{\bfi^+}\bE\left[I(\cF^+,\cG)\bm1(\cG,\tcG\in \Omega) \frac{1}{(d-1)^{(k_1+k)\ell/2}} R_{\bfi^+}\right]\\
    &\lesssim \bE[(d-1)^{2\ell}\bm1(\cG\in \Omega)|Q-\msc|^2(|Q-Y_\ell(Q)|+\Phi)^{2p-1} ] \lesssim \bE[\Psi_p].
\end{align*}

\end{enumerate}
The cases above enumerate all possibilities and complete the proof of \Cref{p:iteration}.

\end{proof}

\bibliography{ref}{}
\bibliographystyle{abbrv}

\end{document}